%% file: main.tex
\documentclass[reqno]{amsart}

\usepackage{enumitem}

\usepackage[dvipsnames]{xcolor}
\usepackage%
[colorlinks=true,linkcolor=Maroon,citecolor=OliveGreen,pagebackref]
{hyperref}
\usepackage{mathtools}
\mathtoolsset{showonlyrefs}

\newtheorem{theorem}{Theorem}[section]
\newtheorem{lemma}[theorem]{Lemma}
\newtheorem{proposition}[theorem]{Proposition}
\newtheorem{corollary}[theorem]{Corollary}

\theoremstyle{definition}
\newtheorem{remark}[theorem]{Remark}
\newtheorem{definition}[theorem]{Definition}

\usepackage{amssymb}
\renewcommand{\leq}{\leqslant}
\renewcommand{\le}{\leq}
\renewcommand{\geq}{\geqslant}
\renewcommand{\ge}{\geq}
\newcommand{\eps}{\epsilon}

\renewcommand{\subseteq}{\GenericError{}{For the sake of consistency, subseteq is a banned command. If we decide to use the symbol then we should override the subset command, or similar}{}{}}

\newcommand\opr[1]{\operatorname{#1}}

\def\C{\mathbf{C}}
\def\HS{{\opr{HS}}}
\def\op{{\opr{op}}}
\def\tr{{\opr{tr}}}
\def\im{{\opr{im}}}
\def\per{\opr{per}}
\def\Aut{\opr{Aut}}
\def\Sym{\opr{Sym}}

\def\GL{\opr{GL}}
\def\SL{\opr{SL}}
\def\PSL{\opr{PSL}}
\def\PSU{\opr{PSU}}

\def\PSp{\opr{PSp}}
\def\Sz{\opr{Sz}}

\def\P{\mathbf{P}}
\def\E{\mathbf{E}}

\newcommand\floor[1]{\left\lfloor{#1}\right\rfloor}
\newcommand\pfrac[2]{\left(\frac{#1}{#2}\right)}
\newcommand\qqquad{\qquad\qquad}

\def\frakS{\mathfrak{S}}
\def\ab{\textup{ab}}
\def\calP{\mathcal{P}}
\def\calQ{\mathcal{Q}}
\def\calR{\mathcal{R}}
\def\supp{\opr{supp}}

\def\hatS{\widehat{1_S}}

\newcommand\hatotimes{\mathbin{\widehat{\otimes}}}

\def\coll{\opr{coll}}

\def\cm{\opr{cm}} 
\def\cx{\opr{cx}} 
\def\rank{\cork}
\newcommand{\cork}[0]{\operatorname{rank}} 
\def\psystem{\mathfrak{P}} 

\usepackage{todonotes}

\begin{document}

\title{An asymptotic for the Hall--Paige conjecture}

\author{Sean Eberhard}
\address{Sean Eberhard, Centre for Mathematical Sciences, Wilberforce Road, Cambridge CB3~0WB, UK}
\email{eberhard@maths.cam.ac.uk}

\author{Freddie Manners}
\address{Freddie Manners, UCSD Department of Mathematics, 9500 Gilman Drive \#0112, La Jolla CA 92093, USA}
\email{fmanners@ucsd.edu}

\author{Rudi Mrazovi\'c}
\address{Rudi Mrazovi\'c, University of Zagreb, Faculty of Science, Department of Mathematics, Zagreb, Croatia}
\email{Rudi.Mrazovic@math.hr}

\thanks{SE has received funding from the European Research Council (ERC) under the European Union’s Horizon 2020 research and innovation programme (grant agreement No. 803711); RM is supported in part by the Croatian Science Foundation under the project UIP-2017-05-4129 (MUNHANAP)} 

\begin{abstract}
  Hall and Paige conjectured in 1955 that a finite group $G$ has a complete mapping if and only if its Sylow $2$-subgroups are trivial or noncyclic. This conjecture was proved in 2009 by Wilcox, Evans, and Bray using the classification of finite simple groups and extensive computer algebra. Using a completely different approach motivated by the circle method from analytic number theory, we prove that the number of complete mappings of any group $G$ of order $n$ satisfying the Hall--Paige condition is $(e^{-1/2} + o(1)) \, |G^\ab| \, n!^2/n^n$.
\end{abstract}

\maketitle

\section{Introduction}

A \emph{complete mapping} of a group $G$ is a bijection $\phi: G \to G$ such that $x \mapsto x \phi(x)$ is also bijective. Complete mappings arise naturally in the theory of Latin squares: the Latin square based on the multiplication table of $G$ has an orthogonal mate\footnote{Two Latin squares of the same dimension are called orthogonal mates if all the pairs of entries in corresponding cells are different.} if and only if $G$ has a complete mapping.

For example, if $n = |G|$ is odd then $x \mapsto x^2$ is bijective, so $\phi(x) = x$ is a complete mapping. On the other hand not all groups have complete mappings. Indeed, note that if $G$ is abelian and $\phi: G \to G$ is complete then
\[
  \prod_{x \in G} x = \prod_{x \in G} (x \phi(x)) = \left( \prod_{x \in G} x\right)^2,
\]
so $\prod_{x \in G} x$ must be trivial. Thus for example cyclic groups of even order do not have complete mappings. This observation goes back in some form to Euler~\cite{euler} and his ``thirty-six officers problem'' (1782), and has been rediscovered several times (see~\cite[Section~3.1.1]{evans-book}).

More generally, if $G$ has a complete mapping then $\prod_{x \in G} x$ must be trivial in the abelianization $G^\ab$, i.e., we must have $\prod_{x \in G} x \in G'$, where $G' = [G,G]$ is the commutator subgroup of $G$. We call this condition \emph{the Hall--Paige condition}. Hall and Paige~\cite{hp} proved that this is equivalent to the condition that the Sylow $2$-subgroups of $G$ are trivial or noncyclic, and they conjectured that this condition is also sufficient for the existence of a complete mapping. This conjecture was finally proved in 2009 in breakthrough work of Wilcox, Evans, and Bray~\cite{wilcox, evans}.

\begin{theorem}[The Hall--Paige conjecture, proved in 2009 by Wilcox, Evans, and Bray]%
  \label{hp-conjecture}
  A finite group $G$ has a complete mapping if and only if $G$ satisfies the Hall--Paige condition.
\end{theorem}

Let us very roughly describe the proof of Theorem~\ref{hp-conjecture}.\footnote{For a readable account of the full proof, see~\cite[Part~II]{evans-book}.} Hall and Paige proved that if $G$ has a normal subgroup $N$ such that both $N$ and $G/N$ have complete mappings then $G$ has a complete mapping, and they used this and related arguments to prove the conjecture when $G$ is solvable. On the other hand, the Feit--Thompson Theorem implies that every nonsolvable group satisfies the Hall--Paige condition. Thus a minimal counterexample to the Hall--Paige conjecture would have to be either simple or a group $G$ having a normal subgroup $N$ such that exactly one of $N$ and $G/N$ fails the Hall--Paige condition. Wilcox showed that we may further assume $|N| = 2$ or $|G/N| = 2$, and showed in these circumstances how to construct a complete mapping of $G$ from one of $N$ or $G/N$, thus reducing the Hall--Paige conjecture to the case of simple groups. Complete mappings had already been constructed for several families of simple groups, including the alternating groups by Hall and Paige themselves. Wilcox gave a unified construction for groups of Lie type, leaving only the Tits group and the $26$ sporadic groups. Evans~\cite{evans} combined Wilcox's method with extensive computer algebra to check all remaining cases save only the fourth Janko group $J_4$, and this case was checked by Bray.\footnote{Bray's work remained unpublished for some time, but finally appeared in~\cite{bray}.}

In this paper we give a completely different proof of Theorem~\ref{hp-conjecture}, for sufficiently large groups, based on the foundational principle of probabilistic combinatorics: to show that a thing exists, it suffices to count them. Using nonabelian Fourier analysis and motivated by the circle method from analytic number theory, we prove the following asymptotic for the number of complete mappings of a group satisfying the Hall--Paige condition.

\begin{theorem}%
  \label{main-theorem-hp-version}
  Let $G$ be a finite group of order $n$. If $G$ satisfies the Hall--Paige condition then the number of complete mappings of $G$ is
  \[
    (e^{-1/2} + o(1)) \, |G^\ab| \, n!^2 / n^n.
  \]
\end{theorem}

In particular, we have a new proof that the Hall--Paige conjecture holds for every sufficiently large finite group. The proof is elementary in that we do not require the classification of finite simple groups, but nonconstructive: the only algorithm our method suggests for constructing a complete mapping is to try bijections at random until one works.

We also prove various extensions of this main result, which we now list.

\subsection{Quantitative bounds}%
\label{subsec:quant}

Our methods in proving Theorem~\ref{main-theorem-hp-version} are effective: one can compute an explicit (and not unreasonable) value for how large $G$ must be so that the proof shows that the number of complete mappings is positive.  However, this value is large enough that checking all the remaining smaller cases of the Hall--Paige conjecture directly is not feasible. We can, however, leverage some of the arguments from the sketch above to give a different proof of the full conjecture, one that avoids extensive case-checking.

A careful quantitative analysis allows us to dispatch all but a few finite simple groups. We defer the details to Section~\ref{sec:quant-proof} (see Theorem~\ref{thm:d13}), but the following proposition is representative.

\begin{proposition}%
  \label{prop:baby-quantitative}
  Let $G$ be a finite group of order $|G| > 10^5$ such that all nontrivial complex representations of $G$ have degree at least $21$. Then $G$ has a complete mapping. The same holds if $|G| > 3 \times 10^5$ and all representations have degree at least $13$. \footnote{These conditions imply that $G$ is perfect, so it automatically satisfies the Hall--Paige condition.} 
\end{proposition}

Given Wilcox's reduction to simple groups, and his proof for simple groups of Lie type, this proposition reduces the possible minimal counterexamples to just the Mathieu groups $M_{11}$ and $M_{12}$ (of orders $7920$ and $95040$), which still require another method.\footnote{According to~\cite[Section~4.3]{evans-book}, the fact that $M_{11}$ and $M_{12}$ are not minimal counterexamples goes back to Aschbacher~\cite{aschbacher}; the fact that they are not counterexamples at all was first proved by Dalla Volta and Gavioli~\cite{dalla-volta-gavioli}.} The main value of our results here is therefore an alternative argument for the large sporadic groups.  We also substantially weaken the dependence on the classification of finite simple groups: we need only a classification of finite simple groups $G$ such that either $G \leq \GL_{12}(\C)$ or $|G| \leq 3 \times 10^5$. \footnote{This is still extremely nontrivial. For example, as late as 1972 it was not known whether there was a finite simple group of order 43200 (see Hall~\cite{hall-one-million}).} 

We finally note that statements such as Proposition~\ref{prop:baby-quantitative} do not represent the absolute limit of these methods for small groups.  Specifically, by a more careful choice of parameters, and replacing analytic bounds on various quantities by their actual computable values, the authors are fairly confident that, for example, the Mathieu group $M_{12}$ could also be handled by the same tools (while for instance $M_{11}$ seems just out of reach without new ideas).  However, we will not attempt to defend these rather involved computations in this paper.

\subsection{An asymptotic expansion}%
\label{subsec:asymptotic}

In Theorem~\ref{main-theorem-hp-version} we find the number of complete mappings up to a factor $1+o(1)$.
By elaborating the proof, we can prove the following finer asymptotic.

\def\inv{\opr{inv}(G)}
\begin{theorem}%
  \label{main-theorem-hp-inv-version}
  Let $G$ be a finite group of order $n$. If $G$ satisfies the Hall--Paige condition then the number of complete mappings of $G$ is
  \[
    e^{-1/2} \left(1 + (1/3 + \inv/4)\, n^{-1} + O(n^{-2})\right) \, |G^\ab| \, n!^2 / n^n,
  \]
  where $\inv = |\{x \in G: x^2 = 1\}| / n$ is the proportion of involutions in $G$.
\end{theorem}
The method allows us in principle to extract further terms in the asymptotic series more or less mechanically, though it is prohibitively tedious to do so.

We deduce the following corollary, confirming an observation of Wanless~\cite[Section~6.5]{wanless} (see also McKay--McLeod--Wanless~\cite[Section~3]{mckay--mcleod--wanless}).
\begin{corollary}%
  \label{C2k-corollary}
Among all groups of order $n = 2^k$ with $k$ sufficiently large, the number of complete mappings is uniquely maximized by the elementary abelian group $G = C_2^k$.
\end{corollary}

\subsection{Counting configurations of permutations}

In previous work~\cite{EMM, moreon} we proved the following theorem, proving conjectures of Wanless~\cite[Conjecture~6.9]{wanless} and Vardi~\cite{vardi}.

\begin{theorem}%
  \label{EMM-main-thm}
  Let $G$ be an abelian group of order $n$ and let $f:\{1, \dots, n\} \to G$ be a function such that
  \[
    \sum_{i=1}^n f(i) = \sum_{x \in G} x.
  \]
  Then the number of solutions to $\pi_1 + \pi_2 + \pi_3 = f$ with $\pi_1, \pi_2, \pi_3: \{1, \dots, n\} \to G$ bijections is
  \[
    (\frakS(f) + o(1)) \, n!^3 / n^{n-1}.
  \]
\end{theorem}

Here $\frakS(f) = \exp(-\coll(f) / n^2)$, where $\coll(f)$ is the number of \emph{collisions} in $f$:
\begin{equation}
  \label{eq:collisions}
  \coll(f) = \big|\big\{ (i,j) \colon 1 \le i < j \le n,\ f(i)=f(j) \big\}\big| = \sum_{x \in G} \binom{|f^{-1}(x)|}{2}.
\end{equation}

In this paper we prove the following generalization to all finite groups, which also generalizes Theorem~\ref{main-theorem-hp-version}.

\begin{theorem}%
  \label{main-thm}
  Let $G$ be a group of order $n$ and let $f:\{1, \dots, n\} \to G$ be a function such that
  \begin{equation} \label{f-condition}
    \prod_{i=1}^n f(i) = \prod_{x \in G} x  \pmod{G'}.
  \end{equation}
  Then the number of solutions to $\pi_1 \pi_2 \pi_3 = f$ with $\pi_1, \pi_2, \pi_3: \{1, \dots, n\} \to G$ bijections is
  \[
    (\frakS(f) + o(1))\, |G^\ab|\, n!^3 / n^n.
  \]
\end{theorem}

The case $f \equiv 1$ is equivalent to Theorem~\ref{main-theorem-hp-version}. Indeed, in this case~\eqref{f-condition} is precisely the Hall--Paige condition, so the theorem asserts that if $G$ satisfies the Hall--Paige condition then the number of solutions to $\pi_1 \pi_2 \pi_3 \equiv 1$ is $(e^{-1/2} + o(1))\, |G^\ab|\, n!^3/n^n$. But for every such triple $(\pi_1, \pi_2, \pi_3)$ we have \[
  \pi_1(x) \pi_2(x) = \pi_3(x)^{-1}
\]
for every $x \in \{1,\dots,n\}$, or equivalently
\[
  y\, \pi_2(\pi_1^{-1}(y)) = \pi_3(\pi_1^{-1}(y))^{-1}
\]
for every $y = \pi_1(x) \in \{1,\dots,n\}$.  So, the map $\phi = \pi_2 \circ \pi_1^{-1}$ is a bijection such that
\[
  y\, \phi(y) = \pi_3(\pi_1^{-1}(y))^{-1}
\]
is also a bijection: thus, $\phi$ is a complete mapping. Conversely, given a complete mapping $\phi$ and any bijection $\pi_1$ we can reverse the argument to find a unique triple $(\pi_1,\pi_2,\pi_3)$ satisfying $\pi_1 \pi_2 \pi_3 \equiv 1$; i.e., the correspondence $(\pi_1,\pi_2,\pi_3) \leftrightarrow \phi$ is $n!$-to-$1$. 

\subsection{Heuristic explanation of the asymptotic}%
\label{subsec:heuristic}

The asymptotic appearing in Theorem~\ref{main-theorem-hp-version} deserves a heuristic explanation. The following argument is similar to the one given in~\cite{EMM}, and uses the ``principle of maximum entropy'' from statistical physics: given limited observations about some unknown quantity, the probability distribution which best represents the current state of knowledge is the one with maximum entropy.\footnote{Cf.~Good~\cite{good}.}

Consider a random bijection $\phi: G \to G$, and let $\psi:G \to G$ be the function defined by
\[
  \psi(x) = x \phi(x).
\]
If we incorporate no knowledge about $\psi$ other than that $\psi$ is a function $G \to G$, the principle of maximum entropy would encourage us to think of $\psi$ as a uniformly random function $G \to G$. Thus a zeroth-order approximation to the true probability that $\psi$ is a bijection would be $n!/n^n$. This would lead us to guess that the number of complete mappings is roughly $n!^2 / n^n$.

We know that that cannot be right in general, because for instance if the Hall--Paige condition is not satisfied then the answer must be zero. We can inform our approximation by observing that
\begin{equation}
  \label{eqn:prod-psi-mod-G'-condition}
  \prod_{x \in G} \psi(x)
  \equiv \left( \prod_{x \in G} x \right)^2
  \equiv 1
  \pmod{G'}.
\end{equation}
The collection of functions $\psi$ satisfying this condition is a subgroup $H$ of $G^G$ of order $n^n / |G^\ab|$, and the most entropic distribution for $\psi$ consistent with this information is the uniform distribution on this subgroup $H$. Thus a first-order approximation to the true probability that $\psi$ is a bijection is $|G^\ab| \cdot n!/n^n$ if the Hall--Paige condition is satisfied, and zero otherwise.

Finally we have the most subtle factor in the asymptotic: the factor of $e^{-1/2}$. This factor is related to the number of collisions in $\psi$. Note that if $f$ is uniformly random, or uniformly random over $H$, then
\[
  \E \coll(f) = \binom{n}{2} \frac1n = \frac{n-1}{2}.
\]
By contrast, consider collisions in $\psi$. For any fixed distinct $x, y$ we have
\[
  \P(\psi(x) = \psi(y))
  = \P(\phi(x) \phi(y)^{-1} = x^{-1} y)
  = \frac1{n-1}
\]
since the random variable $\phi(x) \phi(y)^{-1}$ is uniform on $G \setminus \{1\}$; thus
\[
  \E \coll(\psi) = \binom{n}{2} \frac1{n-1} = \frac{n}{2}.
\]
Thus $\psi$ is slightly more prone to collisions than an ordinary random function. The maximum-entropy distribution for $\psi$ consistent with this observation is the \emph{Gibbs distribution} defined by
\begin{equation}
  \label{gibbs}
  \P(\psi = g) = \frac{e^{\beta \coll(g)}}{Z(\beta)} \cdot \frac{1_H(g)}{|H|},
\end{equation}
where $Z(\beta)$ is a normalizing factor called the \emph{partition function}, and the parameter $\beta$ must be chosen so that
\[
  \E \coll(\psi) = (\log Z)'(\beta) = \frac{n}{2}.
\]
By a Poisson heuristic for $\coll(f)$ we have
\[
  Z(\beta) = \E e^{\beta \coll(f)} \approx e^{(e^\beta - 1) \E \coll(f)} = e^{(e^\beta - 1) (n-1)/2},
\]
so
\[
  (\log Z)'(\beta) \approx e^\beta (n-1)/2,
\]
so we should have
\[
  \beta = 1/n + O(1/n^2).
\]
Since $\coll(g) = 0$ when $g$ is a bijection we therefore need to adjust our previous estimate by a factor of $Z(\beta) \approx e^{1/2}$.\ \footnote{It is interesting to compare~\eqref{gibbs} with the conclusion of Theorem~\ref{main-thm}. While the former is a heuristic approximation for the distribution of $x \phi(x)$, the latter is a rigorous assertion that the distribution of $x \phi(x) \phi'(x)$, where $\phi'$ is another random bijection, is approximately the Gibbs distribution with $\beta = -1/n^2$ (concentrated on a coset of $H$).} 

The further correction expressed in Theorem~\ref{main-theorem-hp-inv-version}, and indeed still smaller corrections, can also be derived in this fashion, by counting ``higher-order'' collisions, but doing so even in this heuristic setting rapidly becomes extremely tedious, as the array of possible collision types exhibits combinatorial explosion. The interested reader should refer to Section~\ref{sec:asymptotic-expansion}, where we develop a more systematic approach.

This argument is a special case of a general method for informing counting conjectures: we guess a model for some random variable (such as $\psi$) and, if it is found to be inadequate, the principle of maximum entropy offers a systematic way of updating our guess.  We caution that the difficult part is knowing when we are close enough to the truth to stop.  For the purposes of the argument above, the only reason that no further corrections to the distribution of $\psi$ are made is that none of the ones we could think of changed the answer by more than a factor of $1+o(1)$, but it is very difficult to rule out the possibility that some hypothetical further observation might change the picture by a much larger amount.

\subsection{Layout of the paper}

The paper is organized as follows.  We next (Section~\ref{sec:preliminaries}) collect some standard tools and conventions which will form the foundation of all our arguments.

With these in place, we can give a detailed account of the proof of Theorem~\ref{main-thm} (and thereby of Theorem~\ref{main-theorem-hp-version}) in Section~\ref{sec-outline}.  The key ingredients for this proof are proven in Sections~\ref{sec:major-arcs},~\ref{sec:low-entropy-minor-arcs}, and~\ref{sec:high-entropy-minor-arcs} (split up in the way explained in Section~\ref{sec-outline}).

The remaining sections explain how to recombine these ingredients to prove the refinements discussed above.  In particular Section~\ref{sec:quant-proof} handles the quantitative results discussed in Section~\ref{subsec:quant}, and Section~\ref{sec:asymptotic-expansion} deals with the asymptotic expansion from Section~\ref{subsec:asymptotic}.

\section{Preliminaries}%
\label{sec:preliminaries}

We review here some relevant background.

\subsection{Nonabelian Fourier analysis}

We briefly recall here the fundamentals of nonabelian Fourier analysis. The reader needing a better introduction could refer to Tao~\cite[Chapter~18]{tao}, with whom our notational conventions agree.

Given a finite group $G$, we write $\int$ for averages over $G$, $\rho$ for a representation of $G$ (usually irreducible, always unitary and finite-dimensional), and $\chi$ for the corresponding character $x \mapsto \tr \rho(x)$ of $G$. The Fourier transform of a function $f:G \to \C$ at an irreducible representation $\rho:G\to U(V)$ is defined by
\[
  \widehat{f}(\rho) = \int_G f(x) \rho(x).
\]
Note that $\widehat{f}(\rho)$ defines an operator on $V$. The space of operators on $V$ is denoted $\HS(V)$ and equipped with the Hilbert--Schmidt inner product
\[
  \langle R, S\rangle = \tr(R S^*).
\]
We have the Fourier inversion formula
\[
  f(x) = \sum_\rho \langle \widehat{f}(\rho), \rho(x)\rangle \dim \rho,
\]
and the Parseval (or Plancherel) identity
\[
  \langle f, g\rangle = \sum_\rho \langle \widehat{f}(\rho), \widehat{g}(\rho)\rangle \dim \rho.
\]
The sums here run over the irreducible representations of $G$. Note in particular that if $\widehat{f}(\rho) = 0$ for all $\rho$, then $f = 0$.

Let $\rho:G\to U(V)$ be an irreducible representation of $G$. Let $e_1, \dots, e_d$ be an orthonormal basis of $V$, and let $E_{ij} = e_i \otimes e_j^* \in \HS(V)$. The functions $\langle E_{ij}, \rho(x)\rangle$, as $\rho$ runs over irreducible representations and $i$ and $j$ run over $\{1, \dots, d\}$ with $d = \dim \rho$, form an orthogonal basis for $L^2(G)$. Indeed, the fact that they span is clear from the Fourier inversion formula, and orthogonality follows from Plancherel: by comparison with the Fourier inversion formula the function $f(x) = \langle E_{ij}, \rho(x) \rangle$ must have Fourier transform
\[
  \widehat{f}(\rho') = \begin{cases}
    E_{ij} / \dim \rho & \text{if}~\rho'=\rho, \\
    0                  & \text{else},
  \end{cases}
\]
so
\[
  \int_G \langle E_{ij}, \rho(x)\rangle \overline{\langle E_{i'j'}, \rho'(x)\rangle} = \begin{cases}
    1 / \dim \rho & \text{if}~\rho=\rho', i=i', j=j' \\
    0             & \text{else}.
  \end{cases}
\]

The convolution $f \ast g$ of two functions $f,g: G \to \C$ is the function $G \to \C$ given (under our conventions) by
\[
  (f \ast g)(x) = \int_G f(y) g(y^{-1} x) .
\]
The key feature of the Fourier transform is that it ``diagonalizes'' (as much as possible anyway) the operation of convolution:
\[
  \widehat{f * g}(\rho) = \widehat{f}(\rho)\, \widehat{g}(\rho).
\]
The operation on the right is the usual multiplication of operators in $\HS(V)$.
%

Given representations $\rho_1: G_1 \to U(V_1)$ and $\rho_2: G_2 \to U(V_2)$, the tensor product $\rho_1 \otimes \rho_2$ is the representation $G_1 \times G_2 \to U(V_1 \otimes V_2)$ defined on pure tensors by
\[
  (\rho_1 \otimes \rho_2)(g,h) \cdot (u \otimes v) = (\rho_1(g)u) \otimes (\rho_2(g) v).
\]
It is well known that the irreducible representations of $G_1 \times G_2$ are precisely the tensor products $\rho_1 \otimes \rho_2$ of irreducible representations $\rho_1$, $\rho_2$ of $G_1$, $G_2$, respectively. Two such representations $\rho_1 \otimes \rho_2$ and $\rho_1' \otimes \rho_2'$ are isomorphic if and only if $\rho_1 \cong \rho'_1$ and $\rho_2 \cong \rho'_2$.

In the special case that $G_1 = G_2 = G$, the restriction of $\rho_1 \otimes \rho_2$ to the diagonally embedded copy of $G$ is again a representation of $G$. Conventionally in representation theory this representation is also denoted simply $\rho_1 \otimes \rho_2$, and it is understood from context whether $\rho_1 \otimes \rho_2$ is a representation of $G\times G$ or of $G$. For us, the interplay between these interpretations of $\otimes$ is essential, so we will use $\hatotimes$ to denote the latter. Thus $\rho_1 \otimes \rho_2$ is a representation of $G^2$ and $\rho_1 \hatotimes \rho_2$ is a representation of $G$.

\subsection{Argument projections}%
\label{subsec:argproj}

Given $X \subset \{1, \dots, n\}$, we identify $L^2(G^X)$ with the subspace of $L^2(G^n)$ consisting of functions $f:G^n \to \C$ of $(g_1,\dots,g_n)$ that depend only on variables $g_i$ for $i \in X$. We denote by $Q_X : L^2(G^n) \to L^2(G^X)$ the corresponding orthogonal projection. Explicitly,
\[
  Q_X = \prod_{i \notin X} E_i,
\]
where $E_i$ is the operator that ``integrates out'' the single variable $g_i$: i.e., for $F \in L^2(G^n)$,
\[
  \big(E_i F\big)(g_1,\dots,g_n) = \int_{g \in G} F(g_1,\dots,g_{i-1},g,g_{i+1},\dots,g_n).
\]
These subspaces $L^2(G^X)$ are nested: if $X \subset Y$ then $L^2(G^X) \subset L^2(G^Y)$. We also define inclusion--exclusion-type projections $P_X$ for $X \subset \{1,\dots,n\}$ by
\[
  P_X = \prod_{i \notin X} E_i \prod_{i \in X} (1 - E_i).
\]
This is the projection onto the space
\[
  L^2(G^X) \cap \bigcap_{Y \subsetneq X} L^2(G^Y)^\perp;
\]
informally, the space of functions that depend ``exactly'' on the variables in $X$.  By inclusion--exclusion we have
\begin{equation} \label{PX-ie}
  P_X = \sum_{Y \subset X} (-1)^{|X|-|Y|} Q_Y
\end{equation}
and
\begin{equation} \label{QX-ie}
  Q_X = \sum_{Y \subset X} P_Y.
\end{equation}

We now describe the relationship between the projections $P_X$ and Fourier analysis on $G^n$. The irreducible representations $\rho$ of $G^n$ are precisely the tensor products
\[
  \rho = \rho_1 \otimes \cdots \otimes \rho_n,
\]
where $\rho_1, \dots, \rho_n$ are irreducible representations of $G$. Let
\[
  \supp \rho = \{i \in \{1, \dots, n\} : \rho_i \neq 1\}.
\]

\begin{lemma}%
  \label{lem:P_X-on-fourier}
Let $F \in L^2(G^n)$.
\begin{enumerate}[label=(\roman*)]
  \item
    \[
      \widehat{P_X F}(\rho) = \begin{cases}
        \widehat{F}(\rho) & \textup{if}~\supp \rho = X, \\
        0 & \textup{else}.
      \end{cases}
    \]
  \item
    \[
      P_X F(g) = \sum_{\rho \colon \supp \rho = X} \langle \widehat{F}(\rho), \rho(g) \rangle \dim \rho.
    \]
\end{enumerate}
\end{lemma}
In other words, the projection $P_X$ simply discards all Fourier coefficients except those with support exactly $X$.
\begin{proof}
  For an irreducible representation $\rho_i \colon G \to \HS(V_i)$ we have
  \[
    \int_{x \in G} \rho_i(x) = \begin{cases} 1 &\colon \rho_i = 1 \\ 0 &\colon \rho_i \ne 1; \end{cases}
  \]
  this follows by considering the Fourier transform of the constant function $1$ on $G$. Hence for any $F \in L^2(G^n)$ and any $\rho = \rho_1 \otimes \dots \otimes \rho_n$,
  \[
    \widehat{E_i F}(\rho)
    = \begin{cases} \widehat{F}(\rho) &\colon \rho_i = 1 \\ 0 &\colon \rho_i \ne 1. \end{cases}
  \]
  The first part follows. The second part follows by Fourier inversion.
\end{proof}

In particular we note the interaction between projections $P_X$ and $Q_X$ and convolution.
\begin{corollary}%
  \label{cor:proj-diag}
  If $F_1, F_2 \in L^2(G^n)$ and $X \subset \{1,\dots,n\}$ then
  \[
    P_X(F_1 \ast F_2) = P_X F_1 \ast P_X F_2
  \]
  and
  \[
    Q_X(F_1 \ast F_2) = Q_X F_1 \ast Q_X F_2.
  \]
  Moreover if $Y \subset \{1,\dots,n\}$ and $Y \ne X$ then
  \[
    P_X F_1 \ast P_Y F_2 = 0.
  \]
\end{corollary}
\begin{proof}
  These follow immediately from Lemma~\ref{lem:P_X-on-fourier} and properties of Fourier analysis and convolution.
\end{proof}

\subsection{M\"obius inversion for partitions}%
\label{subsec:mobius}

\def\discrete{0}
\def\trivial{1}

Let $X$ be an $m$-element set. A \emph{partition} of $X$ is a set $\calP$ of nonempty subsets $p \subset X$ (the \emph{parts} or \emph{cells} of $\calP$) such that every element of $X$ is a member of exactly one part of $\calP$. Given partitions $\calP, \calQ$ of $X$, we say that $\calP$ \emph{refines} $\calQ$, and $\calQ$ \emph{coarsens} $\calP$, and we write $\calP \leq \calQ$, if every cell of $\calP$ is contained in a cell of $\calQ$: this makes the set $\Pi_X$ of all partitions of $X$ into a partially ordered set called the \emph{partition lattice}. As usual, given partitions $\calP$ and $\calQ$ we write $\calP \wedge \calQ$ for their meet (i.e., their coarsest common refinement) and $\calP \vee \calQ$ for their join (i.e., their finest common coarsening).  The partition of $X$ into singletons is called the \emph{discrete partition}, denoted $\discrete$, and the partition $\{X\}$ is called the \emph{trivial partition} (or \emph{indiscrete partition}), denoted $\trivial$: these are the minimal and maximal elements of the partition lattice, respectively.

The \emph{incidence algebra of the partition lattice} is the set of functions $\alpha$ assigning to each pair of partitions $(\calP, \calQ)$ with $\calP \leq \calQ$ a scalar $\alpha(\calP, \calQ)$ (in some unital commutative ring). Addition is defined pointwise, and multiplication is defined by \emph{convolution}:
\[
  (\alpha*\beta)(\calP, \calQ) = \sum_{\calR\colon \calP \leq \calR \leq \calQ} \alpha(\calP, \calR) \beta(\calR, \calQ).
\]
The unit element is
\[
  \delta(\calP, \calQ) = \begin{cases}
    1 & \text{if}~\calP=\calQ, \\
    0 & \text{else}.
  \end{cases}
\]
An element $\alpha$ of the incidence algebra is invertible if and only if each diagonal element $\alpha(\calP, \calP)$ is invertible. The inverse of the constant function $1$ is called the \emph{M\"obius function} $\mu$, and is given by the formula
\[
  \mu(\calP, \calQ) =
  (-1)^{|\calP| - |\calQ|}
  \prod_{q \in \calQ} (|\{p \in \calP: p \subset q\}| - 1)!
\]
(see Stanley~\cite[Example~3.10.4]{stanley}). In the special case that $\calP$ is discrete we omit the symbol from the notation: thus
\[
  \mu(\calQ) = \mu(\discrete, \calQ) = (-1)^{m-|\calQ|} \prod_{q \in \calQ} (|q| - 1)!.
\]
Note that although M\"obius inversion is defined most naturally for functions of pairs of partitions, the following two inversion formulae for univariate functions follow:
\begin{align*}
  \alpha(\calP) = \sum_{\calQ \colon \calP \leq \calQ} \beta(\calQ)
   & \iff \beta(\calP) = \sum_{\calQ \colon \calP \leq \calQ} \mu(\calP, \calQ) \alpha(\calQ); \\
  \alpha(\calP) = \sum_{\calQ \colon \calQ \leq \calP} \beta(\calQ)
   & \iff \beta(\calP) = \sum_{\calQ \colon \calQ \leq \calP} \alpha(\calQ) \mu(\calQ, \calP). \\
\end{align*}

Partitions arise in our setting when we consider the set of injective functions $f \colon X \to G$ and expand its indicator function using inclusion--exclusion; i.e., rewriting inequality constraints $f(x) \ne f(x')$ as equality constraints $f(x) = f(x')$.  M\"obius inversion for partitions captures this cleanly: see Lemma~\ref{mobius-inversion-1} below.
We thereby relate incomplete character sums to sums of complete character sums with attached M\"obius function coefficients.

\subsection{Cauchy's residue theorem}

We will use the following consequence of the residue theorem.

\begin{lemma}%
  \label{lem:cauchy-trick}
    Let $f(u) = a_0 + a_1 u + a_2 u^2 + \cdots$ be a function in a complex variable $u$, that converges uniformly on $|u| \le R$, and obeys the estimate $|f(u)| \le A$ for $|u|=R$.  Then for any $|u|<R$ and $k \ge 0$, we have
    \[
        \big| f(u) - a_0 - a_1 u - \cdots - a_k u^k \big| \le A \frac{(|u|/R)^{k+1}}{1 - |u|/R}.
    \]
\end{lemma}
\begin{proof}
Consider the test function
\begin{align*}
    \rho(z) &= \frac1{z-u} - \frac1z - \frac{u}{z^2} - \cdots - \frac{u^k}{z^{k+1}} \\
    &= \frac{(u/z)^{k+1}}{z-u}.
\end{align*}
By the residue theorem,
\[
    \frac1{2 \pi i} \oint_{|z|=R} \rho(z) f(z)\, dz = f(u) - a_0 - a_1 u - \cdots - a_k u^k,
\]
but since
\[
  |\rho(z)| \le \frac{(|u|/R)^{k+1}}{R-|u|}
\]
on $|z|=R$, the claimed bound follows.
\end{proof}

\section{Outline of the proofs}%
\label{sec-outline}

Let $G$ be a group of order $n$, and denote by $S \subset G^n$ the set of all tuples $(x_1, \dots, x_n) \in G^n$ with $x_i \neq x_j$ for $i \neq j$ (equivalently, the set of bijective functions $\{1,\dots,n\}\to G$). Let $f \in G^n$. Our main theorem, Theorem~\ref{main-thm}, asserts that if $f$ obeys~\eqref{f-condition} then
\[
  1_S * 1_S * 1_S(f) = (\frakS(f) + o(1)) \, |G^\ab| \, \pfrac{n!}{n^n}^3.
\]
By Fourier analysis, we have
\begin{equation} \label{fourier-expression}
  1_S*1_S*1_S(f) = \sum_\rho \langle \hatS(\rho)^3, \rho(f) \rangle \dim \rho,
\end{equation}
where the sum runs over all irreducible representations
\[
  \rho = \rho_1 \otimes \cdots \otimes \rho_n
\]
of $G^n$, where each $\rho_i$ is an irreducible representation of $G$. We will divide the summation in~\eqref{fourier-expression} into several parts depending on the multiplicities of the factors $\rho_1, \dots, \rho_n$.

If almost all of the factors $\rho_i$ are isomorphic to some common one-dimensional representation $\rho_0$, then we call $\rho$ a \emph{major arc}. We will see that $\langle \hatS(\rho)^3, \rho(f)\rangle$ is invariant under shifts of the form $\rho \mapsto \rho \hatotimes \psi^n$ for one-dimensional $\psi$, so the contribution from the major arcs is exactly $|G^\ab|$ (the number of one-dimensional representations) times that from the \emph{sparse} representations, i.e., those $\rho$ in which only $m$ factors $\rho_i$ are nontrivial, for some small $m$. We call $\rho$ \emph{$m$-sparse} if exactly $m$ factors $\rho_i$ are nontrivial.

The sparse representations are the topic of Section~\ref{sec:major-arcs}. Note for example that the contribution from the trivial representation is $(n!/n^n)^3$. Other sparse representations contribute a comparable amount to the sum. Using argument projections and M\"obius inversion on the partition lattice to reduce to complete character sums, we will prove that
\begin{equation}
  \label{major-arcs-estimate}
  \sum_{\substack{m\textup{-sparse}~\rho \\ 0 \leq m \leq 2M}}
  \langle \hatS(\rho)^3, \rho(f)\rangle \dim \rho
  = \left(
  \frakS(f) + O\big(1/(M+1)!\big) + O(M^2/n)
  \right) \pfrac{n!}{n^n}^3
\end{equation}
provided $M <c n^{1/2}$ for some absolute constant $c$.
This will be established in Proposition~\ref{major-arcs}.
In particular, the dominant contribution comes from $O(1)$-sparse representations.

All other $\rho$ are called \emph{minor arcs}, and their contribution is bounded using
\[
  |\langle \hatS(\rho)^3, \rho(f) \rangle| \leq \|\hatS(\rho)\|_3^3 \leq \|\hatS(\rho)\|_\op \|\hatS(\rho)\|_\HS^2,
\]
where $\|\cdot\|_3$ is the Schatten 3-norm\footnote{The Schatten $p$-norm $\|\cdot\|_p$ of a linear operator with singular values $(\lambda_i)$ is $\left(\sum_i \lambda_i^p \right)^{1/p}$, and so $\|\cdot\|_2 = \|\cdot\|_\HS$.  Similarly the operator norm is $\|\cdot\|_\op = \max_i \lambda_i$.}. Minor arcs may be further categorized by their entropy: suppose up to permutation of factors we have
\[
  \rho = \rho_1^{a_1} \otimes \cdots \otimes \rho_k^{a_k},
\]
where $\rho_1, \dots, \rho_k$ are distinct irreducible representations of $G$ and $a_1 + \cdots + a_k = n$. The \emph{entropy} of $\rho$ is defined by
\[
  H(\rho) = \sum_{i=1}^k \frac{a_i}{n} \log \frac{n}{a_i}.
\]
Note that if $H(\rho) = o(1)$ then the largest $a_i$ is $(1-o(1))n$. Informally, we say that $\rho$ is a \emph{low-entropy} minor arc if $H(\rho) = o(1)$, and if additionally the factor $\rho_i$ of multiplicity $(1-o(1))n$ is one-dimensional; otherwise $\rho$ is a \emph{high-entropy} minor arc.

Low-entropy minor arcs are the subject of Section~\ref{sec:low-entropy-minor-arcs}. As with the major arcs we may focus on the sparse case: at the cost of a factor of $|G^\ab|$ we may assume that the representation with multiplicity $(1-o(1))n$ is the trivial representation. We attack these representations with the following weapons:
\begin{enumerate}[label=(\roman*)]
  \item a (more or less sharp) estimate for the total $L^2$ mass on sparse representations (dubbed \emph{sparseval}):
        \[
          \sum_{m\textup{-sparse}~\rho} \|\hatS(\rho)\|_\HS^2 \dim \rho
          \leq O(m^{1/4}) e^{O(m^{3/2} / n^{1/2})} \binom{n}{m}^{1/2} \pfrac{n!}{n^n}^2;
        \]
  \item a uniform bound for the operator norm for an $m$-sparse representation: if $m \leq n/2$ then
        \[
          \|\hatS(\rho)\|_\op \leq \binom{n}{m}^{-1/2} \frac{n!}{n^n};
        \]
  \item an ``inverse theorem'' capturing the near-equality case of the above bound:
        \[
          \|\hatS(\rho)\|_\op \leq e^{-c\eps m} \binom{n}{m}^{-1/2} \frac{n!}{n^n}
        \]
        unless more than $(1-\eps)m$ of the nontrivial factors of $\rho$ are equal to a common one-dimensional representation $\rho_0$ of order two.
\end{enumerate}
By combining these we prove
\begin{equation}
  \label{low-entropy-minor-arcs-estimate}
  \sum_{m\textup{-sparse}~\rho} \|\hatS(\rho)\|_\op \|\hatS(\rho)\|_\HS^2 \dim \rho
  \leq O\left(e^{-c \frac{\log(n/m)}{\log n} m} \pfrac{n!}{n^n}^3\right)
\end{equation}
for $m \leq cn / (\log n)^2$.
This will be established in Proposition~\ref{low-entropy-minor}.

Finally in Section~\ref{sec:high-entropy-minor-arcs} we bound the contribution from high-entropy minor arcs. For these we use the still-cruder bound
\[
  |\langle \hatS(\rho)^3, \rho(f) \rangle|
  \leq \|\hatS(\rho)\|_\op \|\hatS(\rho)\|_\HS^2
  \leq \|\hatS(\rho)\|_\HS^3.
\]
For high-entropy minor arcs $\rho$ we prove a bound for $\|\hatS(\rho)\|_\HS$ roughly of the form
\[
  \|\hatS(\rho)\|_\HS^2 \dim \rho \lesssim e^{-H(\rho) n} \pfrac{n!}{n^n}^2.
\]
Thus we deduce a bound of the rough shape
\[
  e^{H(\rho) n} \|\hatS(\rho)\|_\HS^3 \dim \rho \lesssim e^{-H(\rho) n / 2} \pfrac{n!}{n^n}^3.
\]
Note that $e^{H(\rho)n}$ is roughly the size of the orbit of $\rho$ under permutation of factors. Thus, assuming we can bound the number of orbits satisfactorily, we can try to pigeonhole high-entropy minor arcs $\rho$ by the size of $H(\rho)$, and prove
\[
  \sum_{\rho\colon H(\rho) \geq cn}
  \|\hatS(\rho)\|_\HS^3 \dim \rho
  \leq e^{-c'n} \pfrac{n!}{n^n}^3.
\]
In practice, we argue a little differently: we obtain a tidier argument and a stronger bound by using generating function techniques to bound the sum over orbits (and the quantity $H(\rho)$ does not actually appear outside of this outline). In any event, if $R_m$ is the set of all $\rho$ that have some one-dimensional factor of multiplicity at least $n-m$, then we prove
\begin{equation} \label{high-entropy-minor-arcs-estimate}
  \sum_{\rho \in R_m^c} \|\hatS(\rho)\|_\HS^3 \dim \rho \leq e^{-cm} \pfrac{n!}{n^n}^3
\end{equation}
for $m \geq C n^{3/4}$.
This will be established in Proposition~\ref{prop:dense-minor-arcs}.

By combining~\eqref{major-arcs-estimate},~\eqref{low-entropy-minor-arcs-estimate}, and~\eqref{high-entropy-minor-arcs-estimate}, we have
\begin{align*}
  1_S*1_S*1_S(f)
   & = \sum_\rho \langle \hatS(\rho)^3, \rho(f)\rangle \dim \rho \\
   & = \left(\frakS(f) + O(1/M!) + O(M^2/n) \right) \, |G^\ab| \, \pfrac{n!}{n^n}^3 \\
   & \qquad + O\left( e^{-c \frac{\log(n/M)}{\log n} M} \, |G^\ab| \right) \, \pfrac{n!}{n^n}^3 \\
   & \qquad + e^{-c n^{3/4}} \pfrac{n!}{n^n}^3.
\end{align*}
Theorem~\ref{main-thm} follows by taking $M$ to be a sufficiently slowly growing function of $n$ (say a small power of $n$).

We briefly discuss the other results.
In Section~\ref{sec:quant-proof}, our aim is to prove the Hall--Paige conjecture for all groups, not just sufficiently large groups. An argument of Wilcox reduces the problem to simple groups. In particular, we may assume that $G$ has no low-dimensional representations (a weak version of quasirandomness). In this circumstance our minor arc bounds become easier and stronger. In the low-entropy minor arcs, the near-equality case (see weapons~(ii)--(iii) above) is now impossible, and we can prove a stronger version of~\eqref{low-entropy-minor-arcs-estimate}. In the high-entropy minor arcs, we combine this quasirandomness with sparseval (weapon~(i)) to get an alternative to the bounds in Section~\ref{sec:high-entropy-minor-arcs} which is useful in some regimes. We use these stronger bounds in Section~\ref{sec:quant-proof} to prove Proposition~\ref{prop:baby-quantitative}, which proves the Hall--Paige conjecture except for a handful of simple groups.

In Section~\ref{sec:asymptotic-expansion}, for simplicity in the special case $f \equiv 1$, we discuss lower-order terms in Theorem~\ref{main-thm}, in particular Theorem~\ref{main-theorem-hp-inv-version}. Our approach differs only in its treatment of the major arcs. The task is simplified by working exclusively with the case $f \equiv 1$, but simultaneously harder in that we wish to evaluate the $O(M^2/n)$ term in~\eqref{major-arcs-estimate} up to an error of $O_M(1/n^2)$.

\section{Major arcs}%
\label{sec:major-arcs}

In this section we estimate the contribution to~\eqref{fourier-expression} from the major arcs: those $\rho$ with $n-O(n^{1/2})$ factors isomorphic to the same one-dimensional representation $\rho_0$ of $G$. The following lemma shows that this contribution is exactly $|G^\ab|$ (the number of one-dimensional representations) times the contribution from those with $\rho_0$ trivial.

\begin{lemma}%
  \label{lem-onedim-shift}
  Suppose $\psi$ is a one-dimensional representation of $G$, and suppose $\rho = \rho_1 \otimes \cdots \otimes \rho_n$ and $\rho' = \rho'_1 \otimes \cdots \otimes \rho'_n$ are irreducible representations of $G^n$ such that $\rho'_i = \rho_i \hatotimes \psi$ for each $i$. Then
  \[
    \langle \hatS(\rho')^3, \rho'(f)\rangle = \langle \hatS(\rho)^3, \rho(f)\rangle
    \cdot \prod_{g \in G} \psi(g) \prod_{i=1}^n \overline{\psi(f_i)}.
  \]
  In particular, if
  \[
    \prod_{i=1}^n f_i = \prod_{g \in G} g \pmod{G'},
  \]
  then
  \[
    \langle \hatS(\rho')^3, \rho'(f)\rangle = \langle \hatS(\rho)^3, \rho(f)\rangle.
  \]
\end{lemma}

\begin{proof}
  Note that $\rho' = \rho \hatotimes \psi^n$, where
  \[
    \psi^n = \overbrace{\psi \otimes \cdots \otimes \psi}^n
  \]
  is the one-dimensional representation of $G^n$ defined by
  \[
    \psi^n(g_1, \dots, g_n) = \prod_{i=1}^n \psi(g_i).
  \]
  Since $1_S$ is by definition supported on permutations of $G$, we thus have
  \[
    \hatS(\rho') = \left(\prod_{g \in G} \psi(g)\right)\,\hatS(\rho),
  \]
  and
  \[
    \rho'(f) = \left(\prod_{i=1}^n \psi(f_i)\right)\, \rho(f).
  \]
  Note that $\left(\prod_{g \in G} \psi(g)\right)^2 = 1$: indeed,
  \[
    \prod_{g \in G} \psi(g) = \prod_{g \in G} \psi(g^{-1}) = \Bigg(\prod_{g \in G} \psi(g) \Bigg)^{-1}.
  \]
  The lemma follows.
\end{proof}

Call a representation $\rho = \rho_1\otimes \cdots \otimes \rho_n$ of $G$ \emph{$m$-sparse} if exactly $m$ of the factors $\rho_i$ are nontrivial, i.e., if $|\supp \rho| = m$. The goal of this section is to estimate the total contribution from all $m$-sparse $\rho$ for $m \leq cn^{1/2}$, for some constant $c$. Define
\[
  M_{m, f} = \sum_{\rho \colon |\supp \rho| \leq m} \langle \hatS(\rho)^3, \rho(f) \rangle \dim \rho.
\]
Define also
\[
  \frakS_m(f) = \sum_{2k\leq m} \frac1{k!} \left( - \frac{\coll(f)}{n^2}\right)^k,
\]
noting that
\[
  |\frakS(f) - \frakS_m(f)| \leq \frac{1}{(\floor{m/2}+1)!}.
\]
We will prove the following proposition (the abelian case appeared previously, in a weaker form, as~\cite[Theorem~3.1]{moreon}).

\begin{proposition}\label{major-arcs}
For $m < 0.17 n^{1/2}$,
\[
  \left|M_{m, f} - \frakS_m(f) \pfrac{n!}{n^n}^3\right| \leq O(m^2/n) \pfrac{n!}{n^n}^3.
\]
Concretely, if $n > 10^5$ and $m \leq 20$,
\[
  \left| M_{m, f} - \frakS_m(f) \pfrac{n!}{n^n}^3 \right| < 0.32 \pfrac{n!}{n^n}^3.
\]
\end{proposition}

The estimate~\eqref{major-arcs-estimate} follows immediately from this. The remainder of this section is concerned with the proof of this proposition.

To prove this proposition we will actually move away from the Fourier-analytic formalism (though we will return to it for the minor arcs), using arguments projections and purely ``physical-side'' (as opposed to frequency-side) arguments.\footnote{This fact suggests the interesting possibility that the results of this section may hold in greater generality than just that of groups. We intend to return to this consideration in future work.}

\subsection{Applying argument projections}%
\label{subsec:arg-proj-application}

By Lemma~\ref{lem:P_X-on-fourier},
\begin{equation} \label{eq:M_m-physical}
  M_{m, f} = \sum_{|X| \leq m} (P_X 1_S)^{*3}(f).
\end{equation}
Recall that, according to our convention that $L^2(G^X) \subset L^2(G^n)$, $P_X 1_S$ and $Q_X 1_S$ are identified with functions $X \to G$. Let $S_X \subset G^X$ denote the set of injective functions $X \to G$. Then
\begin{equation}
  \label{eq:qx-1s}
  Q_X 1_S = \frac{(n-|X|)!}{n^{n-|X|}} 1_{S_X} .
\end{equation}
Indeed, a function $f \colon X \to G$ can be extended to an injective function $\{1,\dots,n\} \to G$ in $(n-|X|)!$ ways if $f$ is injective and $0$ ways otherwise, and by definition $(Q_X 1_S)(f)$ is the number of these extensions normalized by $n^{-(n-|X|)}$. From this we can derive a formula for $P_X 1_S$. 

Given a partition $\calP$ of $X\subset \{1, \dots, n\}$, we say $f: X \to G$ is \emph{$\calP$-measurable} if $f$ is constant on each cell of $\calP$. Let $c_\calP$ be the indicator of $\calP$-measurability: thus
\[
  c_\calP(f) = \begin{cases}
    1 & \text{if $f$ is constant on each cell of $\calP$}, \\
    0 & \text{else}.
  \end{cases}
\]
By a further slight abuse of notation, we can consider a partition $\calP$ of $X \subset \{1,\dots,n\}$ to be a partition of the full set $\{1,\dots,n\}$, by giving each element of $\{1,\dots,n\} \setminus X$ its own singleton cell.  Moreover, we can think of two partitions $\calP$ and $\calQ$ on different subsets of $\{1,\dots,n\}$ as being identified if they give rise to the same partition of $\{1,\dots,n\}$ in this way: in other words, if they differ just by adding or deleting singletons.  Note that this hypothesis implies $c_\calP = c_\calQ$ as elements of $L^2(G^n)$, so this is compatible with our existing conventions.

We define the \emph{rank} of a partition $\calP$ of $X \subset \{1,\dots,n\}$ by $\rank \calP = |X| - |\calP|$. Again note that this quantity is invariant under adding or deleting singletons. Note that
\[
  \langle c_\calP, 1\rangle = n^{-\rank \calP}
\]
(since there are $n^{|\calP|}$ $\calP$-measurable functions $X \to G$).

The M\"obius inversion theory from Section~\ref{subsec:mobius} allows us to expand $S_X$ in terms of functions $c_\calP$.
\begin{lemma}%
  \label{mobius-inversion-1}
Let $S_X \subset G^X$ be the set of injective functions $X \to G$. Then
\[
  1_{S_X} = \sum_{\calP\in \Pi_X} \mu(\calP) c_\calP.
\]
\end{lemma}
\begin{proof}
Let $d_\calP(f)$ be the indicator that $f$ is $\calP$-measurable and takes a distinct value on each cell of $\calP$. Then
\[
  c_\calP = \sum_{\calQ\colon \calP \leq \calQ} d_\calQ.
\]
Thus by M\"obius inversion we have
\[
  d_\calP = \sum_{\calQ\colon \calP \leq \calQ} \mu(\calP, \calQ) c_\calQ.
\]
The claimed formula is the case $\calP = \discrete$.
\end{proof}

Finally, denote by $\supp \calP$ the union of the nonsingleton cells of $\calP$.

\begin{remark}%
  \label{rem:px-kills-wrong-support}
  Note that $c_\calP$ only depends on variables $g_i$ for $i \in \supp \calP$; i.e., $c_\calP \in L^2(G^{\supp \calP})$.

  In particular, if $X \supsetneq \supp \calP$ is a proper superset then $P_X c_{\calP} = 0$ 
  (as $\im(P_X) \perp L^2(G^Y)$ for any $Y \subsetneq X$).
\end{remark}

\begin{lemma}%
  \label{lem:P_X-on-1_S}
If $X \subset \{1,\dots,n\}$ has size $m$, then
\[
  P_X 1_S = \frac{(n-m)!}{n^{n-m}} \sum_{\calP\colon \supp\calP = X} \mu(\calP) P_X c_\calP.
\]
\end{lemma}
\begin{proof}
By~\eqref{eq:qx-1s} and the previous lemma we have
\begin{align*}
  P_X 1_S
  &= P_X Q_X 1_S \\
  &= \frac{(n-m)!}{n^{n-m}} P_X 1_{S_X} \\
  &= \frac{(n-m)!}{n^{n-m}} \sum_{\calP \in \Pi_X} \mu(\calP) P_X c_\calP,
\end{align*}
  and by Remark~\ref{rem:px-kills-wrong-support} we may restrict the summation to those $\calP$ with $\supp \calP = X$.
\end{proof}

\subsection{Partition systems}%
\label{subsec:psystem}

A \emph{partition triple} on a set $X \subset \{1,\dots,n\}$ is simply a triple $\psystem = (\calP_1, \calP_2, \calP_3)$ of partitions of $X$.  By our usual convention of adding and deleting singletons, this also makes sense if $\calP_i$ are partitions of $\{1,\dots,n\}$ with $\supp \calP_i \subset X$.  The \emph{support} of $\psystem$ is defined to be
\[
  \supp \psystem = \supp \calP_1 \cup \supp \calP_2 \cup \supp \calP_3
\]
or in other words the smallest set $X \subset \{1,\dots,n\}$ such that $\psystem$ can be thought of as a partition triple on $X$ (up to adding or deleting singletons).

A partition triple is called a \emph{partition system} if the partitions all have the same support, i.e., if $\supp \calP_i = \supp \psystem$ for $i=1,2,3$.

Given~\eqref{eq:M_m-physical}, Lemma~\ref{lem:P_X-on-1_S} and Corollary~\ref{cor:proj-diag}, the task of proving Proposition~\ref{major-arcs} reduces to estimating
\begin{equation} \label{P_XcP-triple-conv}
  P_X c_{\calP_1} * P_X c_{\calP_2} * P_X c_{\calP_3} (f) = P_X(c_{\calP_1} * c_{\calP_2} * c_{\calP_3})(f)
\end{equation}
for each subset $X \subset \{1,\dots,n\}$ of size $\le m$ and each partition triple $\psystem=(\calP_1,\calP_2,\calP_3)$ on $X$, and aggregating the results.  By Remark~\ref{rem:px-kills-wrong-support}, the left-hand side is zero unless $\supp \calP_i = X$ for each $i=1,2,3$, so we may restrict attention to partition systems $\psystem$ with $\supp \psystem = X$.

For most partition systems it will suffice to bound~\eqref{P_XcP-triple-conv}. We will do this (in Proposition~\ref{prop:gamma-bound} below) by relating it to the simpler quantity
\[
  c_{\calP_1} * c_{\calP_2} * c_{\calP_3}(f),
\]
which we can estimate much more easily.

\def\trank{\opr{trank}}

\begin{lemma}%
  \label{lem:srank-conv-bound}
Define the \emph{triple rank} of a partition triple $\psystem = (\calP_1, \calP_2, \calP_3)$ by
\[
  \trank(\psystem) = \max_{\sigma \in S_3} \left(
  \rank(\calP_{\sigma(1)}) + \rank(\calP_{\sigma(2)} \vee \calP_{\sigma(3)})
  \right)
\]
  where $S_3$ denotes the symmetric group.  Then
\[
  0 \le c_{\calP_1} \ast c_{\calP_2} \ast c_{\calP_3}(f) \leq n^{-\trank(\psystem)}.
\]
\end{lemma}
\begin{proof}
By definition, $c_{\calP_1} * c_{\calP_2} * c_{\calP_3}(f)$ is the number of solutions to $h_1 h_2 h_3 = f$ over $\calP_i$-measurable $h_i$ ($i=1,2,3$) normalized by $n^{-2n}$. Our claim is that the number of such solutions is bounded by $n^{|\calP_{\sigma(1)}| + |\calP_{\sigma(2)} \vee \calP_{\sigma(3)}|}$ for each permutation $\sigma \in S_3$. There are in total $n^{|\calP_{\sigma(1)}|}$ choices of $\calP_{\sigma(1)}$-measurable $h_{\sigma(1)}$, so it suffices to show, given $f$ and $h_{\sigma(1)}$, that there are at most $n^{\rank(\calP_{\sigma(2)} \vee \calP_{\sigma(3)})}$ choices of $h_{\sigma(2)}$ and $h_{\sigma(3)}$ such that $h_1 h_2 h_3 = f$.

Fix a set $Y \subset \{1, \dots, n\}$ consisting of one element from each cell of $\calP_{\sigma(2)} \vee \calP_{\sigma(3)}$, and fix a choice of $h_{\sigma(2)}(y)$ for each $y \in Y$. There are $n^{|\calP_{\sigma(2)} \vee \calP_{\sigma(3)}|}$ such choices. It suffices to show that each such choice can be extended to at most one valid choice of $h_1,h_2,h_3$.

  Note that, for any $x \in \{1, \dots, n\}$, if one of $h_{\sigma(2)}(x)$ or $h_{\sigma(3)}(x)$ is determined then so is the other, since if  $a_1,a_2,a_3,b \in G$ are a solution to $a_1 a_2 a_3 = b$ and $b$ is fixed then any two of $a_1,a_2,a_3$ uniquely determine the third.\footnote{Note this only uses the Latin square property of group multiplication, rather than the full power of $G$ being a group.}

Let $Y' \supset Y$ be the set of indices $y$ such that one, or equivalently both, of the values $h_{\sigma(2)}(y)$, $h_{\sigma(3)}(y)$ is uniquely determined by our choices so far. It is clear that if $y \in Y'$ and $x,y$ are in the same cell of $\calP_{\sigma(2)}$ then $x \in Y'$ (as $h_{\sigma(2)}(x) = h_{\sigma(2)}(y)$, as $h_{\sigma(2)}$ is $\calP_{\sigma(2)}$-measurable) and similarly for $\calP_{\sigma(3)}$.  Hence $Y'$ is both $\calP_{\sigma(2)}$- and $\calP_{\sigma(3)}$-measurable, and contains a point of each cell of $\calP_{\sigma(2)} \vee \calP_{\sigma(3)}$, so $Y' = \{1, \dots, n\}$ as required.
\end{proof}

\begin{remark}%
  \label{remark:srank-lower-bound}
Note that $\rank(\calP) \geq |\supp\calP|/2$, with equality if and only if $\calP$ is a \emph{pairing}: a partition of a set $X$ of even order into $|X|/2$ pairs. Thus $\trank(\psystem) \geq |\supp\psystem|$, with equality if and only if $\psystem = (\calP, \calP, \calP)$ for some pairing $\calP$.
\end{remark}

\def\lrank{\opr{lrank}}

Next we introduce a further notion of rank of a partition triple $\psystem$ which is weaker than triple rank $\trank(\psystem)$ defined above, but which is occasionally more convenient.

\begin{lemma}%
  \label{lem:lrank-srank-inequality}
  For a partition triple $\psystem = (\calP_1, \calP_2, \calP_3)$, define the \emph{lower rank} by
\[
\lrank(\psystem) = \frac12 \big( \rank(\calP_1) + \rank(\calP_2) + \rank(\calP_3) + \rank(\calP_1 \vee \calP_2 \vee \calP_3) \big).
\]
Then
\[
  \trank(\psystem) \ge \lrank(\psystem).
\]
\end{lemma}
\begin{proof}
It is immediate from the definition of $\trank(\psystem)$ that
\[
    \trank(\psystem) \ge \frac12 \big(\rank(\calP_1) + \rank(\calP_2 \vee \calP_3) \big)
    + \frac12 \big(\rank(\calP_2) + \rank(\calP_1 \vee \calP_3) \big).
\]
The result follows from this and the submodularity property of rank,
\[
  \rank(\calP \vee \calQ) + \rank(\calP \wedge \calQ) \le \rank(\calP) + \rank(\calQ),
\]
applied to $\calP_2 \vee \calP_3$ and $\calP_1 \vee \calP_3$, and the fact that $\calP_3 \le (\calP_2 \vee \calP_3) \wedge (\calP_1 \vee \calP_3)$.\footnote{To see submodularity, consider a set of ``marriages'' that produces $\calQ$ from $\calP \wedge \calQ$. Applying these to $\calP$ produces $\calP \vee \calQ$. Thus $\rank(\calP \vee \calQ) - \rank(\calP) \leq \rank(\calQ) - \rank(\calP \wedge \calQ)$.}
\end{proof}

Finally, the \emph{complexity} of a partition system $\psystem$ is defined by
\[
  \cx{\psystem} = \trank(\psystem) - |\supp\psystem|.
\]
Note that by Remark~\ref{remark:srank-lower-bound}, $\cx{\psystem} \geq 0$ with equality if and only if $\psystem = (\calP, \calP, \calP)$ for some pairing $\calP$.

We call $\psystem = (\calP_1, \calP_2, \calP_3)$ \emph{connected} if $\calP_1 \vee \calP_2 \vee \calP_3$ is the indiscrete partition $\{\supp \psystem\}$.
In general, the \emph{connected components} of $\psystem$ are its restrictions to each cell of $\calP_1 \vee \calP_2 \vee \calP_3$.

\subsection{The quantities \texorpdfstring{$\gamma$}{gamma} and \texorpdfstring{$\gamma_0$}{gamma0}}

For any $f \colon \{1,\dots,n\} \to G$ and any partition triple $\psystem$, define normalized quantities
\[
  \gamma_0(\psystem, f) = n^{\trank(\psystem)} c_{\calP_1} * c_{\calP_2} * c_{\calP_3}(f)
\]
and
\begin{align*}
  \gamma(\psystem, f) &= n^{\trank(\psystem)} P_X(c_{\calP_1} * c_{\calP_2} * c_{\calP_3})(f) \\
  &= n^{\trank(\psystem)} \big(P_X(c_{\calP_1}) * P_X(c_{\calP_2}) * P_X(c_{\calP_3})\big)(f)
\end{align*}
where $X = \supp \psystem$.
As observed above (using Remark~\ref{rem:px-kills-wrong-support}) $\gamma(\psystem,f) = 0$ unless $\psystem$ is a partition system.
In this notation, Lemma~\ref{lem:srank-conv-bound} asserts that
\[
  \gamma_0(\psystem, f) \in [0, 1].
\]
In the rest of this subsection we show that $\gamma(\psystem, f) \approx \gamma_0(\psystem, f)$ for pairings and in general $\gamma(\psystem, f)$ is not too large.

\begin{lemma}%
  \label{lem:gamma-mult}
  Suppose $\psystem_1 = (\calP_1, \calP_2, \calP_3)$ and $\psystem_2 = (\calP'_1, \calP'_2, \calP'_3)$ are two partition triples such that $\supp \psystem_1$ and $\supp \psystem_2$ are disjoint, and we define
  \[
    \psystem = (\calP_1 \vee \calP'_1, \calP_2 \vee \calP'_2, \calP_3 \vee \calP'_3)
  \]
  to be their union in a natural sense.  Then
  \[
    \gamma_0(\psystem, f) = \gamma_0(\psystem_1, f) \gamma_0(\psystem_2, f)
  \]
  and
  \[
    \gamma(\psystem, f) = \gamma(\psystem_1, f) \gamma(\psystem_2, f).
  \]
\end{lemma}
In other words, $\gamma$ and $\gamma_0$ are multiplicative over connected components.
\begin{proof}
  It is clear that
  \[
    c_{\calP_i \vee \calP'_i}(f) = c_{\calP_i}(f) \cdot c_{\calP'_i}(f)
  \]
  for each $i \in \{1, 2, 3\}$, so
  \[
    c_{\calP_1 \vee \calP'_1} * c_{\calP_2 \vee \calP'_2} * c_{\calP_3 \vee \calP'_3}(f) = c_{\calP_1} * c_{\calP_2} * c_{\calP_3}(f) \cdot c_{\calP'_1} * c_{\calP'_2} * c_{\calP'_3}(f).
  \]
  The claim for $\gamma_0$ follows.

  Let $X = \supp \psystem$, $X_1 = \supp \psystem_1$, $X_2 = \supp \psystem_2$. Then $X$ is the disjoint union of $X_1$ and $X_2$, and for any function $F\colon G^X \to \C$ that factors as $F = F_1 \cdot F_2$ for $F_1\colon G^{X_1} \to \C$ and $F_2\colon G^{X_2} \to \C$ we have\footnote{In other words, $P_X = P_{X_1} \otimes P_{X_2}$.}
  \[
    P_X F = P_{X_1} F_1 \cdot P_{X_2} F_2.
  \]
  Thus
  \[
    P_X c_{\calP_1 \vee \calP'_1} * c_{\calP_2 \vee \calP'_2} * c_{\calP_3 \vee \calP'_3}(f) = P_{X_1}(c_{\calP_1} * c_{\calP_2} * c_{\calP_3})(f|_{X_1}) P_{X_2}(c_{\calP'_1} * c_{\calP'_2} * c_{\calP'_3})(f|_{X_2}).
  \]
  This proves the claim for $\gamma$.
\end{proof}

\begin{lemma}%
\label{lem:gamma-pairing}
Let $\psystem = (\calP, \calP, \calP)$ for some pairing $\calP$ with $|\supp \calP| = 2k$. Then
\[
    \gamma(\psystem, f) = (1 - 1/n)^\ell (-1/n)^{k-\ell},
\]
where $\ell$ is the number of cells of $\calP$ on which $f$ is constant.
In particular $|\gamma(\psystem, f)| \leq 1$ and
\[
  |\gamma(\psystem, f) - c_\calP(f)| \leq k/n.
\]
\end{lemma}
\begin{proof}
  Let $\calP = \{p_1, \dots, p_k\}$. In general, by Lemma~\ref{lem:gamma-mult}, $\gamma$ is multiplicative over connected components; in this case, this means
\[
  \gamma(\psystem, f) = \prod_{i=1}^k \gamma\big((\{p_i\},\{p_i\},\{p_i\}), f|_{p_i}\big).
\]
Hence it suffices to check the case $k=1$. Suppose $\calP = \{X\}$, where $|X|=2$. Then
\begin{align*}
  \gamma(\psystem, f)
  &= n^2 P_X(c_\calP^{*3}(f)) \\
  &= n^2 (c_\calP^{*3}(f) - n^{-3}) \\
  &= c_\calP(f) - 1/n.
\end{align*}
This proves the lemma.
\end{proof}

In the next proposition we will use the following notation. Given two partitions $\calQ \le \calR$, we write $\rank(\calR / \calQ)$ to denote $\rank(\calR) - \rank(\calQ)$.  This can be thought of as the number of marriages that have to be applied to the finer partition $\calQ$ to obtain the coarser partition $\calR$. Additionally, given a partition $\calP$ on a set $X$ and a subset $Y \subset X$, define
\[
  \calP \cap Y = \{ p \cap Y : p \in \calP, p \cap Y \neq \emptyset\},
\]
a partition of $Y$, and given a partition triple $\psystem = (\calP_1, \calP_2, \calP_3)$ on $X$ define
\[
  \psystem_Y = (\calP_1 \cap Y, \calP_2 \cap Y, \calP_3 \cap Y),
\]
which is a partition triple on $Y$ (but the restriction of a partition system is not generally a partition system).

\begin{proposition}%
  \label{prop:gamma-bound}
    Let $\psystem$ be a partition system. Then
    $|\gamma(\psystem, f)| \leq 2^{m'}$, where $m'$ is the number of points of $\supp \psystem$ contained in cells of $\calP_1 \vee \calP_2 \vee \calP_3$ of size at least 3.
\end{proposition}
\begin{proof}
    Let $m = |X|$. First we prove $|\gamma(\psystem, f)| \leq 2^m$.
    From~\eqref{PX-ie} and Corollary~\ref{cor:proj-diag} we have
    \begin{align*}
      P_X(c_{\calP_1} * c_{\calP_2} * c_{\calP_3})(f)
      &= \sum_{Y \subset X} (-1)^{|X|-|Y|} Q_Y (c_{\calP_1} * c_{\calP_2} * c_{\calP_3})(f) \\
      &= \sum_{Y \subset X} (-1)^{|X|-|Y|} (Q_Y c_{\calP_1} * Q_Y c_{\calP_2} * Q_Y c_{\calP_3})(f).
    \end{align*}
    Note that
    \[
      Q_Y c_\calQ = n^{-\rank(\calQ / \calQ \cap Y)} c_{\calQ \cap Y}:
    \]
    indeed, $n^{-\rank(\calQ / \calQ \cap Y)}$ represents the probability that a random function 
    \[
      f \colon \{1,\dots,n\} \to G
    \]
    is $\calQ$-measurable, conditioned on the weaker assumption that $f|_Y \colon Y \to G$ is ${(\calQ \cap Y)}$-measurable.
    Thus, normalizing,
    \begin{equation} \label{eq:gammaprime-sum}
      \gamma(\psystem, f) = \sum_{Y \subset X} (-1)^{|X|-|Y|} \gamma_0(\psystem_Y, f) n^{-t(\psystem, Y)}
    \end{equation}
    where
    \[
      t(\psystem, Y) = \trank(\psystem_Y) - \trank(\psystem) + \sum_{i=1}^3 \rank(\calP_i / \calP_i \cap Y).
    \]
    Since $\gamma_0(\psystem_Y, f) \in [0, 1]$, it suffices to prove that
    \[
        t(\psystem, Y) \geq 0
    \]
    for every $Y \subset X$.
    Assume without loss of generality that $\trank(\psystem) = \rank(\calP_1 \vee \calP_2) + \rank(\calP_3)$.
    Then $\trank(\psystem_Y) \geq \rank((\calP_1 \cap Y) \vee (\calP_2 \cap Y)) + \rank(\calP_3 \cap Y)$, so
    \[
        t(\psystem, Y) \geq - \rank((\calP_1 \vee \calP_2) / (\calP_1 \cap Y) \vee (\calP_2 \cap Y)) + \sum_{i=1}^2 \rank(\calP_i / \calP_i \cap Y).
    \]
    Hence it suffices to prove
    \[
        \rank(\calP_1 \vee \calP_2 / (\calP_1 \cap Y) \vee (\calP_2 \cap Y)) \leq \rank(\calP_1 / \calP_1 \cap Y) + \rank(\calP_2 / \calP_2 \cap Y).
    \]
    More generally, if $\calP_i' \leq \calP_i$ for $i = 1, 2$ then
    \[
        \rank(\calP_1 \vee \calP_2 / \calP'_1 \vee \calP'_2) \leq \rank(\calP_1 / \calP'_1) + \rank(\calP_2 / \calP'_2).
    \]
    To see this, consider a set of marriages that produces $\calP_1$ from $\calP'_1$ and a set of marriages that produces $\calP_2$ from $\calP'_2$, and apply both to $\calP_1' \vee \calP'_2$: the result is $\calP_1 \vee \calP_2$.
    This proves the upper bound $|\gamma(\psystem, f)| \leq 2^m$.
    
    Now we deduce $|\gamma(\psystem, f)| \leq 2^{m'}$.
    As in the previous lemma we may assume $\psystem$ is connected.
    If $m = 2$ then $m'=0$ and we are done by the previous lemma.
    If $m > 2$ then $m'=m$ and we are done by argument above.
\end{proof}

\subsection{The \texorpdfstring{$M_{m,f}(z)$}{Mmf(z)} series}

For a partition triple $\psystem = (\calP_1, \calP_2, \calP_3)$ we use the shorthand
\[
  \mu(\psystem) = \mu(\calP_1)\, \mu(\calP_2)\, \mu(\calP_3).
\]
From~\eqref{eq:M_m-physical} and Lemma~\ref{lem:P_X-on-1_S} we have\footnote{Here, and elsewhere, $(n)_m$ denotes the falling factorial $n(n-1)\dots(n-m+1)$.}
\begin{equation}\label{eq:M_m-psystems}
  M_{m, f} = \pfrac{n!}{n^n}^3 \sum_{|\supp \psystem| \leq m} \pfrac{n^{|\supp \psystem|}}{(n)_{|\supp \psystem|}}^3 \mu(\psystem) \gamma(\psystem, f) n^{-\trank(\psystem)}
\end{equation}
where the sum is over all partition systems on $\{1,\dots,n\}$.
For $z\in\C$ define
\[
  M_{m, f}(z) = \pfrac{n!}{n^n}^3 \sum_{|\supp \psystem| \leq m} \pfrac{n^{|\supp \psystem|}}{(n)_{|\supp \psystem|}}^3 \mu(\psystem) \gamma(\psystem, f) n^{-|\supp\psystem|} z^{\cx \psystem}.
\]
By design, $M_{m, f} = M_{m,f}(1/n)$. We can now summarize the rest of the proof of Proposition~\ref{major-arcs}. As we have seen (Remark~\ref{remark:srank-lower-bound}), $\trank(\psystem) \geq |\supp \psystem|$ for every partition system $\psystem$, so $M_{m,f}(z)$ is a polynomial. In this subsection we will show that $|M_{m, f}(z)| = O((n!/n^n)^3)$ for $|z| \leq c/m^2$. From this it follows by elementary complex analysis that
\[
  M_{m, f} = M_{m, f}(1/n) = M_{m, f}(0) + O(m^2/n) (n!/n^n)^3.
\]
Note that $M_{m, f}(0)$ counts the contribution to $M_{m, f}$ from partition systems $\psystem$ with $\cx \psystem = 0$, or equivalently $\trank(\psystem) = |\supp\psystem|$: as we have noted (Remark~\ref{remark:srank-lower-bound}), these are precisely the partition systems of the form $\psystem = (\calP, \calP, \calP)$ for some pairing $\calP$. In the next subsection, we will show that
\[
  M_{m, f}(0) \approx \frakS_m(f) \pfrac{n!}{n^n}^3,
\]
and this will complete the proof of Proposition~\ref{major-arcs}.

In order to bound $M_{m,f}(z)$ we first need to bound some generating functions related to ``associated Stirling numbers of the second kind'': these numbers count partitions of a set into nonsingleton subsets. Let
\[
  \Pi'_X = \{\calP \in \Pi_X : \supp \calP = X\}.
\]
Let $\Pi'_m = \Pi'_{\{1, \dots, m\}}$. For $m\geq 1$ define
\[
  \alpha_m(t) = \sum_{\calP \in \Pi'_m} |\mu(\calP)|\, t^{\rank(\calP)}.
\]
%

\begin{lemma}%
  \label{lem:stat-phys}
If $t \in (0,1/m)$, we have
\[
  \alpha_m(t) \leq t^{m/2} \frac{m! e^{m/2}}{m^{m/2}} \exp(\phi(mt)\, m)
\]
where $\phi$ is the monotonic function $(0,1) \to (0,\infty)$ given by
\[
  \phi(\theta) = \big({-}\log(1 - \theta^{1/2}) - \theta^{1/2}\big) / \theta - 1/2.
\]
\end{lemma}
\begin{remark}
Note $\phi(\theta) \rightarrow 0$ as $\theta \rightarrow 0$, and the remaining expression is comparable (for simplicity when $m$ is even) to the contribution
\[
  t^{m/2} \frac{m!}{2^{m/2} (m/2)!}
\]
from the minimal-rank partitions $\calP\in\Pi'_m$, the pairings.\footnote{Thinking of $\alpha_m(t)$ as a Gibbs distribution on $\Pi'_m$ where the energy is $\rank \calP$, this says that the range $t < 1/m$ (where $t$ is a proxy for the temperature) is in ``solid state'', in the sense that most of the probability mass is concentrated in the lowest energy states.}
\end{remark}
\begin{proof}
We may identify $\alpha_m(t)/m!$ as the coefficient of $x^m$ in $\exp(a(x,t))$, where 
\[
  a(x, t) = \sum_{k=2}^\infty x^k t^{k-1}/k .
\]
For $t, |x| < 1$ we have
\[
  a(x, t) = -x - \frac{1}{t} \log (1 - xt).
\]
Since the coefficients of $a(x,t)$ and hence of $\exp(a(x,t))$ are nonnegative, for any $x\in(0, 1)$ we have
\[
  \alpha_m(t) \leq m!\, x^{-m} \exp(a(x,t)).
\]
The optimal choice of $x$ has different behavior either side of the ``phase transition'' at $t=1/m$. On the side $t<1/m$ of interest, we set $x = (m/t)^{1/2}$. Then
\[
  a(x, t) = -(m/t)^{1/2} - \frac1t \log\left(1 - (mt)^{1/2}\right) = (1/2 + \phi(m t))\, m.
\]
Substituting this into the expression above gives the claimed bound.
\end{proof}

We prove one more technical lemma in the same spirit.
For $m\geq 0$ define\footnote{Here and subsequently, $[E]$ denotes the value $1$ if the statement $E$ is true and $0$ if it is false.}
\[
  \beta_m(t) = \sum_{\calP \in \Pi'_m} t^{\rank(\calP) - 2m} \prod_{p \in \calP} 2^{|p| [|p| > 2]} \alpha_{|p|}(t)^3.
\]
In the same way that $\alpha_m(t)$ is related to counting partitions $\calP \in \Pi'_m$ with given rank, the generating function $\beta_m(t)$ is related to a weighted count of configurations 
\[
  \big\{ (\calP, \calP_1,\calP_2,\calP_3) \in (\Pi'_m)^4 \colon \calP_i \le \calP \text{ for } i=1,2,3\big\}
\]
where the total rank of $\calP$, $\calP_1$, $\calP_2$ and $\calP_3$ is specified.  This counting problem is rather convoluted, but arises naturally in the proof of Proposition~\ref{prop:major-arcs-estimate} below.

We need the following bound.
\begin{lemma}%
  \label{lem:beta-bound}
If $t \leq 0.18 / m$ then
\[
  \sum_{k=0}^m \beta_k(t) / k! = O(1).
\]
Moreover, quantitatively, if $m \leq 20$ and $t \leq 0.01$ then
\[
  \sum_{k=0}^m \beta_k(t) / k! < e.
\]
\end{lemma}

\begin{proof}
  As formal power series,
  \[
    \sum_{k=0}^\infty \beta_k(t) \frac{x^k}{k!} = \exp\left(\sum_{r=2}^\infty 2^{r[r>2]} \alpha_r(t)^3 t^{-r-1} \frac{x^r}{r!}\right).
  \]
  Therefore, for real $x > 0$ we have
  \[
    \sum_{k=0}^m \beta_k(t) \frac{x^k}{k!}
    \leq \exp\left(
      \sum_{r=2}^m 2^{r[r>2]} \alpha_r(t)^3 t^{-r-1} \frac{x^r}{r!}.
    \right),
  \]
  and in particular by setting $x=1$ we have
  \begin{equation}
    \label{eq:beta-sum-bound}
    \sum_{k=0}^m \beta_k(t) / k!
    \leq \exp\left(
      \sum_{r=2}^m 2^{r[r>2]} \alpha_r(t)^3 t^{-r-1} / r!
    \right).
  \end{equation}
  By the previous lemma and the bound $r! \leq e r^{1/2} (r/e)^r$, the latter series is bounded termwise by
  \[
    O(1) \sum_{r=2}^m r 2^r t^{r/2-1} (r/e)^{r/2} \exp(3 \phi(rt) r).
  \]
  Since $t \leq 0.18 / m \leq 0.18 / r$ and $\phi(rt) \leq \phi(0.18)$ (as $r \le m$ and $\phi$ is monotonic) this series is bounded by
  \[
    O(1) \sum_{r=2}^\infty r^2 2^r (0.18/e)^{r/2} \exp(3 \phi(0.18) r),
  \]
  and it is readily verified that this is a convergent sum.\footnote{Crucially, $\log 4 - 1 + \log \theta + 6 \phi(\theta)$ is negative for $\theta = 0.18$. This function has a root around $\theta \approx 0.186$, so the lemma would not hold for $t = 0.19/m$, for instance.}
  
    For the quantitative part of the lemma, we use the bound~\eqref{eq:beta-sum-bound} and explicit values of $\alpha_m(t)$ for $m \leq 20$.
    For any given $m$, by definition $\alpha_m(t)$ is an integer polynomial of degree at most $m-1$ in $t$, and moreover its coefficients may be computed efficiently: e.g., by the argument in Lemma~\ref{lem:stat-phys}, $\alpha_m(t)/m!$ is the coefficient of $x^m$ in 
    \[
      \sum_{\ell = 0}^{m/2} \frac1{\ell!} \left( \sum_{k=2}^m x^k t^{k-1}/k \right)^{\ell}. 
    \]
    Using a computer we find, for $t = 0.01$,
    \[
        \sum_{r=2}^{20} 2^{r[r>2]} \alpha_r(t)^3 t^{-r-1} / r! \approx 0.981 < 1,
    \]
    as claimed.
\end{proof}

We can now bound $M_{m, f}(z)$.

\begin{proposition}%
  \label{prop:major-arcs-estimate}
  For $|z|^{1/2} \le 0.18 / m$, we have
  \[
    |M_{m,f}(z)| \leq O(1) \pfrac{n^m}{(n)_m}^2 \pfrac{n!}{n^n}^3.
  \]
  Quantitatively, for $m \leq 20$ and $|z|^{1/2} \leq 0.01$ we have
  \[
    |M_{m,f}(z)| \leq e \pfrac{n^m}{(n)_m}^2 \pfrac{n!}{n^n}^3.
  \]
\end{proposition}
\begin{proof}
By the definition of $M_{m,f}(z)$, the triangle inequality, and Proposition~\ref{prop:gamma-bound}, the quantity $(n^n/n!)^3 |M_{m,f}(z)|$
is bounded by
\[
  \sum_{|X| \le m} \left(\frac{n^{|X|}}{(n)_{|X|}}\right)^3 n^{-|X|} \sum_{\supp \psystem = X}  2^{m'(\psystem)} |\mu(\psystem)|\ |z|^{\cx \psystem},
\]
where $m'(\psystem)$ is the number of points of $\supp \psystem$ contained in connected components of size at least 3.
Since the sum over $\psystem$ now depends only on $|X|$, not $X$, we may rewrite this as
\[
  \sum_{k=0}^m \binom{n}{k} \left(\frac{n^k}{(n)_{k}}\right)^3 n^{-k} \sum_{\supp \psystem = \{1,\dots,k\}} 2^{m'(\psystem)} |\mu(\psystem)|\ |z|^{\trank(\psystem) - k}.
\]
When $|z| \le 1$ we may apply the bound $\trank(\psystem) \ge \lrank(\psystem)$ and rearrange again to bound this above by
\[
  \sum_{k=0}^m \left(\frac{n^k}{(n)_{k}}\right)^2 \frac{1}{k!} \sum_{\supp \psystem = \{1, \dots, k\}} 2^{m'(\psystem)}|\mu(\psystem)|\ |z|^{\lrank(\psystem) - k}.
\]
  The expression $\big(n^k/(n)_k)^2$ is largest when $k=m$, so we now apply this bound and pull out this factor. By separating the sum based on the partition $\calQ = \calP_1 \vee \calP_2 \vee \calP_3$, we may rewrite the remaining expression as
\[
  \sum_{k=0}^m \frac{1}{k!} \sum_{\supp \calQ = \{1,\dots,k\}} |z|^{\rank(\calQ) / 2 - k} \sum_{\substack{\psystem = (\calP_1, \calP_2, \calP_3) \\ \supp \calP_i= \{1,\dots,k\} \\ \calP_1 \vee \calP_2 \vee \calP_3 = \calQ}} 2^{m'(\psystem)} \prod_{i=1}^3 |\mu(\calP_i)|\ |z|^{\rank(\calP_i) / 2}.
\]
Replacing the condition $\calP_1 \vee \calP_2 \vee \calP_3 = \calQ$ with the weaker condition $\calP_1, \calP_2, \calP_3 \le \calQ$ yields another upper bound, which rearranges to
\[
  \sum_{k=0}^m \frac{1}{k!} \sum_{\supp \calQ = \{1,\dots,k\}} |z|^{\rank(\calQ) / 2 - k} \prod_{q \in \calQ} 2^{|q| [|q|>2]}\left(\sum_{\supp \calP = q} |\mu(\calP)|\ |z|^{\rank(\calP) / 2}\right)^3,
\]
since choosing $\calP \leq \calQ$ with full support is equivalent to choosing a partition $\calP'$ of $q$ of full support independently for each cell $q$ of $\calQ$, and the terms $|\mu(\calP)|$ and $|z|^{\rank(\calP) / 2}$ are multiplicative across the cells $q$.
Now note that this expression is simply
\[
    \sum_{k=0}^m \frac1{k!} \beta_k(|z|^{1/2}).
\]
Hence
\[
  |M_{m,f}(z)| \leq \left(\frac{n!}{n^n}\right)^3 \left(\frac{n^m}{(n)_m} \right)^2 \sum_{k=0}^m \frac{1}{k!} \beta_k(|z|^{1/2}),
\]
and the proposition thus follows from the previous lemma.
\end{proof}

Finally, we apply Lemma~\ref{lem:cauchy-trick}. Specifically, in this case,
\[
  \left|M_{m, f}(u) - M_{m, f}(0)\right|
  = \left|\frac1{2\pi i} \oint_{|z|=R} \frac{M_{m, f}(z) u}{(z-u)z} \, dz\right|
  \leq \max_{|z|=R} |M_{m,f}(z)| \cdot \frac{|u|/R}{1-|u|/R}
\]
for $|u| < R$. Taking $u=1/n$ and assuming $1/n < R \leq (0.18/m)^2$,
\begin{equation} \label{eq:cauchy-Mfm0-estimate}
  |M_{m,f} - M_{m,f}(0)|
  \leq O(1) \pfrac{n^m}{(n)_m}^2 \pfrac{n!}{n^n}^3 \cdot \frac{n^{-1}/R}{1 - n^{-1}/R},
\end{equation}
or, assuming $1/n < R \leq 10^{-4}$ and $m \leq 20$,
\begin{equation} \label{eq:cauchy-Mfm0-estimate-2}
  |M_{m,f} - M_{m,f}(0)|
  \leq e \pfrac{n^m}{(n)_m}^2 \pfrac{n!}{n^n}^3 \cdot \frac{n^{-1}/R}{1 - n^{-1}/R}.
\end{equation}
It will be useful to control the term $n^m  / (n)_m$ when $m = O(n^{1/2})$.
Note that for $t>0$,
\[
  (1-t)^{-1} = 1 + t/(1-t) \le \exp\bigl(t / (1-t)\bigr)
\]
and so
\begin{equation}%
  \label{eq:falling-fac-bound}
  \log \pfrac{n^m}{(n)_m} = \sum_{j=0}^{m-1} \log \pfrac1{1-j/n} \le \sum_{j=0}^{m-1} \frac{j/n}{1-m/n} \le \frac{m^2 / 2n}{1-m/n}.
\end{equation}

We deduce the following two corollaries of~\eqref{eq:cauchy-Mfm0-estimate-2}.
\begin{corollary}%
  \label{cor:Mmf0-approx}
If $m < 0.17 n^{1/2}$,
\[
  |M_{m, f} - M_{m, f}(0)| \leq O(m^2/n) \pfrac{n!}{n^n}^3.
\]
\end{corollary}
\begin{proof}
Take $R = (0.18/m)^2$ in~\eqref{eq:cauchy-Mfm0-estimate}.
The hypthesis and~\eqref{eq:falling-fac-bound} imply that $n^m/(n)_m = O(1)$.
\end{proof}

\begin{corollary}%
  \label{cor:Mmf0-approx-non-asymp}
If $n > 10^5$ and $m \leq 20$,
\[
  |M_{m, f} - M_{m, f}(0)| < 0.31 \pfrac{n!}{n^n}^3.
\]
\end{corollary}
\begin{proof}
Take $R = 10^{-4}$ in~\eqref{eq:cauchy-Mfm0-estimate-2}.
By~\eqref{eq:falling-fac-bound},
\[
  \frac{n^m}{(n)_m} \le \exp \pfrac{20^2 / 10^5}{1 - 20/10^5} < 1.0021.
\]
Thus~\eqref{eq:cauchy-Mfm0-estimate-2} is bounded by
\[
  (1.0021)^2 e \frac{n^{-1} / R}{1 - n^{-1} / R} \pfrac{n!}{n^n}^3 < 0.31 \pfrac{n!}{n^n}^3.\qedhere
\]
\end{proof}

\subsection{The constant term \texorpdfstring{$M_{m,f}(0)$}{Mmf(0)}}

We have
\[
  M_{m, f}(0) = \pfrac{n!}{n^n}^3 \sum_{\substack{|\supp\psystem| \leq m \\ \cx \psystem = 0}} \pfrac{n^{|\supp\psystem|}}{(n)_{|\supp \psystem|}}^3 \mu(\psystem) \gamma(\psystem, f) n^{-|\supp \psystem|}.
\]
As we have mentioned several times now, $\cx \psystem = 0$ if and only if $\psystem = (\calP, \calP, \calP)$ for some pairing $\calP$. In this case, where say $|\supp \psystem| = 2k$, we have (for $k = O(n^{1/2})$)
\begin{align*}
  &\frac{n^{|\supp \psystem|}}{(n)_{|\supp \psystem|}} = \frac{n^{2k}}{(n)_{2k}} \approx 1,\\
  &\mu(\psystem) = \mu(\calP)^3 = (-1)^k,\\
  &\gamma(\psystem, f) \approx c_\calP(f)
\end{align*}
(the last holds by Lemma~\ref{lem:gamma-pairing}).
Let $N_k$ be the number of pairings $\calP$ of support size $2k$ such that $f$ is $\calP$-measurable; in other words, $N_k$ is the number of unordered $k$-tuples of disjoint collisions of $f$. We have
\[
  N_k \approx \binom{\coll(f)}{k} \approx \frac{\coll(f)^k}{k!}.
\]
Thus we should have
\[
  M_{m,f}(0)
  \approx \pfrac{n!}{n^n}^3 \sum_{k=0}^{\floor{m/2}} (-1)^k \frac{\coll(f)^k}{k!} n^{-2k}
  = \pfrac{n!}{n^n}^3 \frakS_m(f).
\]

To be precise, assuming $2k \leq 0.17 n^{1/2}$ and $n > 10^5$ we have
$2k/n \le 0.17 / n^{1/2} < 0.00054$, so~\eqref{eq:falling-fac-bound} gives
\[
  1 \leq \frac{n^{2k}}{(n)_{2k}} \leq \exp(2.0015 k^2 / n)
\]
and by Lemma~\ref{lem:gamma-pairing} we have
\[
  \left|\gamma(\psystem, f) - c_\calP(f)\right| \leq k/n
\]
so
\begin{align*}
  \left| \pfrac{n^{2k}}{(n)_{2k}}^3 \gamma(\psystem, f) - c_\calP(f) \right|
  &\leq \pfrac{n^{2k}}{(n)_{2k}}^3 \left| \gamma(\psystem, f) - c_\calP(f) \right| + \left|\pfrac{n^{2k}}{(n)_{2k}}^3 - 1 \right| c_\calP(f) \\
  &\leq e^{6.004 k^2 / n} k/n + (e^{6.004 k^2/n} - 1) \\
  &\leq 3 k^2 / n
\end{align*}
(since $x \mapsto e^x  + (e^x - 1)/x$ is bounded by $3$ for $0 \le x \le 6.004 \times 0.17^2 / 4$).

The total number of pairings of support size $2k$ is
\[
  \frac{n!}{2^k k! (n-2k)!} \leq \frac{n^{2k}}{2^k k!}.
\]
Thus
\begin{equation} \label{eq:final-major-arc-approx-1}
  \left|
  \sum_{\substack{|\supp\psystem| = 2k \\ \cx \psystem = 0}} \pfrac{n^{2k}}{(n)_{2k}}^3 \mu(\psystem) \gamma(\psystem, f) n^{-2k}
  - (-1)^k N_k n^{-2k}
  \right|
  \leq \frac{3 k^2 / n}{2^k k!}.
\end{equation}
Also, the difference $\coll(f)^k - k! N_k$ is precisely the number of ordered $k$-tuples of collisions, at least two of which overlap, so
\[
  0 \leq \coll(f)^k - k! N_k \leq 2\binom{k}{2} n^{2k-1}.
\]
Thus
\begin{equation} \label{eq:final-major-arc-approx-2}
  \left| (-1)^k N_k n^{-2k} - (-1)^k \frac{\coll(f)^k}{k!} n^{-2k} \right| \leq \frac{k^2 n^{-1}}{k!}.
\end{equation}
Combining~\eqref{eq:final-major-arc-approx-1} and~\eqref{eq:final-major-arc-approx-2},
\begin{align*}
  \left| \pfrac{n!}{n^n}^{-3} M_{m,f}(0) - \frakS_m(f) \right|
  &\leq \sum_{2k\leq m} \left(\frac{3k^2  / n}{2^k k!} + \frac{k^2n^{-1}}{k!} \right) \\
  &\leq \frac1n \sum_{k=0}^{\infty} \left(\frac{3k^2}{2^k k!} + \frac{k^2}{k!} \right) \\
  &= \bigl( 3 (3 e^{1/2} / 4) + 2 e\bigr)/n < 10 / n.
\end{align*}
By combining with Corollary~\ref{cor:Mmf0-approx} we have
\[
  \left| \pfrac{n!}{n^n}^{-3} M_{m, f} - \frakS_m(f) \right|
  \leq O(m^2/n)
\]
provided $m < 0.17 n^{1/2}$, and by combining with Corollary~\ref{cor:Mmf0-approx-non-asymp} we have
\[
  \left| \pfrac{n!}{n^n}^{-3} M_{m, f} - \frakS_m(f) \right|
  < 0.32
\]
provided $n > 10^5$ and $m \leq 20$. This finishes the proof of Proposition~\ref{major-arcs}.

\section{Low-entropy minor arcs}%
\label{sec:low-entropy-minor-arcs}

\subsection{Sparseval}

For $\rho_1 \otimes \cdots \otimes \rho_n$, recall from Section~\ref{subsec:argproj} that $\supp \rho$ is the set of $i$ such that $\rho_i \neq 1$, and that $\rho$ is called $m$-sparse if $|\supp \rho| = m$.
The first part of following lemma, in the case that $G$ is abelian, is~\cite[Theorem~5.1]{EMM}.
We will deduce the general version essentially by demonstrating that the lemma does not depend on the group operation at all.

\begin{lemma}~\label{sparseval}
  We have
  \[
    \sum_{m\textup{-sparse}~\rho} \|\hatS(\rho)\|_\HS^2 \dim \rho \leq O(m^{1/4}) e^{O(m^{3/2} / n^{1/2})} \binom{n}{m}^{1/2} \pfrac{n!}{n^n}^2.
  \]
  In fact,
  \[
    \sum_{m\textup{-sparse}~\rho} \|\hatS(\rho)\|_\HS^2 \dim \rho \leq (1 - (m/n)^{1/2})^{-1} e^{s(m/n) n} \pfrac{n!}{n^n}^2,
  \]
  where
  \[
    s(t) = t^{1/2} - (1-t) \log(1 + t^{1/2}) - t \log(t^{1/2}).
  \]
\end{lemma}

Although expressed in terms of Fourier analysis, we will show that this lemma has nothing to do with group theory, and in particular we will deduce it from the abelian case.

\begin{proof}
  We recall also from Section~\ref{subsec:argproj} the projection operators $P_X$ on $L^2(G^n)$, for each $X \subset \{1,\dots,n\}$.  By Lemma~\ref{lem:P_X-on-fourier} and Parseval we have
  \[
    \|P_X 1_S\|^2 = \sum_{\supp \rho = X} \|\hatS(\rho) \|_\HS^2 \dim \rho
  \]
  and hence
  \begin{equation} \label{eq:reducingtoEMM}
      \sum_{m\textup{-sparse}~\rho} \|\hatS(\rho)\|_\HS^2 \dim \rho = \sum_{|X| = m} \|P_X 1_S\|^2.
  \end{equation}
  However, the right-hand side does not depend on the group operation on $G$, so the first statement actually \emph{follows} from~\cite[Theorem~5.1]{EMM}.

  The more precise second bound follows from the proof of~\cite[Theorem~5.1]{EMM}. Following the notation of that proof, let the quantity in~\eqref{eq:reducingtoEMM} be equal to
  \[
  Q(m,n)n!^2/n^{2n}.
  \]
  Then for any $r$ in the range $0 < r < n$ we have
  \[
    Q(m, n) \leq \max_\pm \frac{e^{\pm r}}{(1\pm r/n)^{n-m+1}} \frac{n^m}{r^m}.
  \]
  Taking $r = (mn)^{1/2}$, and writing $t = m/n$, it follows that
  \begin{align*}
    Q(m, n)
     & \leq \max_\pm \exp \left( \pm t^{1/2} - (1-t + 1/n) \log(1 \pm t^{1/2}) - t \log (t^{1/2}) \right)^n               \\
     & \leq (1 - t^{1/2})^{-1} \max_\pm \exp \left( \pm t^{1/2} - (1-t) \log(1 \pm t^{1/2}) - t \log (t^{1/2}) \right)^n.
  \end{align*}
  The larger value is achieved by $+$, where we get $(1-t^{1/2})^{-1} e^{s(t)n}$.
\end{proof}

\subsection{An inverse theorem for \texorpdfstring{$m$}{m}-sparse representations}

In this subsection we prove the following uniform bound for the operator norm $\|\hatS(\rho)\|_\op$ for $m$-sparse $\rho$  (cf.~\cite[Lemma~4.1]{moreon}).

\begin{lemma}%
  \label{op-norm-bound}
  Let $\rho = \rho_1 \otimes \cdots \otimes \rho_n$ be an $m$-sparse representation of $G^n$, where $m \leq n/2$. Then
  \[
    \|\hatS(\rho)\|_\op \leq \binom{n}{m}^{-1/2} \frac{n!}{n^n}.
  \]
\end{lemma}

We will need to know when this bound is sharp within a subexponential factor. The following ``inverse theorem'' characterizes this situation: this bound is nearly sharp only when $\rho$ is roughly $\rho_0^m \otimes 1^{n-m}$ for some one-dimensional $\rho_0$ of order two.

\begin{theorem}%
  \label{op-norm-bound-inverse-theorem}
  Suppose $\rho$ is an $m$-sparse representation of $G^n$, where $m\leq n/3$,\footnote{We could replace $n/3$ by $0.49n$ at the cost of worsening the constant $0.99$.} and suppose that no more than $(1-\eps)m$ of the nontrivial factors of $\rho$ are equal to the same one-dimensional representation $\rho_0$ of order two. Then
  \[
    \|\hatS(\rho)\|_\op \leq 0.99^{\eps m} \binom{n}{m}^{-1/2} \frac{n!}{n^n}.
  \]
\end{theorem}

We begin with the proof of Lemma~\ref{op-norm-bound}. Assume for notational convenience that $\rho = \rho_1 \otimes \cdots \otimes \rho_m \otimes 1^{n-m}$, where $\rho_i\colon G \to U(V_i)$ ($1 \leq i \leq m$) are nontrivial irreducible representations of $G$ (permuting the factors does not affect the operator norm). Then $\hatS(\rho)$ is an operator on $V_1 \otimes \cdots \otimes V_m$, and by definition
\[
  \hatS(\rho)
  = \frac{(n-m)!}{n^n} \sum_{\substack{x_1, \dots, x_m\\\text{distinct}}} \rho_1(x_1) \otimes \cdots \otimes \rho_m(x_m).
\]
Since $\rho_m$ is nontrivial and irreducible, $\sum_x \rho_m(x) = 0$, so
\begin{equation}\label{section-6-method-step-one}
  \hatS(\rho)
  = -\frac{(n-m)!}{n^n} \sum_{\substack{x_1, \dots, x_{m-1}\\ \text{distinct}}} \sum_{x_m \in \{x_1, \dots, x_{m-1}\}} \rho_1(x_1) \otimes \cdots \otimes \rho_m(x_m).
\end{equation}
Exchanging the order of summation, we write $\hatS(\rho) = \sum_{j=1}^{m-1} R_j$ where
\[
  R_j = -\frac{(n-m)!}{n^n} \sum_{\substack{x_1, \dots, x_{m-1}\\ \text{distinct}}} \rho_1(x_1) \otimes \cdots \otimes \rho_{m-1}(x_{m-1}) \otimes \rho_m(x_j)
\]
for $1 \le j \le m-1$. For simplicity consider $R_{m-1}$. In equivalent notation,
\begin{equation} \label{section-6-method-one-term}
  R_{m-1} = -\frac{(n-m)!}{n^n} \sum_{\substack{x_1, \dots, x_{m-1} \\ \text{distinct}}} \rho_1(x_1) \otimes \cdots \otimes (\rho_{m-1} \hatotimes \rho_m)(x_{m-1}).
\end{equation}
We can decompose $\rho_{m-1} \hatotimes \rho_m$ as an orthogonal direct sum of irreducible representations:
\[
  \rho_{m-1} \hatotimes \rho_m = \sigma_1 \oplus \cdots \oplus \sigma_k,
\]
and correspondingly $R_{m-1} = \bigoplus_{r=1}^k R_{m-1, \sigma_r}$ where
\[
  R_{m-1,\sigma_r} = -\frac{(n-m)!}{n^n} \sum_{\substack{x_1, \dots, x_{m-1} \\ \text{distinct}}} \rho_1(x_1) \otimes \cdots \otimes \sigma_r(x_{m-1}).
\]
We observe that $R_{m-1,\sigma_r}$ is essentially the same as $-\frac1{n-m+1} \hatS(\rho_1 \otimes \cdots \otimes \rho_{m-2} \otimes \sigma_r)$, and certainly these have the same operator norm.
Since the direct sum above is orthogonal, and since $\|R\oplus S\|_\op  = \max(\|R\|_\op, \|S\|_\op)$, we have
\[
  \|R_{m-1}\|_\op
  = \max_{1 \leq r \leq k} \frac{1}{n-m+1} \|\hatS(\rho_1 \otimes \cdots \otimes \rho_{m-2} \otimes \sigma_r)\|_\op.
\]
Note that $\rho_1 \otimes \cdots \otimes \rho_{m-2} \otimes \sigma_r$ is either $(m-1)$- or $(m-2)$-sparse, depending on whether $\sigma_r$ is trivial.

The situation for other $R_j$ is identical up to permuting factors.  Applying the triangle inequality to $\hatS(\rho) = \sum_{j=1}^{m-1} R_j$, we deduce
\[
  \|\hatS(\rho)\|_\op \leq
  \frac{m-1}{n-m+1} \max_\text{$(m-1)$- or $(m-2)$-sparse $\rho'$} \|\hatS(\rho')\|_\op.
\]
The claimed bound (which is monotonic in $m$ for $m\leq n/2$) follows from this by induction, with the base cases $\hatS(1^n) = n!/n^n$ and $\hatS(\rho_1 \otimes 1^{n-1}) = 0$.

\begin{remark}%
  \label{remark-lazy}
  We are being a little lazy with the form of the bound. The same recurrence actually proves
  \[
    \|\hatS(\rho)\|_\op \leq \frac{(m-1) (m-3) \cdots 1}{(n-m+1) (n-m+3) \cdots (n-1)} \cdot \frac{n!}{n^n}
  \]
  when $m$ is even, and
  \[
    \|\hatS(\rho)\|_\op \leq \frac{m-1}{n-m+1} \cdot \frac{(m-2) (m-4) \cdots 1}{(n-m+2) (n-m+4) \cdots (n-1)} \cdot \frac{n!}{n^n}
  \]
  when $m$ is odd.
\end{remark}

\def\im{\operatorname{im}}

In order to prove Theorem~\ref{op-norm-bound-inverse-theorem}, we re-examine the above proof.  Given $\rho = \rho_1 \otimes \dots \otimes \rho_m \otimes 1^{n-m}$, and given indices $1 \le i < j \le m$ and $\sigma$ an irreducible component of $\rho_i \hatotimes \rho_j$, we write
\[
  \rho'_{i,j,\sigma} = \rho_1 \otimes \dots \otimes \sigma \otimes \dots \otimes 1 \otimes \dots \otimes \rho_m \otimes 1^{n-m}
\]
for the representation obtained by replacing $\rho_i$ with $\sigma$ and $\rho_j$ with the trivial representation.  (In the above, we always permute factors so that $i=m$.)

Two potentially weak inequality steps in the above proof are (i) pessimistically assuming that $\rho'_{i,j,\sigma}$ is $(m-2)$-sparse rather than $(m-1$)-sparse when applying the recursive bound, and (ii) using the triangle inequality on $\left\|\sum_{j=1}^{m-1} R_j\right\|_\op$.

These and all other steps in the proof are sharp when each of the $m$ nontrivial factors $\rho_i$ is equal to the same one-dimensional $\rho_0$ of order two: in this case the representation $\sigma_i$ is always the trivial representation, so the sparsity is indeed $m-2$, and because the representation is one-dimensional the triangle inequality is sharp. Thus
\[
  \|\hatS(\rho)\|_\op = \frac{(m-1)(m-3) \cdots 1}{(n-m+1)(n-m+3) \cdots (n-1)} \frac{n!}{n^n}
\]
for such $\rho$.

Let us say that $\rho$ has \emph{height} $h$ if it takes $h$ iterations of the recursion in the proof of Lemma~\ref{op-norm-bound} to get to a representation of the above form. In other words,
\begin{enumerate}
  \item $\rho$ has height zero if, up to permutation of factors, $\rho = \rho_0^m \otimes 1^{n-m}$ for some even $m$ and some one-dimensional $\rho_0$ of order two;
  \item $\rho = \rho_1 \otimes \dots \otimes \rho_m \otimes 1^{n-m}$ has height at most $h$ if one can pick indices $1 \le i < j \le m$ and an irreducible component $\sigma$ of $\rho_i \hatotimes \rho_j$, such that $\rho'_{i,j,\sigma}$ has
        height at most $h-1$.
\end{enumerate}
Note that $\rho$ has finite height if and only if $\rho_1 \hatotimes \cdots \hatotimes \rho_n$ contains a copy of the trivial representation, i.e., if and only if $\langle \chi_1 \cdots \chi_n, 1 \rangle \neq 0$. If $\rho$ does not have finite height then $\hatS(\rho) = 0$.

Using height for bookkeeping, we will use the idea of the proof of Lemma~\ref{op-norm-bound} to prove the following recurrence.

\begin{proposition}%
  \label{recurrence-proposition}
  Let $F(m,h)$ be the maximum value of $\|\hatS(\rho)\|_\op$ over all $m$-sparse $\rho$ of height at least $h$. If $m \geq 2$ and $h>0$ then
  \[
    F(m, h) \leq \frac{m-1}{n-m+1} \max\begin{cases}
      F(m-1, h-1) \\
      (1-\theta) F(m-1, h-1) + \theta F(m-2, h-1),
    \end{cases}
  \]
  where
  \[
    \theta = \max \left(
    \frac{m}{2m-2},
    \left(\frac1{m-1} \left(1 + \frac{m-2}{d}\right)\right)^{1/2}
    \right) ,
  \]
  and where $d$ is the minimal dimension of a non-one-dimensional self-dual representation of $G$.
\end{proposition}

We recall some representation-theoretic preliminaries.  For a representation $U$ of $G$, we denote by $U^G$ the $G$-invariant subspace of $U$.  The case $\sigma = 1$ in the above discussion (which is of interest as this is when the sparsity decreases by $2$) corresponds to considering the subspaces $(V_i \otimes V_j)^G$.

For given representations $U, V$, there is a natural correspondence between $U \otimes V$ and the space of linear maps $U^* \to V$, where $U^\ast$ is the linear dual (identifying $u \otimes v$ with the map $\psi \mapsto \psi(u) v$).  The subspace $(U \otimes V)^G$ corresponds to the $G$-equivariant maps $U^\ast \to V$.  If $U,V$ are irreducible, by Schur's lemma the space of $G$-equivariant maps $U^\ast \to V$ has dimension $1$ if $U^\ast$ and $V$ are isomorphic as $G$-representations (spanned by such an isomorphism) and is zero otherwise.  If $V = U^\ast$, the element of $(U \otimes V)^G$ corresponding to the identity may be written explicitly as
\[
  \sum_{i=1}^d u_i \otimes u_i^\ast
\]
where $u_1,\dots,u_d$ is any basis for $U$ and $u_1^\ast,\dots,u_d^\ast$ is the dual basis.

It follows that the case $\sigma = 1$ can only arise if $V_i$ and $V_j$ are dual to each other.  In this case, we will need the following lemma, which can be interpreted as saying that the subspaces of $V_1 \otimes \dots \otimes V_m$ induced by $(V_i \otimes V_j)^G$ are somewhat orthogonal across different choices of $j$.

\begin{lemma}%
  \label{lemma-somewhat-orthog}
  Let $\rho\colon G \to U(V)$ be an irreducible representation of $G$.
  Suppose $v,w \in V \otimes V^\ast \otimes V$ are unit vectors such that
  \[
    v \in (V \otimes V^\ast)^G \otimes V
  \]
  and
  \[
    w \in V \otimes (V^\ast \otimes V)^G.
  \]
  Then $|\langle v, w \rangle| \le 1/\dim V$.
\end{lemma}
\begin{proof}
  Let $u_1,\dots,u_d$ be any orthonormal basis for $V$ where $d=\dim V$.  By the discussion above, we may write
  \[
    v = \left(\sum_{i=1}^d u_i \otimes u_i^\ast \right) \otimes v'
  \]
  and
  \[
    w = w' \otimes \left(\sum_{i=1}^d u_i^\ast \otimes u_i \right)
  \]
  for some $v',w' \in V$.  Then
  \[
    \|v\|^2 = 1 = \sum_{i=1}^d \|u_i \otimes u_i^\ast \otimes v'\|^2 = d \|v'\|^2
  \]
  so $\|v'\| = 1/\sqrt{d}$, and similarly $\|w'\| = 1/\sqrt{d}$.  But
  \begin{align*}
    \langle v, w \rangle = & \sum_{i,j=1}^d \langle u_i, w' \rangle\, \langle u_i^\ast, u_j^{\ast} \rangle\, \langle v', u_j \rangle \\
    =                      & \sum_{i=1}^d \langle v', u_i \rangle\,\langle u_i, w' \rangle                                           \\
    =                      & \langle v', w' \rangle
  \end{align*}
  and so $|\langle v, w \rangle| \le \|v'\|\,\|w'\| = 1/d$ as claimed.
\end{proof}

\begin{proof}[Proof of Proposition~\ref{recurrence-proposition}]
  Suppose $\rho$ has sparsity $m \geq 2$ and height $h>0$. We may assume $\rho = \rho_1 \otimes \cdots \otimes \rho_m \otimes 1^{n-m}$, where each $\rho_i$ is nontrivial. Since $\rho$ has positive height, one of the following alternatives holds:
  \begin{enumerate}
    \item there is some nontrivial factor, without loss of generality $\rho_m$, that is dual to at most $m/2$ of the other factors $\rho_i$;
    \item $\rho = \rho_0^m \otimes 1^{n-m}$ for some self-dual $\rho_0$ of dimension at least $2$.
  \end{enumerate}

  Assume (1) holds. Then we proceed as in the proof of Lemma~\ref{op-norm-bound}. Let $A \subset \{1,\dots,m-1\}$ be the indices $j$ such that $\rho_j \cong \rho_m^\ast$; so $|A| \le m/2$.  The tensor product $\rho_j \hatotimes \rho_m$ ($1 \leq j \leq m-1$) contains a trivial component only if $j \in A$, so in the language of the proof of Lemma~\ref{op-norm-bound},
  \[
    \|R_j\|_\op \le \frac1{n-m+1} \begin{cases} F(m-2, h-1) &\colon j \in A \\ F(m-1,h-1) &\colon j \notin A. \end{cases}
  \]
  Applying the triangle inequality to $\hatS(\rho) = \sum_{j=1}^{m-1} R_j$,
  \[
    \|\hatS(\rho)\|_\op
     \leq \frac{m-1}{n-m+1} \left( (1-\theta) F(m-1, h-1) + \theta F(m-2, h-1) \right)
  \]
  for some $\theta \leq (m/2) / (m-1)$.

  Now assume (2) holds, for say $\rho_0 : G \to U(V_0)$. Thus $V_1, \dots, V_m \cong V_0$. For $1 \leq j \leq m-1$, let $\Phi_j$ be the projection from $V_1 \otimes \cdots \otimes V_m$ to the subspace obtained by replacing $V_j \otimes V_m$ with its $G$-invariant subspace $(V_j \otimes V_m)^G$.  By Lemma~\ref{lemma-somewhat-orthog}, for any $1 \le i < j \le m-1$, if $u \in \im \Phi_i$ and $w \in \im \Phi_j$ then\footnote{It is straightforward to show that if $U, U'$ are subspaces with an upper bound on inner products of this form, then $U \otimes W$ and $U' \otimes W$ obey the same bound for any inner product space $W$.}
  \[
    |\langle u, w \rangle|  \le \|u\| \|w\| /\dim V_0.
  \]
  Again in the language of Lemma~\ref{op-norm-bound}, recall that
  \[
    \hatS(\rho)
    = \sum_{j=1}^{m-1} \bigoplus_\sigma R_{j,\sigma},
  \]
  where the direct sum runs over irreducible components $\sigma$ of $\rho_j \hatotimes \rho_m$, and the operator $R_{j,\sigma}$ acts like $-\frac1{n-m+1} \hatS(\rho_1 \otimes \cdots \otimes \widehat{\rho_j} \otimes \cdots \otimes \rho_{m-1} \otimes \sigma)$. Note that $R_{j,1} = R_{j,1} \Phi_j$, and $R_{j, \sigma} \Phi_j = 0$ for $\sigma$ nontrivial. Thus for a unit vector $v \in V_1 \otimes \dots \otimes V_m$,
  \begin{align*}
    \left\| R_j v \right\|^2
     & = \sum_\sigma \|R_{j,\sigma} v\|^2                                                                       \\
     & = \|R_{j,1} \Phi_j v\|^2 + \sum_{\sigma \ne 1} \|R_{j,\sigma} (1 - \Phi_j) v\|^2                         \\
     & \leq \|R_{j,1}\|_\op^2 \|\Phi_j v\|^2 + \max_{\sigma\ne 1} \|R_{j,\sigma}\|_\op^2 (1 - \|\Phi_j v \|^2).
  \end{align*}
  Note that
  \[
    \|R_{j,1}\|_\op \leq \frac1{n-m+1} F(m-2, h-1)
  \]
  and
  \[
    \max_{\sigma \ne 1} \|R_{j,\sigma}\|_\op \leq \frac1{n-m+1} F(m-1, h-1).
  \]
  Hence
  \[
    \|R_j v\| \leq \frac1{n-m+1} \left( F(m-2, h-1)^2 \|\Phi_j v\|^2 + F(m-1, h-1)^2 (1 - \|\Phi_j v\|^2) \right)^{1/2},
  \]
  and
  \begin{align*}
     & \| \hatS(\rho)\, v \| \le \frac1{n-m+1} \cdot                                                                      \\
     & \qquad \sum_{j=1}^{m-1} \left(F(m-2, h-1)^2 \, \|\Phi_j v\|^2  + F(m-1,h-1)^2 (1 - \|\Phi_j v \|^2) \right)^{1/2}.
  \end{align*}
  Applying Cauchy--Schwarz and rearranging gives
  \begin{equation} \label{almost-final-hats-bound}
    \| \hatS(\rho)\, v \| \le \frac{m-1}{n-m+1} \left( \theta F(m-2,h-1)^2 + (1 - \theta)  F(m-1, h-1)^2 \right)^{1/2},
  \end{equation}
  where
  \[
    \theta = \frac1{m-1} \sum_{j=1}^{m-1} \| \Phi_j v \|^2.
  \]

  To bound $\theta$, note that if $u_1,\dots,u_{m-1}$ are vectors with $u_j \in \im \Phi_j$, then
  \[
    \left\| \sum_{j=1}^{m-1} u_j \right\|^2 = \sum_{j} \|u_j\|^2 + \sum_{i \ne j} \langle u_i, u_j \rangle,
  \]
  which by Lemma~\ref{lemma-somewhat-orthog} and the AM--GM inequality is bounded by
  \[
    \sum_{j} \|u_j\|^2 + \sum_{i \ne j} \frac1{\dim V_0} \|u_i\| \|u_j\| \le \left(1 + \frac{m-2}{\dim V_0}\right) \sum_{j=1}^{m-1} \|u_j\|^2.
  \]
  In other words, the map $\bigoplus_{j=1}^{m-1} \im \Phi_j \to V_1 \otimes \dots \otimes V_m$ sending $(u_1,\dots,u_{m-1}) \mapsto u_1+\cdots+u_{m-1}$ has operator norm at most $\left(1 + \frac{m-2}{\dim V_0}\right)^{1/2}$, and hence so does its adjoint, which is the map $v \mapsto \big(\Phi_j v)_{j=1}^{m-1}$.  It follows that
  \[
    \sum_{j=1}^{m-1} \|\Phi_j v\|^2 \le \left(1 + \frac{m-2}{\dim V_0}\right) \|v\|^2
  \]
  and so
  \[
    \theta \le \frac1{m-1} \left(1 + \frac{m-2}{\dim V_0}\right).
  \]

  Finally, note the inequality
  \[
    ((1 - \theta) x^2 + \theta y^2)^{1/2} \leq (1-\theta^{1/2}) x + \theta^{1/2} y
  \]
  for $0 \leq x \leq y$ and $\theta \in [0,1]$. Indeed we have
  \begin{align*}
    (1-\theta) x^2 + \theta y^2
    &= (1-\theta^{1/2})(1+\theta^{1/2}) x^2 + \theta y^2 \\
    &\leq (1-\theta^{1/2})^2 x^2 + 2\theta^{1/2} (1-\theta^{1/2}) xy + \theta y^2 \\
    &= \left((1-\theta^{1/2}) x + \theta^{1/2} y\right)^2.
  \end{align*}
  This completes the proof.
\end{proof}

\begin{corollary}%
  \label{cor:Fmh-bound}
  For $m \leq n/3$ we have
  \[
    F(m,h) \leq 0.97^h \binom{n}{m}^{-1/2} \frac{n!}{n^n}.
  \]
\end{corollary}
\begin{proof}
We use the previous proposition and induction on height. The case $h = 0$ follows from Lemma~\ref{op-norm-bound}. If $m = 2$ and $h=1$, from Remark~\ref{remark-lazy} we have
\[
  \|\hatS(\rho)\|_\op \leq \frac1{n-1} \frac{n!}{n^n} = \frac{(n/(n-1))^{1/2}}{2^{1/2}} \binom{n}{2}^{-1/2} \frac{n!}{n^n} \le 0.78 \binom{n}{2}^{-1/2} \frac{n!}{n^n}
\]
(as $n \ge 6$). Hence assume $h \geq 1$, $m\geq 3$. Then by the previous proposition and the inductive hypothesis we have
\begin{align*}
  F(m, h) &\leq \frac{m-1}{n-m+1} \left( (1-\theta) \binom{n}{m-1}^{-1/2} + \theta \binom{n}{m-2}^{-1/2} \right) 0.97^{h-1} \frac{n!}{n^n} \\
  &= \frac{m-1}{n-m+1} \left( (1-\theta) \left(\frac{n-m}{m} \right)^{1/2} + \theta \left(\frac{(n-m)(n-m+1)}{m(m-1)} \right)^{1/2} \right) \\
  &\qqquad \times 0.97^{h-1} \binom{n}{m}^{-1/2} \frac{n!}{n^n} \\
  &\leq \left( (1-\theta) \pfrac{m-1}{n-m+1}^{1/2} + \theta \right) 0.97^{h-1} \binom{n}{m}^{-1/2} \frac{n!}{n^n}.
\end{align*}
Now note that if $3 \leq m \leq n/3$ then $\theta \leq (3/4)^{1/2}$, so
\[
  (1-\theta) \pfrac{m-1}{n-m+1}^{1/2} + \theta \leq 0.97. \qedhere
\]
\end{proof}

The claimed inverse theorem, Theorem~\ref{op-norm-bound-inverse-theorem}, follows directly: if no more than $(1-\eps)m$ of the nontrivial factors of $\rho$ are equal to the same one-dimensional $\rho_0$ of order two, then $\rho$ has height at least $\eps m / 2$.

\subsection{The \texorpdfstring{$m$}{m}-sparse contribution}

The total contribution to~\eqref{fourier-expression} from $m$-sparse representations is bounded by
\[
  C_m = \sum_{m\textup{-sparse}~\rho} \|\hatS(\rho)\|_\op \|\hatS(\rho)\|_\HS^2 \dim \rho.
\]
We now use the bounds proved in this section to bound this sum (and in particular prove~\eqref{low-entropy-minor-arcs-estimate}). Recall that, together with Lemma~\ref{lem-onedim-shift} and the major arc bounds, this dispatches all representations $\rho = \rho_1 \otimes \dots \otimes \rho_n$ where all but $m$ representations $\rho_i$ are equal to the same one-dimensional representation $\rho_0$.

\begin{proposition}%
  \label{low-entropy-minor}
  There is a constant $c>0$ such that
  \[
    C_m \leq
    O\left(
    e^{-c \frac{\log(n/m)}{\log n} m} \pfrac{n!}{n^n}^3
    \right)
  \]
  for $m \leq c n / (\log n)^2$ and sufficiently large $n$.
\end{proposition}

Call an $m$-sparse representation $\rho$ \emph{exceptional} if more than $(1-\eps) m$ of its nontrivial factors are equal to the same one-dimensional $\rho_0$ of order two, where $\eps>0$ is a parameter we can optimize. Let $E_m$ be the set of exceptional $\rho$. By Theorem~\ref{op-norm-bound-inverse-theorem} we have
\[
  \max_{\rho \notin E_m} \|\hatS(\rho)\|_\op \leq 0.99^{\eps m} \binom{n}{m}^{-1/2} \frac{n!}{n^n},
\]
so by Lemma~\ref{sparseval},
\begin{equation}
  \label{unexceptional-contribution}
  \sum_{\rho \notin E_m} \|\hatS(\rho)\|_\op \|\hatS(\rho)\|_\HS^2 \dim \rho
  \leq O(m^{1/4}) e^{O(m^{3/2} / n^{1/2})}0.99^{\eps m} \pfrac{n!}{n^n}^3.
\end{equation}

For $\rho \in E_m$ we just use Lemma~\ref{op-norm-bound}. Note that $\dim \rho \leq n^{\eps m / 2}$ (since irreducible representations of $G$ have dimension at most $n^{1/2}$). Thus
\begin{align*}
  \|\hatS(\rho)\|_\op \|\hatS(\rho)\|_\HS^2 \dim \rho
   & \leq \|\hatS(\rho)\|_\op^3 (\dim \rho)^2               \\
   & \leq \binom{n}{m}^{-3/2} n^{\eps m} \pfrac{n!}{n^n}^3.
\end{align*}
The number of $\rho \in E_m$ is at most
\[
  \binom{n}{m} m^{\eps m + 1} n^{\eps m + 1}
\]
(as there are $\binom{n}{m}$ ways to choose which factors should be nontrivial, at most $m$ ways to choose how many $\rho_i$ should be equal to $\rho_0$, at most $\binom{m}{\lfloor \eps m \rfloor} \le m^{\eps m}$ ways to choose which factors should be equal to $\rho_0$, and at most $n^{\eps m+1}$ ways to choose $\rho_0$ and the other $\rho_i$ from the $n$ or fewer irreducible representations of $G$). Thus
\begin{align*}
  \sum_{\rho \in E_m} \|\hatS(\rho)\|_\op \|\hatS(\rho)\|_\HS^2 \dim \rho
   & \leq \binom{n}{m}^{-1/2} n^{3\eps m + 2} \pfrac{n!}{n^n}^3 \\
   & \leq (m/n)^{m/2} n^{3 \eps m + 2} \pfrac{n!}{n^n}^3.
\end{align*}
The proposition follows from this and~\eqref{unexceptional-contribution} by taking $\eps = \frac1{10} \frac{\log (n/m)}{\log n}$.

Finally, we note some quantitative improvements in the case that $G$ has no low-dimensional self-dual representations. In particular assume that $|G^\ab|$ is odd. Then there are in fact no order-two one-dimensional representations $\rho_0$, and so every $m$-sparse representations has height at least $m/2$.
Hence, under this hypothesis the height may be completely ignored: writing $F(m)$ for the maximum value of $\|\hatS(\rho)\|_\op$ over all $m$-sparse $\rho$,
Proposition~\ref{recurrence-proposition} states more simply that for $m \ge 2$,
\begin{equation}
  \label{eq:fm}
  F(m) \le \frac{m-1}{n-m+1} \max \begin{cases} F(m-1) \\ (1-\theta) F(m-1) + \theta F(m-2), \end{cases}
\end{equation}
where $\theta$ is as in the original statement.

Let us now assume $n \geq 10^5$ and $G$ has no self-dual representation of dimension less than $4$.
Then in Proposition~\ref{recurrence-proposition} we have
\[
  \theta \le \frac12 \pfrac{m+2}{m-1}^{1/2}
\]
(recalling $m \ge 2$).  If we define $\eta(m)$ by
\[
  F(m) = \eta(m) \binom{n}{m}^{-1/2} \frac{n!}{n^n}
\]
then~\eqref{eq:fm} implies the bound (for $m \ge 2$)
\[
  \eta(m) \leq \left( (1-\theta) \pfrac{m-1}{n-m+1}^{1/2} + \theta \right) \max\big(\eta(m-1), \eta(m-2)\big).
\]
Using this recurrence and the initial values $\eta(0) = 1$ and $\eta(1) = 0$ (see, e.g., Remark~\ref{remark-lazy}),
we can tabulate bounds for $\eta(m)$ for $m \leq 1000$, say, in the worst case $n = 10^5$.
For $m \in [1000, 0.06n]$, we claim
\[
    \eta(m) \leq 0.8^m.
\]
This holds for $m \in \{1000,1001\}$ by direct calculation, and for larger $m$ it suffices to calculate that $\theta < 0.51$ and
\[
  (1-\theta) \pfrac{m-1}{n-m+1}^{1/2} + \theta < 0.8^2.
\]
We use these effective bounds in the following proposition.
%

\begin{proposition}%
  \label{prop:explicit-low-entropy-minor-arcs}
Suppose that $G$ has no self-dual representation of dimension less than $4$ and $n > 10^5$. Then
\[
  \sum_{m=21}^{0.06n} C_m
  < 0.12 \pfrac{n!}{n^n}^3.
\]
\end{proposition}
\begin{proof}
By the previous discussion and Lemma~\ref{sparseval} we have
\[
  C_m \leq \bigl(1 - (m/n)^{1/2}\bigr)^{-1} \eta(m) e^{s(m/n) n} \binom{n}{m}^{-1/2} \pfrac{n!}{n^n}^3.
\]
Note that
\[
  \log \binom{n}{m} \ge \int_{n-m}^n \log x\ dx - \big(m \log m - m + 1 + \frac12 \log m\big) = n h(m/n) - \frac12 \log m - 1
\]
where $h(t)$ is the entropy function
\[
  h(t) = t \log(1/t) + (1-t) \log(1/(1-t)).
\]
Thus
\[
  C_m \leq \bigl(1 - (m/n)^{1/2}\bigr)^{-1} e^{1/2} m^{1/4} \eta(m) e^{f(m/n) m} \pfrac{n!}{n^n}^3,
\]
where
\[
  f(t) = \frac{s(t) - h(t)/2}{t}.
\]
Note that $f$ is monotonically increasing on $[0, 0.06]$.
Now by direct calculation, using tabulated bounds for $\eta(m)$, we have
\[
    e^{1/2} \sum_{m=21}^{1000} \bigl(1 - (m/10^5)^{1/2}\bigr)^{-1} m^{1/4} \eta(m) e^{f(m/10^5) m} < 0.11.
\]
On the other hand, the sum over $m \in [1001, 0.06n]$ is bounded by the convergent sum
\begin{align*}
  1.33 e^{1/2} \sum_{m=1001}^\infty m^{1/4} 0.8^m e^{f(0.06)m} &\le \frac{1.33 e^{1/2}}{1001^{3/4}} \sum_{m=1001}^\infty m \alpha^m \\
                                                               &= \frac{1.33 e^{1/2}}{1001^{3/4}} \frac{\alpha^{1001} (1001 - 1000 \alpha)}{(1-\alpha)^2} \\
                                                               &< 10^{-22}
\end{align*}
where $\alpha = 0.8 e^{f(0.06)}$.
\end{proof}

\section{High-entropy minor arcs}%
\label{sec:high-entropy-minor-arcs}

Finally we turn to the bound on high-entropy minor arcs,~\eqref{high-entropy-minor-arcs-estimate}.

\subsection{Bounding the Hilbert--Schmidt norm}

In the previous section we proved and used bounds for $\|\hatS(\rho)\|_\op$ for sparse $\rho$. In this subsection we prove the following general bound for $\|\hatS(\rho)\|_\HS$, which is the crucial ingredient for~\eqref{high-entropy-minor-arcs-estimate}.

\begin{theorem}%
  \label{nonabelian-GEC}
  Suppose $G$ is a group, and suppose $\rho = \rho_1^{a_1} \otimes \cdots \otimes \rho_k^{a_k}$, where $\rho_1, \dots, \rho_k$ are distinct irreducible representations of $G$. Let $d_j = \dim \rho_j$. Then
  \[
    \binom{n}{a_1, \dots, a_k} \|\hatS(\rho)\|_\HS^2 \dim \rho \leq  \frac{n!}{n^n} \prod_{j=1}^k \frac{(a_j + d_j^2 - 1)!}{a_j^{a_j} (d_j^2 - 1)!}.
  \]
\end{theorem}

The abelian case is worth highlighting, as it sharpens~\cite[Theorem~4.1]{EMM}. In this case $d_j=1$ for each $j$ and $\hatS(\rho)$ is a scalar, so we get
\begin{equation} \label{abelian-GEC}
  \binom{n}{a_1, \dots, a_k} |\hatS(\rho)|^2 \leq \frac{a_1! \cdots a_k!}{a_1^{a_1} \cdots a_k^{a_k}} \frac{n!}{n^n}.
\end{equation}
Another illustrative case is $k = 1, a_1 = n$: in this case the theorem states
\[
  \|\hatS(\rho_1^{\otimes n})\|_\HS^2 \leq \frac1{d^n} \binom{n + d^2-1}{d^2-1} \pfrac{n!}{n^n}^2,
\]
where $d = d_1 = \dim \rho_1$. Thus $\|\hatS(\rho_1^{\otimes n})\|_\HS^2$ is exponentially smaller than $(n!/n^n)^2$ whenever $\dim \rho_1 \geq 2$ (while, if $\dim \rho_1 = 1$, then $\hatS(\rho_1^{\otimes n})$ is $n!/n^n$ or $0$ depending on whether $\chi_1^n = 1$).

Similarly to the proof of Lemma~\ref{sparseval}, this theorem turns out to have very little to do with group theory: the bound holds in general for projections of $1_S$ onto tensor products of subspaces of $L^2(G)$ of dimension $d_j^2$, and this statement is independent of the group operation.  The full abstract formulation is Lemma~\ref{abstract-nonabelian-GEC} below. First we prove a key lemma in this direction.

\begin{lemma}%
  \label{GEC-lemma}
  Let $V$ be an inner product space with orthonormal basis $e_1, \dots, e_n$, and let $v_1, \dots, v_k \in V$ be orthogonal. For $r: \{1, \dots, n\} \to \{1, \dots, k\}$, write $r \sim (a_1, \dots, a_k)$ if $|r^{-1}(i)| = a_i$ for each $i$. Then
  \[
    \sum_{a_1 + \cdots + a_k = n} \left| \sum_{r \sim a} \langle e_1, v_{r(1)}\rangle \cdots \langle e_n, v_{r(n)}\rangle \right|^2
    \leq \left((|v_1|^2 + \cdots + |v_k|^2) / n \right)^n.
  \]
\end{lemma}
\begin{proof}
  Consider the integral
  \[
    I = \int_{(z_1,\dots,z_k) \in (S^1)^k} \prod_{i=1}^n \left| \sum_{j=1}^k \langle e_i, v_j \rangle z_j \right|^2.
  \]
  By expanding the product we get
  \[
    \int_{(z_1,\dots,z_k) \in (S_1)^k} \sum_{r,s\colon \{1, \dots,n\} \to \{1, \dots, k\}}
    \left( \prod_{i=1}^n \langle e_i, v_{r(i)}\rangle z_{r(i)} \right)
    \left( \prod_{i=1}^n \overline{\langle e_i, v_{s(i)}\rangle z_{s(i)}} \right).
  \]
  For any $j$, if $|r^{-1}(j)| \neq |s^{-1}(j)|$ then the product results in a nonzero power of $z_j$, so the integral vanishes, while if $|r^{-1}(j)| = |s^{-1}(j)|$ for all $j$ then the integrand is constant. Thus
  \[
    I = \sum_{a_1 + \cdots + a_k = n} \left| \sum_{r \sim a} \langle e_1, v_{r(1)}\rangle \cdots \langle e_n, v_{r(n)}\rangle \right|^2.
  \]
  On the other hand by the AM--GM inequality we have
  \begin{align*}
    I
     & \leq \int_{(z_1,\dots,z_k) \in (S^1)^k} \left(\frac1n \sum_{i=1}^n \left|\sum_{j=1}^k \langle e_i, v_j \rangle z_j \right|^2 \right)^n \\
     & = \int_{(z_1,\dots,z_k) \in (S^1)^k} \left(\frac1n \sum_{j=1}^k |v_j|^2 |z_j|^2 \right)^n                                       \\
     & = \left( (|v_1|^2 + \cdots + |v_k|^2) / n \right)^n. \qedhere
  \end{align*}
\end{proof}

In particular, suppose we fix $a = (a_1, \dots, a_k)$. Then
\[
  \left |\sum_{r\sim a} \langle e_1, v_{r(1)}\rangle \cdots \langle e_n, v_{r(n)} \rangle \right|^2 \leq ((|v_1|^2 + \cdots + |v_k|^2) / n)^n.
\]
Note that the left-hand side is $2a$-homogeneous in $v_1, \dots, v_k$. By applying the same inequality to the rescaled vectors $v_i' = \frac{a_i^{1/2}}{|v_i|} v_i$, so that $|v_i'|^2 = a_i$, we deduce that
\begin{equation}\label{abstract-GEC}
  \left |\sum_{r\sim a} \langle e_1, v_{r(1)}\rangle \cdots \langle e_n, v_{r(n)} \rangle \right|^2 \leq \frac1{a_1^{a_1} \cdots a_k^{a_k}} |v_1|^{2a_1} \cdots |v_k|^{2a_k}.
\end{equation}
(The inequality is trivial if any $v_i$ is zero.)
The abelian case~\eqref{abelian-GEC} of Theorem~\ref{nonabelian-GEC} follows directly from~\eqref{abstract-GEC} by taking $v_1, \dots, v_k$ to be $\rho_1, \dots, \rho_k \in L^2(G)$.

\begin{remark}
  Put another way, if $W$ is the matrix whose columns comprise $a_1$ copies of $v_1$, $a_2$ copies of $v_2$, etc., where $v_1,\dots,v_k$ are orthogonal, then the permanent $\per W$ obeys
  \begin{equation} \label{our-perm}
    |\per W| \leq \frac{a_1! \cdots a_k!}{a_1^{a_1/2} \cdots a_k^{a_k/2}}  |v_1|^{a_1} \cdots |v_k|^{a_k}.
  \end{equation}
  This is sharp when $v_i$ have unit-norm entries, and disjoint supports of size $a_i$.

  The inequality~\eqref{our-perm} can be compared with an inequality of Carlen, Lieb, and Loss~\cite[Theorem~1.1]{cll}, which states that
  \begin{equation} \label{their-perm}
    |\per W| \leq \frac{n!}{n^{n/2}} |w_1| \cdots |w_n|
  \end{equation}
  for any $n \times n$ matrix $W$ with columns $w_1, \dots, w_n$ (with no orthogonality condition). Neither result implies the other. They agree when $w_1=\cdots=w_n=v_1$, $k=1$ and $a_1=n$, in which case the result is just AM--GM.%

  In fact,~\eqref{our-perm} and~\eqref{their-perm} admit a common generalization. We have
  \begin{equation} \label{common-generalization}
    |\per W| \leq \frac{a_1! \cdots a_k!}{a_1^{a_1/2} \cdots a_k^{a_k/2}}  |w_1| \cdots |w_n|
  \end{equation}
  whenever the columns $w_1, \dots, w_n$ of $W$ can be partitioned into sets of sizes $a_1, \dots, a_k$ such that vectors in different sets are orthogonal. This follows from~\eqref{our-perm} and an observation originally due to Banach~\cite{banach} that the injective tensor norm of a symmetric tensor is achieved at diagonal tensors $x \otimes \dots \otimes x$: see the discussion in~\cite[Problem 73]{scottish}, and~\cite[Proposition 1.1(2)]{bs},~\cite[Theorem 4]{harris-1}, or~\cite[Theorem 2.1]{pappas-et-al} for modern proofs.
\end{remark}

The full nonabelian case of Theorem~\ref{nonabelian-GEC} is a little more involved.  We now give the analogous abstract statement.

\begin{lemma}%
  \label{abstract-nonabelian-GEC}
  Let $V$ be an inner product space with orthonormal basis $e_1, \dots, e_n$, let $W_1, \dots, W_k$ be orthogonal subspaces of $V$, let $d_i = \dim W_i$, and let $a_1, \dots, a_k \geq 0$ be integers such that $a_1 + \cdots + a_k = n$. Let $s \in V^{\otimes n}$ be the element
  \[
    s = \frac1{n!} \sum_{\sigma \in S_n} e_{\sigma(1)} \otimes \cdots \otimes e_{\sigma(n)}.
  \]
  Then
  \[
    \binom{n}{a_1, \dots, a_k} \big|P_{W_1^{a_1} \otimes \cdots \otimes W_k^{a_k}} (s)\big|^2 \leq  \|s\|^2 \prod_{j=1}^k \frac{(a_j + d_j - 1)!}{a_j^{a_j} (d_j-1)!}.
  \]
\end{lemma}
\begin{proof}
  Let $(w_{ij})_{1 \leq j \leq d_i}$ be an orthonormal basis for $W_i$, for each $i$. For $b = (b_{11}, \dots, b_{kd_k})$, write $w^b = w_{11}^{b_{11}} \otimes \cdots \otimes w_{kd_k}^{b_{kd_k}}$. Let $t_{ij} \ge 0$ be arbitrary scalars.  By applying Lemma~\ref{GEC-lemma} to the collection of all vectors $t_{ij}^{1/2} w_{ij}$ we have
  \[
    \sum_{b_{11} + \cdots + b_{kd_k} = n} \binom{n}{b_{11}, \dots, b_{kd_k}}^2 |\langle s, w^b \rangle|^2
    \,t_{11}^{b_{11}} \cdots t_{kd_k}^{b_{kd_k}}
    \leq \left((t_{11} + \cdots + t_{kd_k}) / n \right)^n.
  \]
  Let $B(a)$ be the set of those $b$ such that $b_{i1} + \cdots + b_{id_i} = a_i$ for each $i$. Then by simply restricting the sum above it is clear that
  \[
    \sum_{b \in B(a)} \binom{n}{b_{11}, \dots, b_{kd_k}}^2 |\langle s, w^b \rangle|^2
    \,t_{11}^{b_{11}} \cdots t_{kd_k}^{b_{kd_k}}
    \leq \left((t_{11} + \cdots + t_{kd_k}) / n \right)^n.
  \]
  We now integrate over all choices of $t_{ij} \ge 0$ satisfying $t_{i1} + \cdots + t_{id_i} = a_i$ for each $1 \leq i \leq k$.  Note the right-hand side is $1$ for all such choices.  Using
  \[
    \int_{x_1 + \cdots + x_m = a} x_1^{r_1} \cdots x_m^{r_m} = a^{r_1+\cdots+r_m} \frac{r_1! \cdots r_m!}{(r_1 + \cdots + r_m + m - 1)!},
  \]
  for any $a>0$ and integers $r_1,\dots,r_m \ge 0$, we get
  \[
    \sum_{b \in B(a)} \binom{n}{b_{11}, \dots, b_{kd_k}}^2 |\langle s, w^b \rangle|^2
    \left(\prod_{i,j} b_{ij}! \right)
    \left(\prod_{i} \frac{a_{i}^{a_i}}{(a_i + d_i -1)!} \right)
    \leq \prod_{i=1}^k \frac1{(d_i-1)!}
  \]
  and rearranging gives
  \[
    \sum_{b \in B(a)} \binom{n}{b_{11}, \dots, b_{kd_k}} |\langle s, w^b \rangle|^2 \leq \frac{(a_1 + d_1 - 1)!}{a_1^{a_1} (d_1-1)!} \cdots \frac{(a_k+d_k-1)!}{a_k^{a_k} (d_k-1)!} \frac1{n!}.
  \]
  The left-hand side is $\binom{n}{a_1, \dots, a_k} |P_{W_1^{a_1} \otimes \cdots \otimes W_k^{a_k}}(s)|^2$.
\end{proof}

Apply this to $V = L^2(G)$, $e_i = n^{1/2} 1_{g_i}$ for some enumeration $G=\{g_1,\dots,g_n\}$, and $W_j = \left\{ \langle v, \rho_j\rangle : v \in \HS(V_j)\right\}$ (the $\rho_j$-isotypic component). We have
\[
  \|\hatS(\rho)\|_\HS^2 \dim \rho = \|P_{W_1^{a_1} \otimes \cdots \otimes W_k^{a_k}} (1_S) \|_2^2.
\]
Note that $\dim W_j = d_j^2$ and $1_S =\frac{n!}{n^{n/2}} s$. Theorem~\ref{nonabelian-GEC} follows immediately.

\subsection{The sum over orbits}

Assume that $\rho_1, \dots, \rho_k$ is a complete list of the distinct irreducible representations of $G$. From Theorem~\ref{nonabelian-GEC} and elementary manipulations, it follows that
\begin{equation}
  \label{eq:one-orbit-bound}
  \binom{n}{a_1, \dots, a_k} \|\hatS(\rho)\|_\HS^3 \dim \rho
  \leq \prod_{j=1}^k d_j^{-a_j/2} \binom{a_j + d_j^2 -1}{a_j}^{3/2} \frac{a_j!^2}{a_j^{3a_j/2}} \frac{n!}{n^{3n/2}}.
\end{equation}
To show~\eqref{high-entropy-minor-arcs-estimate}, we need to bound the sum of the left-hand side over all $\rho$ in the high-entropy minor arcs.  As stated in Section~\ref{sec-outline}, it is cleaner to do this with generating function techniques.

If we sum~\eqref{eq:one-orbit-bound} over all choices of $a_1, \dots, a_k$ such that $a_1 + \cdots + a_k = n$, and such that $a_j \leq n-m$ wherever $d_j = 1$, we obtain exactly the coefficient of $z^n$ in the power series
\[
  \prod_{j=1}^k \theta_{d_j}(z) \cdot \frac{n!}{n^{3n/2}},
\]
where
\[
  \theta_d(z) = \sum_{a=0}^n d^{-a/2} \binom{a + d^2 - 1}{a}^{3/2} \frac{a!^2}{a^{3a/2}} z^a
\]
for $d > 1$, while for $d = 1$ we take the truncation
\[
  \theta_1(z) = \sum_{a=0}^{n-m} \frac{a!^2}{a^{3a/2}} z^a.
\]
To bound the sum, it therefore suffices to bound the generating functions $\theta_d(z)$ for some suitable choice of $z$.  It turns out the correct choice is always of the shape $z=w e^2 n^{-1/2}$ where $w \ge 1$ is some small constant.  With this in mind we prove the following technical bounds.

\begin{lemma}%
  \label{lem:theta-d}
  Let $\theta_d(z)$ be defined as above.
  \begin{enumerate}[label=(\roman*)]
    \item If $z \le 0.15$ and $n-m \le e^4 z^{-2} \big(1 - 0.66 z^2 \log (1/z)\big)$, then
    \[
      \theta_1(z) \leq \exp(z + z^3 / 10).
    \]
    \item If $n > 10^4$, $d > 1$ and $z \leq \min(0.9 d^{1/2}, 2) e^2 n^{-1/2}$, then
    \[
      \theta_d(z) \leq \exp(d^{5/2} z).
    \]
  \end{enumerate}
\end{lemma}
\begin{proof}
  We first consider (i). We will in fact show that, under these conditions,
  \[
    \theta_1(z) \le e^z + (1/10) (z^3 + z^4)
  \]
  which is sufficient (as $e^z \ge 1 + z$).  We have
  \[
    \theta_1(z) = 1 + z + z^2 / 2 + \frac{3!^2}{3^{9/2}} z^3 + \sum_{a=4}^{n-m} \frac{a!^2}{a^{3a/2}} z^a
  \]
  and hence
  \[
    \theta_1(z) - e^z \le \left(\frac{3!^2}{3^{9/2}} - \frac1{3!}\right) z^3 + \left(\sum_{a=4}^9 \left(\frac{a!^2}{a^{3a/2}} - \frac1{a!} \right) (0.15)^{a-4} \right) z^4 + \sum_{a=10}^{n-m} \frac{a!^2}{a^{3a/2}} z^a.
  \]
  By direct computation, the terms $3 \le a \le 9$ contribute at most $0.09 z^3 + 0.12 z^4$, and it is easy to check this is at most $0.095 (z^3+z^4)$ when $z \le 0.15$.

  For $a \ge 10$, using $a! \leq e\,a^{1/2} (a/e)^a$ we have
  \[
    \frac{a!^2}{a^{3a/2}} z^a \leq e^2 a (a^{1/2} z/e^2)^a = \exp(\chi(a))
  \]
  where $\chi$ is the function
  \[
    \chi(x) = 2 + \log x + x \big((1/2) \log x + \log z - 2 \big).
  \]
  We set
  \[
    A = e^4 z^{-2} \big(1- 0.66 z^2 \log (1/z)\big) > 2000
  \]
  Since $\chi$ is convex on $x \ge 2$, the maximum value of $\chi$ on the interval $[10,A]$ occurs at one of the endpoints.  We claim that in fact it occurs at $10$.  Indeed, we have
  \[
    \log A \le 4 + 2 \log (1/z) - 0.66 z^2 \log(1/z) \le 4 + 2 \log (1/z)
  \]
  and for $z \le 0.15$ we have $A \ge 0.9718 e^4 z^{-2}$, so
  \begin{align*}
    \chi(A) &= 2 + \log A + A \big((1/2) \log A + \log z - 2 \big) \\
    &\le 6 + 2 \log (1/z) - 0.9718 (0.33 e^4 \log (1/z) \big) \\
    &\le 6 - 15.5 \log (1/z)
  \end{align*}
  whereas
  \[
    \chi(10) \ge -4.2 - 10 \log (1/z)
  \]
  and it is straightforward to deduce $\chi(A) \le \chi(10)$ when $z \le 0.15$.  This proves the claim.

  Bounding each term by this maximum value, we deduce that
  \begin{align*}
    \sum_{a=10}^{n-m} \frac{a!^2}{a^{3a/2}} z^a &\le (n-m) \exp(\chi(10)) \\
    &\le (n-m) 0.0153 z^{10} \le (n-m)\, (1.162 \times 10^{-6}) z^5 \le 0.00007 z^3
  \end{align*}
  where we used the bounds $z \le 0.15$ and $n-m \le e^4 / z^2$ again.  Combining this with the bounds on $a=3,\dots,9$ gives (i).

  Now we consider (ii).  We write $z = u\, e^2 n^{-1/2}$, where $u \le \min\big(0.9 d^{1/2}, 2\big)$ by hypothesis.  We may expand
  \[
    \theta_d(z) = 1 + d^{5/2} z + \left(\frac{1+1/d^2}{2}\right)^{3/2} \frac{\big(d^{5/2} z\big)^2}{2!} + \sum_{a=3}^n d^{-a/2} \binom{a+d^2-1}{a}^{3/2} \frac{a!^2}{a^{3a/2}} z^a
  \]
  and note that $\left(\frac{1+1/d^2}{2}\right)^{3/2} \le (5/8)^{3/2} < 1/2$.  Hence it suffices to show that
  \begin{equation}
    \label{eq:thetad-stp-1}
    \sum_{a=3}^n d^{-a/2} \binom{a+d^2-1}{a}^{3/2} \frac{a!^2}{a^{3a/2}} z^a \le \frac{d^5 z^2}{4} + \sum_{a=3}^{n} \frac{(d^{5/2} z)^a}{a!}.
  \end{equation}
  We observe that
  \[
    d^{-a/2} \binom{a+d^2-1}{a}^{3/2} \frac{a!^2}{a^{3a/2}} z^a = \frac{(d^{5/2} z)^a}{a!} \left(\frac{a!}{a^a} \prod_{r=1}^{a-1} (1+r/d^2)\right)^{3/2}.
  \]
  We claim that, for $d \ge 2$ and $3 \le a \le 3 d^2$, we have the inequality
  \begin{equation}
    \label{eq:eta-inequality}
    \frac{a!}{a^a} \prod_{r=1}^{a-1} (1+r/d^2) \le 1.
  \end{equation}
  Indeed, we note that
  \[
    \sum_{r=0}^{a-1} \log (1 + r/d^2) \le \int_{0}^a \log (1 + x/d^2)\ dx = a \big((1 + d^2/a) \log (1 + a/d^2) - 1\big) = a\, \eta(a/d^2)
  \]
  where $\eta \colon [0,\infty) \to [0,\infty]$ is the monotonic function $\eta(t) = (1 + 1/t) \log(1+t) - 1$, where in particular $\eta(3) < 0.85$.  
  As usual we may bound $\log(a! / a^a) \le 1 + (1/2) \log a - a$; hence, the inequality holds provided
  \[
    -a + (1/2) \log a + 1 + 0.85 a < 0,
  \]
  i.e., provided $a \ge 16$.  On the other hand, for fixed $a$ the left-hand side of~\eqref{eq:eta-inequality} is monotonic in $d$, so it suffices to check the cases $d=2$, $a \in \{3,\dots,12\}$ and $d=3$, $a \in \{13,14,15\}$ of~\eqref{eq:eta-inequality} directly, which all hold by direct calculation, proving the claim.

  The bound~\eqref{eq:eta-inequality} handles the terms $3 \le a \le 3 d^2$ in~\eqref{eq:thetad-stp-1}, and in particular it now suffices to show that
  \[
    \sum_{a=3 d^2 + 1}^n d^{-a/2} \binom{a+d^2-1}{a}^{3/2} \frac{a!^2}{a^{3a/2}} z^a \le \frac{d^5 z^2}{4}.
  \]
  Expanding $z=u\, e^2 n^{-1/2}$ and again bounding $a! \le e a^{1/2} (a  / e)^a$, the left-hand side is at most
  \[
    \sum_{a=3 d^2 + 1}^n e^2 a \binom{a+d^2-1}{a}^{3/2} \big(u / d^{1/2}\big)^{a} (a / n)^{a/2},
  \]
  and dividing by $d^5 z^2 / 4$, it suffices to show that
  \[
    4 e^{-2} d^{-6} \sum_{a=3 d^2 + 1}^n a^2  \binom{a+d^2-1}{a}^{3/2} \big(u / d^{1/2}\big)^{a-2} (a / n)^{a/2-1} \le 1.
  \]
  Since $n > 10^4$, we may in turn bound this by the infinite sum
  \begin{equation}
    \label{eq:thetad-stp-2}
    4 e^{-2} d^{-6} \sum_{a=3 d^2 + 1}^\infty a^2  \binom{a+d^2-1}{a}^{3/2} \min\big(0.9, 2 / d^{1/2}\big)^{a-2} \min(a / 10^4, 1)^{a/2-1}.
  \end{equation}
  This sum is convergent (as the exponential saving $0.9^{a}$ dominates), and depends only on $d$. To complete the proof, we claim that~\eqref{eq:thetad-stp-2} is bounded by $1$ for all $d \ge 2$.

For $2 \le d \le 38$, it is routine to compute the sum~\eqref{eq:thetad-stp-2} to sufficient precision to verify that bound is indeed satisfied, so assume $d \ge 39$. Then $\min(0.9, 2/d^{1/2}) = 2 / d^{1/2}$. We may ignore the other min factor. Hence it suffices to bound
  \begin{equation}
    \label{eq:thetad-stp-3}
    4 e^{-2} d^{-6} \sum_{a=3 d^2 + 1}^\infty a^2  \binom{a+d^2-1}{a}^{3/2} (2 / d^{1/2})^{a-2}.
  \end{equation}
  For $a \ge 3 d^2$, for any $x \in (0, 1)$ we may bound
  \[
    \binom{a+d^2-1}{a} \le x^{-a} (1-x)^{-(d^2-1)} \le x^{-a} (1-x)^{-a/3+1} \leq (x^{-1} (1-x)^{-1/3})^{a},
  \]
  and setting $x=3/4$ gives an upper bound of $2.1166^a$. Thus~\eqref{eq:thetad-stp-3} is bounded by
  \[
    e^{-2} d^{-5} \sum_{a=3 d^2 + 1}^\infty a^2  \big(6.16 / d^{1/2}\big)^{a}
  \]
  and since
  \[
    \sum_{r=1}^\infty r^2 y^r \le \sum_{r=1}^\infty r(r+1) y^r = \frac{2 y}{(1-y)^3} \le \frac2{(1-y)^3}
  \]
  for $0 \le y < 1$, when $d \ge 39$ this is bounded by
  \[
    e^{-2} 39^{-5} \frac{2}{(1 - 6.16 / 39^{1/2})^3} < 0.002  < 1.
  \]
  This completes the (very technical, for which we apologize) proof.
\end{proof}

Let $R_m$ be the set of all $\rho$ that have some one-dimensional factor of multiplicity at least $n-m$. Then by the discussion preceding the lemma we have
\[
  \sum_{\rho \in R_m^c} \|\hatS(\rho)\|_\HS^3 \dim \rho \leq \prod_{j=1}^k \theta_{d_j}(z) \cdot \frac1{z^n} \frac{n!}{n^{3n/2}}
\]
for all $z > 0$. Set $z = w e^2 n^{-1/2}$ for some $w \in [1, 2]$.  Suppose the hypotheses of Lemma~\ref{lem:theta-d} are satisfied for this $z$ (and the given values $n$, $m$ and $d_1,\dots,d_k$); in particular, for (i) it is sufficient that $n \geq (2e^2 / 0.15)^2$ (i.e., $n \geq 9707$), and
\[
  w \le (1 - m/n)^{-1/2} \big(1 - 0.66 e^4 w^2 / n \log (e^{-2} w^{-1} n^{1/2})\big)
\]
and for (ii) we require $n > 10^4$ and $w \le 0.9 d_j^{1/2}$ for each $d_j \ne 1$.
Then we conclude
\begin{align*}
  \sum_{\rho \in R_m^c} \|\hatS(\rho)\|_\HS^3 \dim \rho
  &\leq \exp\left( \sum_{j=1}^k d_j^{5/2} z + (1/10) z^3 k \right) \pfrac{n!}{z^n n^{3n/2}} \\
  &\leq \exp\left( w e^2 \sum_{j=1}^k d_j^{5/2} n^{-1/2} - n \log w + (1/10) z^3 k \right) \pfrac{n!}{e^{2n} n^n}.
\end{align*}
Since
\[
  \sum_{j=1}^k d_j^{5/2} \leq \max_j d_j^{1/2} \sum_{j=1}^k d_j^2 \leq n^{5/4},
\]
we therefore have
\begin{equation}
  \label{eq:orbit-bound}
  \sum_{\rho \in R_m^c} \|\hatS(\rho)\|_\HS^3 \dim \rho
  \leq \exp( w e^2 n^{3/4} - n \log w + (1/10) z^3 n) \frac{n!}{e^{2n} n^n}.
\end{equation}

\begin{proposition}%
  \label{prop:dense-minor-arcs}
For some constants $C, c>0$, the following holds for sufficiently large $n$. If $m\geq C n^{3/4}$, then
\[
  \sum_{\rho \in R_m^c} \|\hatS(\rho)\|_\HS^3 \dim \rho \leq e^{-cm} \pfrac{n!}{n^n}^3.
\]
\end{proposition}
\begin{proof}
  We may apply~\eqref{eq:orbit-bound} with $w = e^{m/5n}$.  For $n$ and $m$ sufficiently large it is clear that the hypotheses above are satisfied, and we have
\begin{align*}
  w e^2 n^{3/4} - n \log w + (1/10) (e^2 w / n^{1/2})^3 n
  &\leq O(n^{3/4}) - m/5 + o(1).
\end{align*}
As long as $m \geq Cn^{3/4}$ for a sufficiently large constant $C$, and $n$ is sufficiently large, this is bounded by $-cm$ for some $c>0$.
\end{proof}

\begin{proposition}%
  \label{prop:non-asymptotic-dense-minor-arcs-big}
Suppose $m \geq 0.06 n$ and that $n > 10^{10}$. Then
\[
  \sum_{\rho \in R_m^c} \|\hatS(\rho)\|_\HS^3 \dim \rho
  \leq e^{-0.005n} \pfrac{n!}{n^n}^3.
\]
\end{proposition}
\begin{proof}
  We may apply~\eqref{eq:orbit-bound} with $w = 1.03$; it is easy to check that the hypotheses are satisfied.  Then
\[
  w e^2 n^{3/4} - n \log w + (1/10) (e^2 w / n^{1/2})^3 n
  \leq -0.005n
\]
  for $n > 10^{10}$.
  Finally, note $n!/(e^{2n} n^n) < (n!/n^n)^3$.
\end{proof}

\begin{proposition}%
  \label{prop:non-asymptotic-dense-minor-arcs}
Suppose $m \geq 0.71 n$, that $n > 3\times 10^5$, and that $G$ has no irreducible representations of dimension $d \in [2, 4]$. Then
\[
  \sum_{\rho \in R_m^c} \|\hatS(\rho)\|_\HS^3 \dim \rho
  \leq e^{-0.03n} \pfrac{n!}{n^n}^3.
\]
\end{proposition}
\begin{proof}
  We may apply~\eqref{eq:orbit-bound} with $w = 1.85$; again it is straightforward to check that the hypotheses hold for $n > 3 \times 10^5$, and
\[
  w e^2 n^{3/4} - n \log w +(1/10) (e^2 w / n^{1/2})^3 n
  \leq -0.03n.
\]
Now again use $n!/(e^{2n} n^n) < (n!/n^n)^3$.
\end{proof}

\section{Proof of the Hall--Paige conjecture}%
\label{sec:quant-proof}

We have now completed the proof of Theorem~\ref{main-theorem-hp-version}. In particular, the Hall--Paige conjecture holds for every sufficiently large group. In this section we explain some quantitative improvements in the special case that $G$ has no low-dimensional representations, and use these to give a complete proof of the Hall--Paige conjecture, apart from a few explicit cases that we list.

Let $d(G)$ denote the minimal degree of a nontrivial complex representation of $G$. The quality of our bounds depends on $d(G)$: for every $d_0 \leq 20$, say, we can compute some $n_0$ such that any counterexample $G$ would have to satisfy either $d(G) \leq d_0$ or $|G| \leq n_0$, with larger values of $d_0$ leading to smaller values of $n_0$. The choices $d_0 \in \{3, 11, 20\}$ are sufficiently representative for our needs.

\begin{theorem}%
  \label{thm:d13}
Suppose $G$ is a counterexample to the Hall--Paige conjecture.
\begin{enumerate}[label=(\roman*)]
  \item Either $d(G) \leq 3$ or $|G| \leq 10^{10}$.
  \item Either $d(G) \leq 11$ or $|G| \leq 3 \times 10^5$.
  \item Either $d(G) \leq 20$ or $|G| \leq 10^5$.
\end{enumerate}
\end{theorem}

\begin{proof}
Let $G$ be a finite group of order $n > 10^5$ and minimal complex representation degree $d \geq 4$. Recall that $S \subset G^n$ denotes the set of bijections $\{1, \dots, n\} \to G$. We claim that
\[
  1_S * 1_S * 1_S(1) > 0.
\]
By~\eqref{fourier-expression} we have
\[
  1_S * 1_S * 1_S(1) = \sum_\rho \langle \hatS(\rho)^3, \rho(1)\rangle \dim \rho.
\]
Let $C_m$ be the contribution to this sum from $m$-sparse $\rho$, and let
\[
  M_m = C_0 + C_1 + \cdots + C_m.
\]
By Proposition~\ref{major-arcs}, we have
\[
  \left|M_{20} \pfrac{n!}{n^n}^{-3} - \frakS_{20}\right| < 0.32,
\]
where
\[
  \frakS_{20} = \sum_{k \leq 10} \frac1{k!} \left( -\frac{n-1}{2n} \right)^k > 0.6.
\]
Thus
\[
  M_{20} > 0.28 \pfrac{n!}{n^n}^3.
\]
Meanwhile, by Proposition~\ref{prop:explicit-low-entropy-minor-arcs},
\[
  \sum_{m=21}^{0.06n} |C_m|
  < 0.12 \pfrac{n!}{n^n}^3.
\]
Thus we need only show
\begin{equation} \label{eq:Cm-large-m-bound}
  \sum_{m \geq 0.06n} |C_m| < 0.16 \pfrac{n!}{n^n}^3.
\end{equation}

If $n > 10^{10}$,~\eqref{eq:Cm-large-m-bound} is immediate from Proposition~\ref{prop:non-asymptotic-dense-minor-arcs-big}.

If $n > 3 \times 10^5$ and $d \ge 5$, from Proposition~\ref{prop:non-asymptotic-dense-minor-arcs},
\[
  \sum_{m \geq 0.71n} |C_m| < e^{-0.03 n} \pfrac{n!}{n^n}^3.
\]
Hence we need only worry about the intermediate range $m/n \in [0.06, 0.71]$. It turns out that we can eliminate this range using Lemma~\ref{sparseval} alone, assuming $d \geq 12$. By Lemma~\ref{sparseval} we have
\[
  \sum_{m\textup{-sparse}~\rho} \|\hatS(\rho)\|_\HS^2 \dim \rho \leq (1 - (m/n)^{1/2})^{-1} e^{s(m/n) n} \pfrac{n!}{n^n}^2.
\]
Note that $\dim \rho \geq d^m$. Moreover, every $m$-sparse $\rho$ has a permutation orbit of size at least $\binom{n}{m}$. Thus
\[
  \binom{n}{m} \|\hatS(\rho)\|_\HS^2 d^m \leq (1 - (m/n)^{1/2})^{-1} e^{s(m/n) n} \pfrac{n!}{n^n}^2.
\]
Thus
\begin{align*}
  &\sum_{m\textup{-sparse}~\rho} \|\hatS(\rho)\|_\HS^3 \dim \rho \\
  &\qqquad \leq \binom{n}{m}^{-1/2} d^{-m/2} (1 - (m/n)^{1/2})^{-3/2} e^{(3/2) s(m/n) n} \pfrac{n!}{n^n}^3 \\
  &\qqquad\leq (1 - (m/n)^{1/2})^{-3/2} e^{1/2} m^{1/4} e^{f_d(m/n) n / 2} \pfrac{n!}{n^n}^3,
\end{align*}
where
\[
  f_d(t) = 3s(t) - h(t) - t \log d,
\]
for $h(t) = t \log (1/t) + (1-t) \log (1/(1-t))$ as in Section~\ref{sec:low-entropy-minor-arcs}. The function $f_{12}$ has roots near $.0424\dots$ and $0.7172\dots$, and
\[
  \max_{t \in [0.06, 0.71]} f_{12}(t) \leq -0.005.
\]
Hence
\[
  \sum_{\substack{m\textup{-sparse}~\rho \\ m/n \in [0.06, 0.71]}} \|\hatS(\rho)\|_\HS^3 \dim \rho
  \leq 100 n^{1/4} e^{-0.0025 n} \pfrac{n!}{n^n}^3 \leq 10^{-10} \pfrac{n!}{n^n}^3.
\]
The required bound~\eqref{eq:Cm-large-m-bound} follows.

Finally, assume $n > 10^5$ and $d \geq 21$. As above we have
\[
  \sum_{m\text{-sparse}} \|\hatS(\rho)\|_\HS^3 \dim \rho \leq (1 - (m/n)^{1/2})^{-3/2} e^{1/2}\, m^{1/4} e^{f_d(m/n) n/2} \pfrac{n!}{n^n}^3.
\]
The function $f_{21}$ is uniformly negative on $[0.06, 1]$, and
\[
  \max_{t \in [0.06, 1]} f_{21}(t) = f_{21}(0.06) < -0.03.
\]
Hence
\begin{align*}
  \sum_{\substack{m\textup{-sparse}~\rho \\ m/n \in [0.06, 1)}} \|\hatS(\rho)\|_\HS^3 \dim \rho 
  &\leq (1 - (1-1/n)^{1/2})^{-3/2} e^{1/2} n^{1/4} e^{-0.015 n} \pfrac{n!}{n^n}^3\\
  &\leq 10^{-10} \pfrac{n!}{n^n}^3.
\end{align*}
The endpoint $m = n$ can be handled almost identically, replacing Lemma~\ref{sparseval} with Parseval's identity
\[
  \sum_\rho \|\hatS(\rho)\|_\HS^2 \dim \rho = \frac{n!}{n^n}.
\]
This completes the proof.
\end{proof}

\begin{corollary}
If $G$ is a nonabelian simple counterexample to the Hall--Paige conjecture, then $G$ is either $A_n$ $(5 \leq n \leq 12)$, $\PSL_2(q)$ $(7\leq q \leq 53)$, or one of the groups listed in Table~\ref{table:possible-counterexamples}.
\end{corollary}
\begin{proof}
For $G = A_n$ we have $d(G) = n-1$ for $n\geq 6$ and $|G| = n!/2$, so we must have $n \leq 12$. For $G = \PSL_2(q)$ we have $d(G) = q-1$ if $q$ is even, $(q+1)/2$ if $q \equiv 1 \pmod 4$, $(q-1)/2$ if $q \equiv 3 \pmod 4$, and $|G| = (q^3 - q)/(2, q-1)$, so we must have $q \leq 53$.

Let $\tilde d(G)$ denote the minimal degree of a nontrivial complex \emph{projective} representation of $G$ (equivalently, representation of a central extension).
The value of $\tilde d(G)$ is given for all classical groups by Tiep and Zalesskii~\cite{tiep--zalesskii} and for exceptional groups by L\"ubeck~\cite{lubeck}. Minimal degrees for sporadic groups are listed in Jansen~\cite{jansen}. (See also Hiss--Malle~\cite{hiss--malle} for a list of low-dimensional representations.)

Using $d(G) \geq \tilde d(G)$, the value of $\tilde d(G)$ given in the cited literatue can be used to eliminate all simple groups except those in Table~\ref{table:possible-counterexamples} and additionally $\PSU_4(3)$ and $\Omega_8^+(2)$. The exact values of $d(G)$ and $|G|$ for all these groups can be computed in GAP.%
\end{proof}

\begin{table}
\centering
\begin{tabular}{lrrr}
$G$ & $d(G)$ & $|G|$ \\ \hline 
$\PSL_3(3)$ & $12$ & $5616$ \\
$\PSU_3(3)$ & $6$ & $6048$ \\
$M_{11}$ & $10$ & $7920$ \\
$\PSL_3(4)$ & $20$ & $20 160$ \\
$\PSU_4(2)$ & $5$ & $25 920$ \\
$\Sz(8)$ & $14$ & $29 120$ \\
$\PSU_3(4)$ & $12$ & $62 400$ \\
$M_{12}$ & $11$ & $95 040$ \\
$\PSU_3(5)$ & $20$ & $126 000$ \\
$\PSp_6(2)$ & $7$ & $1 451 520$ \\
$\PSU_5(2)$ & $10$ & $13 685 760$
\end{tabular}
\caption{Nonabelian simple groups not ruled out by Theorem~\ref{thm:d13}, apart from $A_n$ ($5 \leq n \leq 12$) and $\PSL_2(q)$ ($7\leq q \leq 53$)}%
\label{table:possible-counterexamples}
\end{table}

By Wilcox~\cite[Theorem~12]{wilcox}, a minimal counterexample $G$ to the Hall--Paige conjecture would have to be simple (Wilcox quotes the Feit--Thompson Theorem to prove this, but it is not necessary: see Section~\ref{avoiding-FT}). Cyclic and alternating groups were known to Hall and Paige to satisfy their conjecture~\cite{hp}. For case-specific constructions for $\PSL_2(q)$ and Mathieu groups, as well as most of the other groups listed in Table~\ref{table:possible-counterexamples}, see Evans's book~\cite{evans-book} (the only exceptions seem to be $\PSL_3(3)$, $\PSU_3(3)$, and $\PSU_3(5)$; see~\cite[Theorem~7.17]{evans-book}). As we have mentioned, Wilcox gave a unified proof for groups of Lie type. The main new contribution of this section therefore is a uniform proof for sporadic groups other than $M_{11}$ and $M_{12}$. In fact we do not need the full strength of the classification of finite simple groups: we need only a classification of the finite simple groups satisfying the conclusion of Theorem~\ref{thm:d13}.

\subsection{Further numerical improvements}

As noted in the introduction, with further computational effort, the authors believe it is possible to extend the range of these arguments to include some, but not all, of the groups in Table~\ref{table:possible-counterexamples}, without introducing any genuinely new ideas.  For example, $M_{12}$ and $\PSU_3(5)$ should definitely be tractable, $M_{11}$ does not appear to be, and the tipping point (in terms of the sizes of $|G|$ and $d(G)$) is somewhere in between.

We have not attempted to put these numerical calculations into the form of a proof.  Instead, for reference, we briefly outline the various tweaks that we believe allow these improvements.  The general rule is that whenever something may be computed efficiently and exactly rather than bounded, do so.
\begin{itemize}
  \item Throughout, we may use an explicit list of dimensions $d_1,\dots,d_k$ of the irreducible representations of $G$, rather than generalities.
  \item In two notable places we make explicit and not necessarily optimal choices of tunable parameters: the values $R$ in~\eqref{eq:cauchy-Mfm0-estimate} and $w$ in~\eqref{eq:orbit-bound}.  In both cases, we are free to search for closer-to-optimal values.
  \item The values $\theta_d(z)$, which we bound in Lemma~\ref{lem:theta-d}, may be computed directly from their definition.  Similarly, the functions $\alpha_m(t)$ considered in Lemma~\ref{lem:stat-phys} may be computed exactly using that their coefficients are associated Stirling numbers of the second kind.  This in turn allows improved estimates of $\beta_m(t)$ in the proof of Lemma~\ref{lem:beta-bound}.
  \item The value $\sum_{m\textup{-sparse}~\rho} \|\hatS(\rho)\|_\HS^2 \dim \rho$, estimated in Lemma~\ref{sparseval}, may be computed exactly using the recurrence in the proof of~\cite[Theorem~5.1]{EMM} (although the need for high or exact numerical precision makes this expensive for large $m$).
\end{itemize}

\subsection{Reducing to simple groups without Feit--Thompson}%
\label{avoiding-FT}

Wilcox quotes the Feit--Thompson theorem as part of his reduction of the Hall--Paige conjecture to the case of simple groups, but it is not needed, as we explain now.

\def\HP{\texttt{HP}}
Let $\HP$ denote the Hall--Paige condition, i.e., $G$ satisfies $\HP$ if and only if its Sylow 2-subgroups are trivial or noncyclic.
The following is a special case of Burnside's transfer theorem.

\begin{lemma}%
    \label{lem:burnside}
Let $G$ be a finite group not satisfying $\HP$.
Then $G$ has a unique subgroup $K$ such that $|K|$ is odd and $[G:K]$ is a power of 2.
\end{lemma}

\begin{proposition}
    Suppose $G$ is not simple and satisfies $\HP$.
    Then $G$ has a nontrivial proper normal subgroup $N$ such that each of $N$ and $G/N$ either satisfies $\HP$ or has order 2.
\end{proposition}
\begin{proof}
    Let $N$ be a nontrivial proper normal subgroup of $G$.
    If $N$ and $G/N$ both satisfy $\HP$ we are done.
    
    Suppose $N$ does not satisfy $\HP$.
    Then by the lemma $N$ has a unique subgroup $K$ such that $|K|$ is odd and $[N:K]$ is a power of 2.
    But then $K$ is normal in $G$ and $G/K$ has the same Sylow 2-subgroups as $G$, so we are done unless $K$ is trivial, i.e., $N$ is a cyclic 2-group.
    By replacing it with its unique order-2 subgroup, we may assume $|N| = 2$, and thus we are done if $G/N$ satisfies $\HP$.
    
    Hence it suffices to consider the case in which $G/N$ does not satisfy $\HP$.
    In this case the lemma shows that $G/N$ and hence $G$ itself has an index-2 subgroup, so we may assume $[G:N]=2$.
    If $N$ does not satisfy $\HP$ then by repeating the argument of the previous paragraph we are either done or $N$ is a 2-group.
    In the latter case $G$ itself is a noncyclic 2-group, and a short argument completes the proof in this case.
\end{proof}

This proposition together with~\cite[Propositions~3, 7, 11]{wilcox} reduces the Hall--Paige conjecture to the simple case without the use of Feit--Thompson.
(However, the fact that every insoluble group satisfies $\HP$ is plainly equivalent to the Feit--Thompson theorem, in view of Lemma~\ref{lem:burnside} and the insoluble groups $G \times C_2$ with $G$ nonabelian simple.)

\section{The asymptotic expansion}%
\label{sec:asymptotic-expansion}

In this final section we derive an algebraic-combinatorial formula for lower-order terms in the asymptotic in Theorem~\ref{main-theorem-hp-version}, or equivalently Theorem~\ref{main-thm} with $f=1$. This formula enables us in principle to compute the number of complete mappings of an arbitrary finite group $G$ of order $n$, asymptotically as $n\to\infty$, up to a multiplicative error of $1 + O(n^{-m})$ for any fixed $m > 0$.

In Section~\ref{sec:major-arcs} we indexed the main contributions (or major arcs) to $1_S * 1_S * 1_S(f)$ by partition systems $\psystem = (\calP_1, \calP_2, \calP_3)$. In this section, in the special case $f=1$, two partition systems contribute the same if they are the same up to permuting the base set $\{1,\dots,n\}$ (as these permutations now do not affect $f$), so we may aggregate the contributions from each $S_n$-orbit of partition systems, and thus express the asymptotic in terms of partition systems up to isomorphism. The orbit size of a partition system $\psystem$ depends on the size of its automorphism group $\Aut \psystem$, and thus its total contribution carries a factor of $1/|\Aut \psystem|$. Additionally, we get a simplified asymptotic by decomposing an arbitrary partition system into its connected components and applying the exponential formula from enumerative combinatorics (see Wilf~\cite[Chapter~3]{gfology} for background).

When the dust settles we will find that the main term $e^{-1/2}$ comes from the single isomorphism class $\psystem$ given by $\calP_1 = \calP_2 = \calP_3 = \{\{1, 2\}\}$ (with in particular the ``$2$'' in ``$e^{-1/2}$'' coming from $|\Aut \psystem| = 2$), and lower-order terms come from connected partition systems of increasing complexity. The lower-order terms can be computed mechanically, though extremely tediously, and expressed in terms of invariants of the underlying group $G$. We do the calculation explicitly for the $1/n$ term, with Theorem~\ref{main-theorem-hp-inv-version} as the result.

To state the formula we first need to further develop the language from Section~\ref{sec:major-arcs}. We recall here the relevant definitions from Section~\ref{sec:major-arcs} and we add several more.

\begin{definition} We recall the convention (see Section~\ref{subsec:arg-proj-application}) that two partitions on different (possibly overlapping) base sets may be regarded as the same if one can be obtained from the other by repeatedly adding or removing singletons.
  \begin{enumerate}[label=(\roman*)]
    \item A \emph{partition system} on a set $X$ is a triple $\psystem = (\calP_1, \calP_2, \calP_3)$ of partitions of $X$ with the same support, denoted $\supp \psystem$.  (By our convention, we lose nothing by assuming $X = \supp \psystem$.)
    \item The \emph{M\"obius function} is defined on a partition system $\psystem = (\calP_1, \calP_2, \calP_3)$ by
    \[
      \mu(\psystem) = \mu(\calP_1)\, \mu(\calP_2)\, \mu(\calP_3).
    \]
    \item The \emph{complexity} of a partition system $\psystem = (\calP_1, \calP_2, \calP_3)$ is
    \[
      \cx \psystem = \max_{\sigma \in S_3} \left( \rank(\calP_{\sigma(1)}) + \rank(\calP_{\sigma(2)} \vee \calP_{\sigma(3)})\right) - |\supp\psystem|.
    \]
    \item The \emph{gamma function} of a partition system $\psystem = (\calP_1, \calP_2, \calP_3)$ is
    \[
      \gamma_G(\psystem) = n^{|X| + \cx \psystem} P_X(c_{\calP_1} * c_{\calP_2} * c_{\calP_3})(1),
    \]
      where $X = \supp \psystem$.  Note this value depends on the group $G$, not just on $n = |G|$.
    \item An \emph{isomorphism} from a partition system $\psystem = (\calP_1, \calP_2, \calP_3)$ to a partition system $\psystem' = (\calP'_1, \calP'_2, \calP'_3)$ is a bijection $f \colon \supp\psystem \to \supp\psystem'$ that sends $\calP_i$ to $\calP'_i$ for each $i$.
    \item The \emph{automorphism group} $\Aut \psystem$ of a partition system $\psystem$ is the subgroup of $\Sym(\supp\psystem)$ consisting of all isomorphisms $\psystem\to\psystem$.
  \end{enumerate}
\end{definition}

The following lemma will be used to show that if we only care about asymptotics up to a given order $n^{-C}$ then we only need to worry about finitely many partition systems.

\begin{lemma}%
  \label{lem:finitely-many-triples}
For any connected partition system $\psystem$ we have
\[
  |\supp \psystem| \leq 4 \cx \psystem + 2.
\]
In particular, there are only finitely many isomorphism classes of connected partition system of any given complexity, and the only connected partition system of complexity zero up to isomorphism is
\[
  \psystem_0 = \left(
    \{\{1, 2\}\},
    \{\{1, 2\}\},
    \{\{1, 2\}\}
  \right).
\]
\end{lemma}
\begin{proof}
Let $\psystem$ be a connected partition system of support size $m$. Then $\calP_1$, $\calP_2$, $\calP_3$ all have rank at least $m/2$, and $\calP_1 \vee \calP_2 \vee \calP_3$ has rank $m-1$, so by Lemma~\ref{lem:lrank-srank-inequality},
\[
    \trank(\psystem) \geq \lrank(\psystem) \geq (m/2 + m/2 + m/2 + m-1)/2 = 5m/4 - 1/2;
\]
see Section~\ref{subsec:psystem} for the definitions of these terms.
Hence $\cx \psystem \geq m/4 - 1/2$.
\end{proof}

We are now ready to state the main formula. We use a formal device inspired by ``umbral calculus'' (the unfamiliar reader is advised only to consult sources at least as modern as Roman--Rota~\cite{roman--rota}). Let $u$ and $z$ be formal variables, let $m,C \geq 0$ be cut-off parameters, and let\footnote{``In the nineteenth century---and among combinatorialists well into the twentieth---the linear functional $L$ would be called an umbra, a term coined by Sylvester, that great inventor of unsuccessful terminology.''~\cite{roman--rota}} $L = L_{m,C}$ be the linear map defined on $u$-monomials by
\[
  L\, u^k = \begin{cases}
    n^{2k} / (n)_k^2 &: k \leq m, \\
    0 &: k > m,
  \end{cases}
\]
on $z$-monomials by
\[
  L\, z^k = \begin{cases}
    n^{-k} &: k \leq C, \\
    0 &: k > C,
  \end{cases}
\]
and on a general power series in $u$ and $z$ by
\[
  L \sum_{i,j\geq 0} a_{ij} u^i z^j = \sum_{i,j\geq 0} a_{ij} (L u^i) (L z^j) \qquad (a_{ij} \in \C).
\]
For this to make sense we assume $n \geq m$; thus the image of $L$ is a function of $n$ for integers $n \geq m$. The map $L$ simply allows us to express certain sums more compactly, such as
\begin{equation} \label{Lu-example}
  L \exp(-u^2/2) = \sum_{2k \leq m} (-1)^k \frac{n^{4k}}{2^k k! (n)_{2k}^2},
\end{equation}
or
\begin{equation} \label{Luz-example}
  L \exp(uz) = \sum_{k \leq \min(C, m)} \frac{n^{k}}{k!(n)_k^2}.
\end{equation}

\begin{theorem}%
  \label{thm:asymptotic-expansion}
  Let $\cm(G)$ be the number of complete mappings of a finite group $G$ of order $n$ satisfying the Hall--Paige condition, and let $f_G(u,z)$ be the formal power series
  \[
    f_G(u,z) = \sum_{\psystem~\textup{connected}} \frac{\mu(\psystem)}{|\Aut \psystem|} \gamma_G(\psystem) u^{|\supp \psystem|} z^{\cx \psystem}
  \]
  where the sum extends over all connected partition systems $\psystem$ up to isomorphism.  Then 
  for any fixed integer $C \geq 0$ we have
  \[
    \frac{\cm(G)}{|G^\ab| \, n!^2 / n^n}
    = L_{m,C} \exp \big( f_G(u,z) \big)
    + O(n^{-C-1}),
  \]
  where $L_{m,C}$ is as above and $m = (\log n)^2$.
\end{theorem}

\begin{remark}\leavevmode 
  \begin{enumerate}[label=(\roman*)]
\begin{samepage} 
\item Formally, the sum defining $f_G(u,z)$ is not restricted to partition systems $\psystem$ of bounded complexity or bounded support (we even permit $|\supp \psystem| > n$).  However, for the purposes of computing $L_{m,C} \exp(f_G(u,z))$, we may restrict the sum to isomorphism classes of connected partition systems $\psystem$ with $\cx \psystem \le C$ without changing the answer.  By Lemma~\ref{lem:finitely-many-triples}, the restricted sum is finite.
\end{samepage}
  \item The form of the cut-off $m = (\log n)^2$ is not essential; anything growing faster than $\log n$ but slower than $n^{1/2-\eps}$ would also work, with suitable modifications.
\end{enumerate}
\end{remark}

\begin{proof}
  Note that
  \[
    \frac{\cm(G)}{n!^2/n^n} = \frac{1_S * 1_S * 1_S(1)}{(n!/n^n)^3}.
  \]
  We will estimate $1_S * 1_S * 1_S(1)$ using~\eqref{fourier-expression} as usual. By Propositions~\ref{low-entropy-minor} and~\ref{prop:dense-minor-arcs}, we may ignore the contribution from the minor arcs: that is,
  \[
    \frac{1_S * 1_S * 1_S(1)}{|G^\ab| (n!/n^n)^3} = M_m(1/n) + O(e^{-cm}),
  \]
  where, as in Section~\ref{sec:major-arcs},
  \[
    M_m(z) = \sum_{|\supp \psystem| \leq m} \pfrac{n^{|\supp \psystem|}}{(n)_{|\supp \psystem|}}^3 \mu(\psystem) \gamma_G(\psystem) n^{-|\supp\psystem|} z^{\cx \psystem}.
  \]
  It follows from Lemma~\ref{lem:cauchy-trick} (with $f = M_m$, $u=1/n$, $R = c/m^2$, $k=C$) and Proposition~\ref{prop:major-arcs-estimate} that
  \[
    M_m(1/n) = \sum_{\substack{|\supp \psystem| \leq m \\ \cx \psystem \leq C}} \pfrac{n^{|\supp \psystem|}}{(n)_{|\supp \psystem|}}^3 \mu(\psystem) \gamma_G(\psystem) n^{-|\supp\psystem| - \cx \psystem} + O(m^2 / n)^{C+1}.
  \]
  Hence
  \begin{align*}
    \frac{\cm(G)}{|G^\ab| n!^2/n^n}
    &= \sum_{\substack{|\supp \psystem| \leq m \\ \cx \psystem \leq C}} \pfrac{n^{|\supp \psystem|}}{(n)_{|\supp \psystem|}}^3 \mu(\psystem) \gamma_G(\psystem) n^{-|\supp\psystem| - \cx \psystem} \\
    &\qqquad + O(n^{-C-1+o(1)}).
  \end{align*}
  Isomorphic partition systems contribute the same amount to the sum,\footnote{This is where we need the hypothesis $f=1$: otherwise we would need to consider isomorphism types of pairs $(\psystem, f)$.} and each abstract partition system $\psystem$ appears exactly $(n)_{|\supp\psystem|} / |\Aut \psystem|$ times in the sum, so
  \begin{align*}
    \frac{\cm(G)}{|G^\ab| n!^2/n^n}
    &= \sum_{\substack{|\supp \psystem| \leq m \\ \cx \psystem \leq C \\ \text{(up to isomorphism)}}} \pfrac{n^{|\supp \psystem|}}{(n)_{|\supp \psystem|}}^2 \frac{\mu(\psystem)}{|\Aut \psystem|} \gamma_G(\psystem) n^{-\cx \psystem} \\
    &\qqquad + O(n^{-C-1+o(1)}).
  \end{align*}
  Using $L=L_{m,C}$, $u$, and $z$, this can be written (dropping the ``up to isomorphism'' warning from now on)
  \begin{equation} \label{disconnected-psystem-sum}
    \frac{\cm(G)}{|G^\ab| n!^2/n^n}
    = L \sum_\psystem \frac{\mu(\psystem)}{|\Aut \psystem|} \gamma_G(\psystem) u^{|\supp\psystem|} z^{\cx \psystem}  + O(n^{-C-1+o(1)}).
  \end{equation}

  Our next move is to relate the sum in~\eqref{disconnected-psystem-sum} over all partition systems to a sum just over connected partition systems. To do this, we need to show that each of the factors appearing in~\eqref{disconnected-psystem-sum} is ``multiplicative'' with respect to connected components. Consider an arbitrary partition system $\psystem$. By decomposing $\psystem$ into its connected components we have
  \[
    \psystem = \psystem_1^{e_1} \cup \cdots \cup \psystem_k^{e_k}.
  \]
  Here $\psystem_1, \dots, \psystem_k$ are distinct connected partition systems, and $e_1, \dots, e_k$ are multiplicities, and $(\psystem_1, e_1), \dots, (\psystem_k, e_k)$ are uniquely determined up to order by $\psystem$, and conversely. In particular
  \[
    |\supp \psystem| = e_1 |\supp \psystem_1| + \cdots + e_k |\supp \psystem_k|,
  \]
  and
  \[
    \cx\psystem = e_1 \cx \psystem_1 + \cdots + e_k \cx \psystem_k.
  \]
  It is trivial that\footnote{The multiplicativity of $u^{|\supp\psystem|}$ is the reason for introducing $L$ and $u$: certainly $(n^{|\supp\psystem|}/(n)_{|\supp\psystem|})^2$ is not multiplicative.}
  \begin{align*}
    \mu(\psystem) &= \mu(\psystem_1)^{e_1} \cdots \mu(\psystem_k)^{e_k},\\
    u^{|\supp \psystem|} &= (u^{|\supp \psystem_1|})^{e_1} \cdots (u^{|\supp \psystem_k|})^{e_k},\\
    z^{\cx \psystem} &= (z^{\cx \psystem_1})^{e_1} \cdots (z^{\cx\psystem_k})^{e_k}.
  \end{align*}
  It is not hard to see that
  \[
    \Aut\psystem \cong (\Aut \psystem_1 \wr S_{e_1}) \times \cdots \times (\Aut \psystem_k \wr S_{e_k}),
  \]
  and in particular
  \[
    |\Aut \psystem| = |\Aut \psystem_1| e_1! \cdots |\Aut\psystem_k| e_k!.
  \]
  Finally, the (not quite so obvious) identity
  \begin{equation}\label{fp-mult}
    \gamma_G(\psystem) = \gamma_G(\psystem_1)^{e_1} \cdots \gamma_G(\psystem_k)^{e_k}
  \end{equation}
  follows from repeated application of Lemma~\ref{lem:gamma-mult}.

  Hence, from~\eqref{disconnected-psystem-sum},
  \begin{align*}
    \frac{\cm(G)}{|G^\ab| n!^2/n^n}
    &= L\left[ \sum_{\psystem = \psystem_1^{e_1} \cup \cdots \cup \psystem_k^{e_k}} \prod_{i=1}^k \frac{\mu(\psystem_i)^{e_i}}{e_i! |\Aut \psystem_i|^{e_i}} \gamma_G(\psystem_i)^{e_i} u^{|\supp \psystem_i| e_i} z^{e_i \cx \psystem_i} \right] \\
    &\qqquad + O(n^{-C-1+o(1)}) \\
    &= L \left[ \exp\left(
    \sum_{\psystem ~ \textup{connected}} \frac{\mu(\psystem)}{|\Aut \psystem|} \gamma_G(\psystem) u^{|\supp \psystem|} z^{\cx \psystem}
    \right)\right] \\
    &\qqquad + O(n^{-C-1+o(1)}).
  \end{align*}
  This proves the theorem with an error of the slightly poorer quality $O(n^{-C-1+o(1)})$ in place of $O(n^{-C-1})$.

  Finally, we argue that the error self-improves to the sharper form $O(n^{-C-1})$.
  To see this, we apply the bound above for $C+1$, giving an acceptable error term $O(n^{-C-2+o(1)})$, and show that the contribution from terms $z^{C+1}$ is $O(n^{-C-1})$.  I.e., it suffices to show that
  \[
    L \big( \big[z^{C+1}\big] \exp(f_G(u,z)) \big) = O_C(1),
  \]
  where by $[z^{C+1}] F(z, u)$ we mean the coefficient of $z^{C+1}$ in $F$ as an element of $\C[[u]][[z]]$, which is an element of $\C[[u]]$. Placing absolute value signs everywhere, it suffices to show that
  \[
    L \left( \big[z^{C+1}\big] \exp\left(\sum_{\substack{\psystem~\textup{connected} \\ \cx \psystem \le C+1}} \frac{|\mu(\psystem)|}{|\Aut \psystem|} |\gamma_G(\psystem)|\, u^{|\supp \psystem|} z^{\cx \psystem} \right)\right) = O_C(1)
  \]
  since the left-hand side is an upper bound for the previous quantity.  As all the coefficients of this power series in $u,z$ are now nonnegative, using the bound
  \[
    L (u^k) = n^{2k}/(n)_k^2 \leq e^{O(k^2/n)} = O(1)
  \]
  when $0 \le k \le m$ (and vacuously $L(u^k) = O(1)$ for $k >m$), in turn it suffices to show that
  \[
    \big[z^{C+1}\big] \exp\left(\sum_{\substack{\psystem~\textup{connected} \\ \cx \psystem \le C+1}} \frac{|\mu(\psystem)|}{|\Aut \psystem|} |\gamma_G(\psystem)|\, z^{\cx \psystem} \right) = O_C(1).
  \]
  However, this last fact is clear, as the power series inside the exponential has coefficients $O_C(1)$ (by Lemma~\ref{lem:finitely-many-triples} and Proposition~\ref{prop:gamma-bound}), and this property is preserved after taking the exponential.
\end{proof}

There is one further operation we can apply to partition systems $\psystem = (\calP_1, \calP_2, \calP_3)$ to reduce the number of possibilities we need to consider: we can \emph{reorder} the constituent factors $\calP_1, \calP_2, \calP_3$. Clearly $\mu(\psystem)$, $\Aut \psystem$, $|\supp \psystem|$, and $\cx \psystem$ are invariant under reordering. Less obviously, $\gamma_G(\psystem)$ is also invariant. It suffices to observe that
\[
  c_{\calQ_1} * c_{\calQ_2} * c_{\calQ_3}(1)
\]
is invariant under permutation of indices for any triple of partitons $(\mathcal{Q}_1, \mathcal{Q}_2, \mathcal{Q}_3)$. Up to normalization this quantity is just
\[
  \P(h_1h_2h_3 = 1),
\]
where $h_i$ is a random $\calQ_i$-measurable function. Now note that
\[
  h_1 h_2 h_3 = 1 \iff h_2 h_3 h_1 = 1,
\]
and
\[
  h_1 h_2 h_3 = 1 \iff h_3^{-1} h_2^{-1} h_1^{-1} = 1,
\]
and all permutations of the indices are generated in this way.

\begin{table}
  \centering
  \input{psystems.tex}
  \caption{Connected partition systems $\psystem=(\calP_1, \calP_2, \calP_3)$ of support at most $4$, up to isomorphism and reordering}%
  \label{table:psystems}
\end{table}

Table~\ref{table:psystems} lists all connected partition systems of support size $m \leq 4$ up to isomorphism and reordering. These include all partition systems of complexity $\cx\psystem \leq 1$, as verified by a direct computer check of systems of support size $m\leq 6$ (which is enough by Lemma~\ref{lem:finitely-many-triples}). By Theorem~\ref{thm:asymptotic-expansion}, to understand the asymptotic number of complete mappings up to order $n^{-C-1}$, we need only consider the connected partition systems $\psystem$ of complexity $\cx\psystem \leq C$.

\begin{corollary}%
  \label{cor:inv-asymptotic}
As $n\to\infty$ we have
\[
  \frac{\cm(G)}{|G^\ab| \, n!^2/n^n} = e^{-1/2}\left(1 + (1/3 + \inv/4) n^{-1} + O(n^{-2})\right),
\]
where $\inv$ is the proportion of involutions in $G$.
\end{corollary}

\begin{proof}
  We consider all connected partition systems $\psystem$ listed in Table~\ref{table:psystems} of complexity $\cx\psystem \leq 1$. These are listed again in Table~\ref{table:psystems-cx1} together with the relevant data.
  We now justify the listed estimates of $\gamma_G(\psystem)$.
  For $\psystem_0$, see Lemma~\ref{lem:gamma-pairing}.
 For each of the systems of complexity 1, we check that,
 in the notation of the proof of Proposition~\ref{prop:gamma-bound},
 $t(\psystem, Y) \geq 1$ for each $Y \subsetneq \supp \psystem$; it follows that
\[
  \gamma_G(\psystem) = n^{|\supp\psystem|+1} c_{\calP_1} * c_{\calP_2} * c_{\calP_3}(1) + O(1/n).
\]
Furthermore,
\[
  c_{\calP_1} * c_{\calP_2} * c_{\calP_3}(1) = n^{-\rank(\calP_1) - \rank(\calP_2)} \P(h_1h_2~\text{is}~\calP_3\text{-measurable}),
\]
where $h_i$ is a random $\calP_i$-measurable function (and similarly for other permutations of the indices).

\begin{table}
  \centering
  \input{psystems-cx1.tex}
  \caption{Partition systems $\psystem$ from Table~\ref{table:psystems} with $\cx\psystem \leq 1$: support size, M\"obius value, automorphism group size, number of nonisomorphic reorderings, and gamma function}%
  \label{table:psystems-cx1}
\end{table}

The one interesting case is the ``Klein pairing'' $\psystem_2$, defined by
\begin{align*}
  \calP_1 &= \{\{1, 2\}, \{3, 4\}\}, \\
  \calP_2 &= \{\{1, 3\}, \{2, 4\}\}, \\
  \calP_3 &= \{\{1, 4\}, \{2, 3\}\}.
\end{align*}
In this case if we represent
\begin{align*}
  h_1 &= (x_1, x_1, x_2, x_2), \\
  h_2 &= (y_1, y_2, y_1, y_2),
\end{align*}
then
\[
  h_1 h_2 = (x_1 y_1, x_1 y_2, x_2 y_1, x_2 y_2),
\]
and this is $\calP_3$-measurable if and only if
\begin{align*}
  x_1 y_1 &= x_2 y_2, \\
  x_1 y_2 &= x_2 y_1,
\end{align*}
or equivalently
\begin{align*}
  x_1 &= x_2 z, \\
  y_1 &= z y_2
\end{align*}
for some involution $z$. Thus
\[
  \P(h_1h_2~\text{is}~\calP_3\text{-measurable}) = \inv / n.
\]

Thus the contributions to the sum
\[
  \sum_{\psystem ~ \textup{connected}} \frac{\mu(\psystem)}{|\Aut\psystem|} \gamma_G(\psystem) u^{|\supp\psystem|} z^{\cx \psystem}
\]
are
\begin{align*}
  &\frac{\mu(\psystem_0)}{|\Aut\psystem_0|} \gamma_G(\psystem_0) u^{|\supp\psystem_0|} z^{\cx \psystem_0}
  &&= \frac{-1}{2} (1 - 1/n) u^2,\\
  &\frac{\mu(\psystem_1)}{|\Aut\psystem_1|} \gamma_G(\psystem_1) u^{|\supp\psystem_1|} z^{\cx \psystem_1}
  &&= \frac{8}{6} (1 + O(n^{-1})) u^3 z,\\
  &\frac{\mu(\psystem_2)}{|\Aut\psystem_2|} \gamma_G(\psystem_2) u^{|\supp\psystem_2|} z^{\cx \psystem_2}
  &&= \frac{1}{4} (\inv + O(n^{-1})) u^4 z,\\
  3 \times &\frac{\mu(\psystem_3)}{|\Aut\psystem_3|} \gamma_G(\psystem_3) u^{|\supp\psystem_3|} z^{\cx \psystem_3}
  &&= 3\times \frac{1}{4} (1 + O(n^{-1})) u^4 z,\\
  3\times &\frac{\mu(\psystem_4)}{|\Aut\psystem_4|} \gamma_G(\psystem_4) u^{|\supp\psystem_4|} z^{\cx \psystem_4}
  &&= 3\times \frac{-6}{8} (1 + O(n^{-1})) u^4 z,\\
\end{align*}
so
\begin{multline*}
  \sum_{\psystem~\textup{connected}} \frac{\mu(\psystem)}{|\Aut\psystem|} \gamma_G(\psystem) u^{|\supp\psystem|} z^{\cx \psystem}
  =
  - \frac12(1-1/n) u^2 \\
  + \left(\frac43 u^3 - \frac32 u^4 + \frac14 \inv u^4 + O((u^3+u^4)/n) \right) z
  + O(z^2).
\end{multline*}
Hence by Theorem~\ref{thm:asymptotic-expansion} with $C=1$ we have
\[
  \frac{\cm(G)}{|G^\ab| \, n!^2 / n^n}
  = L\left[ e^{-\frac12 (1-1/n) u^2} \left(1 + \left( \frac43 u^3 - \frac32 u^4 + \frac14 \inv u^4\right) z \right) \right] + O(n^{-2}).
\]
Noting that $n^{2k} / (n)_k^2 = 1 + 2\binom{k}{2} / n + O(k^3/n^2)$ when $k < \sqrt{n} / 10$, it is routine to check that
\begin{align*}
  L \left[e^{-\frac12 (1-1/n) u^2}\right] &= \sum_{k=0}^m \frac{(-1/2)^k}{k!} \Big(1-k/n + 2k (2k-1) /n + O(k^3/n^2) \Big) \\
  &= e^{-1/2} + e^{-1/2} \frac12 n^{-1} + O(n^{-2}),
\end{align*}
and similarly that
\[
  L \left[e^{-\frac12 (1-1/n) u^2} p(u) z \right] = e^{-1/2} p(1) / n + O(n^{-2})
\]
for any fixed polynomial $p$. Hence
\[
  \frac{\cm(G)}{|G^\ab| \, n!^2 / n^n}
  = e^{-1/2} \left(
  1 + \left(\frac13 + \frac14 \inv\right) n^{-1} + O(n^{-2})
  \right),
\]
as claimed.
\end{proof}

\begin{corollary}
If $n = 2^k$ is sufficiently large then $C_2^k$ has more complete mappings than any other group of order $n$.
\end{corollary}
\begin{proof}
If $G = C_2^k$ then $|G^\ab| = n$ and $\inv = 1$, so by the previous corollary the number of complete mappings in $G$ satisfies
\[
  \frac{\cm(G)}{n!^2/n^n} = n e^{-1/2}\left(1 + (1/3 + 1/4) n^{-1} + O(n^{-2})\right).
\]
On the other hand if $|G|=n$ and $G \not\cong C_2^k$ then either $|G^\ab| \leq n/2$ or $\inv \leq 1/2$, so
\[
  \frac{\cm(G)}{n!^2/n^n} \leq n e^{-1/2}\left(1 + (1/3 + 1/8)n^{-1} + O(n^{-2})\right).
\]
Thus if $n$ is sufficiently large we have $\cm(G) < \cm(C_2^k)$.
\end{proof}

More generally, let $n$ be any positive integer, and let $2^k$ be the $2$-part of $n$. By the asymptotic in Corollary~\ref{cor:inv-asymptotic}, if $n$ is sufficiently large then any group $G$ of order $n$ that maximizes $\cm(G)$ must be abelian, and if $k$ is sufficiently large then the Sylow $2$-subgroup of $G$ must be elementary abelian. (Note that, if $G$ is abelian, either $\inv = 2^k/n$ or $\inv \leq 2^{k-1}/n$.) To say more about $G$ we would need to compute more terms in the expansion.

We can also say something about groups $G$, satisfying the Hall--Paige condition, that minimize $\cm(G)$. The abelianization $|G^\ab|$ must be as small possible, so in particular if there is a perfect group of order $n$ then $G$ must be perfect. For example, if $n = p(p-1)(p+1)$ for some prime $p > 3$, then the only perfect group of order $n$ is $\SL_2(p)$ (see~\cite{mo-unique-perfect}), so $G = \SL_2(p)$ is the unique minimizer of $\cm(G)$ among groups of this order if $p$ is sufficiently large. Among groups with $|G^\ab|$ as small as possible, the number of involutions in $G$ must be within $O(1)$ of the minimum.

\bibliography{refs}
\bibliographystyle{alpha}

\end{document}

%% file: psystems.tex
\begin{tabular}{lllll}
$\psystem$ & $\calP_1$ & $\calP_2$ & $\calP_3$ & $\cx\psystem$ \\ \hline
$\psystem_{0}$ & $\{\{1, 2\}\}$ & $\{\{1, 2\}\}$ & $\{\{1, 2\}\}$ & $0$ \\
$\psystem_{1}$ & $\{\{1, 2, 3\}\}$ & $\{\{1, 2, 3\}\}$ & $\{\{1, 2, 3\}\}$ & $1$ \\
$\psystem_{2}$ & $\{\{1, 2\}, \{3, 4\}\}$ & $\{\{1, 3\}, \{2, 4\}\}$ & $\{\{1, 4\}, \{2, 3\}\}$ & $1$ \\
$\psystem_{3}$ & $\{\{1, 2\}, \{3, 4\}\}$ & $\{\{1, 2\}, \{3, 4\}\}$ & $\{\{1, 3\}, \{2, 4\}\}$ & $1$ \\
$\psystem_{4}$ & $\{\{1, 2, 3, 4\}\}$ & $\{\{1, 2\}, \{3, 4\}\}$ & $\{\{1, 2\}, \{3, 4\}\}$ & $1$ \\
$\psystem_{5}$ & $\{\{1, 2, 3, 4\}\}$ & $\{\{1, 2\}, \{3, 4\}\}$ & $\{\{1, 3\}, \{2, 4\}\}$ & $2$ \\
$\psystem_{6}$ & $\{\{1, 2, 3, 4\}\}$ & $\{\{1, 2, 3, 4\}\}$ & $\{\{1, 2\}, \{3, 4\}\}$ & $2$ \\
$\psystem_{7}$ & $\{\{1, 2, 3, 4\}\}$ & $\{\{1, 2, 3, 4\}\}$ & $\{\{1, 2, 3, 4\}\}$ & $2$ \\
\end{tabular}

%% file: psystems-cx1.tex
\begin{tabular}{lllllll}
$\psystem$ & $\cx\psystem$ & $|\supp\psystem|$ & $\mu(\psystem)$ & $|\Aut\psystem|$ & $\#\text{reorderings}$ & $\gamma_G(\psystem)$ \\ \hline
$\psystem_{0}$ & $0$ & $2$ & $-1$ & $2$ & $1$ & $1 - 1/n$ \\
$\psystem_{1}$ & $1$ & $3$ & $8$ & $6$ & $1$ & $1 + O(n^{-1})$ \\
$\psystem_{2}$ & $1$ & $4$ & $1$ & $4$ & $1$ & $\inv + O(n^{-1})$ \\
$\psystem_{3}$ & $1$ & $4$ & $1$ & $4$ & $3$ & $1 + O(n^{-1})$ \\
$\psystem_{4}$ & $1$ & $4$ & $-6$ & $8$ & $3$ & $1 + O(n^{-1})$ \\
\end{tabular}